\documentclass[reqno,11pt]{amsart}
 
\usepackage{amsmath}
\usepackage{amsthm,amssymb,color}
\usepackage{bbm}
\usepackage{mathrsfs}
\usepackage{hyperref}
\usepackage[TS1,T1]{fontenc}
\usepackage[utf8]{inputenc}
\usepackage{dsfont}
\usepackage{tikz}
\usepackage{pgfplots}
\usepackage{enumitem}
\usepackage{subfigure}
\usepackage{mathtools}
\usepackage{bbold}
\usepackage{esint}
\usepackage{mathabx}
\usepackage[sort]{cite}% to arrange citations in increasing order
\usepackage{comment}
\usepackage{float}

\numberwithin{equation}{section}

\newtheorem{thm}{Theorem}[section]

\newtheorem{lem}[thm]{Lemma}

\theoremstyle{definition}

\newtheorem{rem}[thm]{Remark}
\newtheorem{cor}[thm]{Corollary}

\newtheorem{Ex}[thm]{Example}

\DeclareMathOperator{\Id}{Id}

\DeclareMathOperator{\DIV}{div}

\DeclareMathOperator*{\argmax}{argmax}

\newcommand{\KL}{\mathrm{KL}}

\newcommand{\Lip}{\text{Lip}}

\newcommand{\R}{\mathbb{R}}

\newcommand{\Rd}{\mathbb{R}^{d}}
\newcommand{\Pro}{\mathscr{P}(\Rd)}

\newcommand{\N}{\mathbb{N}}

\newcommand{\eps}{\varepsilon}

\newcommand{\rkhs}{\mathcal{H}}

\newcommand{\supp}{\text{supp}}
\newcommand{\diff}{\mathop{}\!\mathrm{d}}

\newcommand{\ra}{\rightarrow}

\newcommand{\rkc}{\mathcal{D}}
\newcommand{\wstar}{\overset{\ast}{\rightharpoonup}}

\newcommand{\vp}{\varphi}

\newcommand{\inner}[2]{\left\langle #1,#2\right\rangle}

\newcommand{\rF}{\mathcal{F}}

\newcommand{\pot}{V_0}
\newcommand{\DD}{\mathbb{D}}

\newcommand{\id}{\mathrm{id}}
\newcommand{\trace}{\mathrm{trace}}

\newcommand{\weaks}{\overset{\ast}{\rightharpoonup}}

\newcommand{\weak}{\rightharpoonup}
\newcommand{\abs}[1]{\left\vert #1\right\vert}

\newcommand{\prt}{\partial}
\newcommand{\divv}{\mathrm{div}}

\newcommand{\loc}{\mathrm{loc}}

\DeclarePairedDelimiter{\norm}{\lVert}{\rVert}

\newcommand{\pwc}{\tau}
\newcommand{\spk}{\mathcal{K}}
\newcommand{\ngib}{\mathcal{C}}
\makeatletter
\newcommand{\doublewidetilde}[1]{{%
  \mathpalette\double@widetilde{#1}%
}}
\newcommand{\double@widetilde}[2]{%
  \sbox\z@{$\m@th#1\widetilde{#2}$}%
  \ht\z@=.9\ht\z@
  \widetilde{\box\z@}%
}
\makeatother

\makeatletter
\renewcommand\subsubsection{\@startsection{subsubsection}{3}{\z@}%
  {-3.25ex\@plus -1ex \@minus -.2ex}%
  {1.5ex \@plus .2ex}%
  {\normalfont\bfseries}}
\makeatother

\author{José A. Carrillo}
\address{{\it José A. Carrillo:} Mathematical Institute, University of Oxford, Woodstock Road, Oxford, OX2 6GG, United Kingdom}
\email{carrillo@maths.ox.ac.uk}

\author{Jakub Skrzeczkowski}
\address{{\it Jakub Skrzeczkowski: } Mathematical Institute, University of Oxford, Woodstock Road, Oxford, OX2 6GG, United Kingdom}
\email{jakub.skrzeczkowski@maths.ox.ac.uk}

\author{Jethro Warnett}
\address{{\it Jethro Warnett: } Mathematical Institute, University of Oxford, Woodstock Road, Oxford, OX2 6GG, United Kingdom}
\email{warnett@maths.ox.ac.uk}

\begin{document}

\title[The Stein-log-Sobolev inequality]{The Stein-log-Sobolev inequality and the exponential rate of convergence for the continuous Stein variational gradient descent method}

\begin{abstract}
The Stein Variational Gradient Descent method is a variational inference method in statistics that has recently received a lot of attention. The method provides a deterministic approximation of the target distribution, by introducing a nonlocal interaction with a kernel. Despite the significant interest, the exponential rate of convergence for the continuous method has remained an open problem, due to the difficulty of establishing the related so-called Stein-log-Sobolev inequality. Here, we prove that the inequality is satisfied for each space dimension and every kernel whose Fourier transform has a quadratic decay at infinity and is locally bounded away from zero and infinity. Moreover, we construct weak solutions to the related PDE satisfying exponential rate of decay towards the equilibrium. The main novelty in our approach is to interpret the Stein-Fisher information, also called the squared Stein discrepancy, as a duality pairing between $H^{-1}(\Rd)$ and $H^{1}(\Rd)$, which allows us to employ the Fourier transform. We also provide several examples of kernels for which the Stein-log-Sobolev inequality fails, partially showing the necessity of our assumptions.
\end{abstract}

\keywords{Stein variational gradient descent; Particle system; Mean-field limit; Sampling; Bayesian inference; Stability analysis; log-Sobolev inequality; Stein-log-Sobolev inequality}

\subjclass{35Q62, 35Q68, 35B40, 62-08, 62D05}

%%%%%%%%%% codes:
% 35Q62 = PDEs in connection with statistics
% 35B40 = Asymptotic behavior of solutions to PDEs
% 35Q68 = PDEs in connection with computer science
% 62-08  	Computational methods for problems pertaining to statistics
% 62D05  	Sampling theory, sample surveys

\maketitle
\setcounter{tocdepth}{1}

\section{Introduction}
Bayesian inference and statistical physics face an ongoing challenge: how to sample from~a high-dimensional probability distribution with an unknown normalization constant. The target probability distribution $\rho_\infty\in\Pro$ often has the form of a Gibbs measure
\begin{align*}
    \rho_\infty(x)=\frac{1}{Z}e^{-V(x)}
\end{align*}
with a potential $V:\Rd\ra\R$ and unknown normalization constant $Z>0$. Numerous sampling algorithms have been proposed to address this problem. The Markov Chain Monte Carlo (MCMC) methods are a popular class of algorithms that seek to simulate a Markov chain, which are efficient to simulate and have the stationary distribution $\rho_\infty$. Common examples of algorithms in this class include the Metropolis-Hastings algorithm \cite{metropolis1953equation, MR3363437}, Gibbs sampling \cite{geman1984stochastic}, the Metropolis-adjusted Langevin algorithm \cite{roberts1998optimal} and Stochastic gradient Langevin dynamics \cite{welling2011bayesian}. 
%https://ebookcentral.proquest.com/lib/oxford/detail.action?docID=762505&pq-origsite=primo
The Variational Inference (VI) methods are another prevalent family of algorithms that attempt to approximate the distribution $\rho_\infty$ by a family of simpler distributions (that are usually parameterized) by minimizing their Kullback-Leibler divergence. Widely used methods include Stochastic Variational Inference \cite{MR3081926} and Black Box Variation Inference \cite{ranganath2014black}.

\smallskip

In this paper, we focus on a recent sampling algorithm called Stein Variational Gradient Descent \cite{liu2016stein}. It is a time-discrete deterministic VI method using a particle ensemble approach. The method iteratively updates a set of particles along a transformation within the unit ball of a $d$-dimensional reproducing kernel Hilbert space $(\rkhs_K,\norm{\cdot}_{K})$, with positive definite kernel $K:\Rd\times\Rd\ra \R$ (see \cite{MR3560890} for a comprehensive introduction on reproducing kernel Hilbert spaces), that pushes particles to high probability regions whilst balancing repulsion between particles. More precisely, given a set of $N\geq 1$ particles at step $n\geq 0$ denoted by $X_n^1,\ldots,X_n^N\in\Rd$ and represented by the empirical measure
\begin{align*}
    \rho_n^N:=\frac{1}{N}\sum_{j=1}^N \delta_{X_n^j},
\end{align*}
the positions of particles are updated using the map
\begin{align*}
    \phi^\ast(\bullet;\, \rho) :=\argmax_{\substack{\phi\in \rkhs_K \\ \norm{\phi}_K\leq 1}}\left.\frac{d}{d\varepsilon}\right\vert_{\varepsilon=0}\KL((\Id+\varepsilon\phi)_\#\rho||\rho_\infty),
\end{align*}
where we define the Kullback-Leibler divergence by
\begin{equation}\label{eq:KL_div}
\KL(\rho || \rho_\infty) = \begin{cases}
\displaystyle\int_{\Rd} \ln\left(\frac{\diff \rho}{\diff \rho_\infty}(x) \right) \, \frac{\diff \rho}{\diff \rho_\infty}(x)  \diff \rho_\infty(x)  &\mbox{ if } \rho\ll\rho_\infty \mbox{ and}\\
&\quad\frac{\diff \rho}{\diff \rho_\infty}\ln\left(\frac{\diff \rho}{\diff \rho_\infty}\right)\in L^1(\Rd,\rho_\infty),\\[.6em]
+\infty &\mbox{ otherwise}.
\end{cases}
\end{equation}
The resulting interacting particle system has the form
\begin{align*}
    X_{n+1}^i = X_{n}^i - \varepsilon_n\,\phi^\ast(X_n^i;\, \rho_n^N)
\end{align*}
for adaptive step sizes $\{\varepsilon_n\}_{n\geq 1}\subseteq (0,\infty)$. It was shown in \cite[Equation 3]{liu2016stein} and \cite[Theorem~3.8]{liu2016kernelized} that $\phi^\ast$ is well-defined under mild assumptions of $V$ and $\rho$, and explicitly given by an expectation of the so-called Stein operator applied to the kernel $K$
\begin{align*}
    \phi^\ast(x;\, \rho)
    :=\int_{\Rd}
    K(x,y) \nabla \ln(\rho_\infty(y))\;\rho(y)\,\diff y
    -\int_{\Rd}
    \nabla_x K(x,y)\;\rho(y) \,\diff y,
\end{align*}
where $\nabla_x K$ is the gradient of $K$ in the first variable. This algorithm enjoys several properties such as determinism, ease of implementation and moment matching \cite{liu2018stein}.

\smallskip

The SVGD algorithm has seen several enhancements. Adding stochastic noise \cite{gallego2020stochasticgradientmcmcrepulsive,priser2024long} makes it more robust, and it can be made gradient-free \cite{han2018stein}. Also, the the scalar-valued kernel can be replaced with a matrix-valued one \cite{wang2019stein}. Projecting it \cite{chen2020projected} helps tackle the curse of dimensionality. Message passing \cite{zhuo2018message} prevents particles from collapsing to the target mode. A multilevel approach \cite{MR4768482, alsup2021multilevelsteinvariationalgradient} improves computational speed, while incorporating second-order information through a Newton formulation \cite{NEURIPS2018_fdaa09fc} and a projected Newton method \cite{NEURIPS2019_eea5d933} further accelerates the process. An interpolation between SVGD and the Wasserstein gradient flow \cite{he2024regularizedsteinvariationalgradient} achieves exponential convergence with the chi-squared distance using slightly augmented particle dynamics \cite{NEURIPS2020_16f8e136}. Additionally, the Wasserstein gradient flow with respect to the Stein-Fisher information or squared Stein Discrepancy (see \eqref{eq:stein_fisher_information}) has been explored \cite{pmlr-v139-korba21a}.

\smallskip

Numerical experiments in \cite{MR4512980, MR4113783, MR4768482} suggest that SVGD has an exponential rate of decay in the Kullback-Leibler divergence. Several other papers make the assumption that the method does indeed have an exponential rate of convergence \cite{MR4768482, alsup2021multilevelsteinvariationalgradient}. A sufficient condition for an exponential rate of convergence is the Stein-log-Sobolev inequality \cite{MR4582478}. Several properties of the mean-field algorithm have been studied under various assumptions, including the Stein-log-Sobolev inequality, in \cite{korba2020nonasymptotic}. 

\smallskip

This paper provides for the first time kernels for which the Stein-log-Sobolev inequality holds closing this gap in the literature. As a consequence we rigorously achieve the sought after exponential rate of convergence for the continuous formulation of the algorithm.

\smallskip

Several results about the finite-particle system of SVGD, such as convergence rates, are discussed in \cite{balasubramanian2024improved,shi2024finite,carrillo2023convergence,liu2024towards}. In \cite{liu2017stein} it was shown that the mean-field ensemble weakly converges to the target distribution $\rho_\infty$ and exhibits a time-discrete gradient flow structure.

It is well-known that there exists a Riemannian structure on the Wasserstein space to describe continuous gradient flows (see \cite{MR1617171, MR1842429, MR2192294, MR2401600}, \cite[Chapter 9]{MR1964483}, \cite[Chapter 15]{MR2459454}) and the log-Sobolev inequality \cite{MR420249, To99} is a sufficient condition for exponential convergence \cite{MR1760620}.

\smallskip

It is hence reasonable to expect that a similar approach will work for SVGD. The first step in this direction was to formulate a continuous version of the SVGD algorithm \cite{MR3919409, liu2017stein}. This is described by the following nonlinear and nonlocal continuity equation, which is referred to as the mean field limit of the Stein Variational Gradient Descent,
\begin{align}
    \label{eq:stein_gradient_descent_kullback_leibler}
    \prt_t \rho(x)=\divv\left(\rho(x) \int_{\Rd} K(x,y)\;\nabla_y\left( \frac{\diff \rho}{\diff \rho_\infty}(y)\right)\diff \rho_\infty(y)\right),
    \tag{MF SVGD}
\end{align}
where the dissipation (or the Stein-Fisher information or the squared Stein Discrepancy) $\DD^2(\rho_t||\rho_\infty)$ is defined as
\begin{align}
    \label{eq:stein_fisher_information}
    \DD^2(\rho_n||\rho_\infty)=\int_{\Rd}\int_{\Rd} 
    \nabla_x\left(\frac{\diff \rho_n}{\diff \rho_\infty}(x)\right)\cdot K(x,y)\nabla_y\left(\frac{\diff \rho_n}{\diff \rho_\infty}(y)\right)\diff \rho_\infty(y)\diff\rho_\infty(x).
\end{align}
It should be noted that \cite{pmlr-v139-korba21a} uses a different formulation of the dissipation, which is equivalent to \eqref{eq:stein_fisher_information} under differentiability assumptions on $K$. It was shown in \cite[Theorem 3.4]{liu2017stein}, that the evolution of densities $\{\rho_t\}_{t\geq 0}$ is equipped with the following energy dissipation identity
\begin{align*}
    \partial_t \KL(\rho_t || \rho_{\infty}) + \DD^2(\rho_t||\rho_\infty)
    =0.
\end{align*}

A formal geometry, called the Stein geometry, was established and discussed in \cite{MR4582478, MR4622922} to interpret the equation \eqref{eq:stein_gradient_descent_kullback_leibler} as a gradient flow. Additionally in \cite{MR4582478}, the authors study the curvature of the Kullback-Leibler divergence around the equilibrium. In particular, they prove in one space dimension for the "weighted Matérn kernel" $K(x,y)=e^{\frac{V(x)}{2}} e^{-|x-y|}e^{\frac{V(y)}{2}}$ and for quadratic potentials $V(x)$ that the Hessian of the Kullback-Leibler divergence at $\rho_\infty$ is bounded from below by a multiple of the identity in the Stein metric. This kernel was the inspiration of our method. A recent publication \cite{zhu2024kernel} has analyzed the effects of different geometries on the gradient flow of a functional related to the SVGD.

\smallskip

Formally by \cite{MR4582478}, the sufficient condition of proving the exponential rate of convergence is showing the so-called Stein-log-Sobolev inequality: there exists a constant $\lambda>0$ depending on $\rho_\infty$ such that
\begin{equation}
    \label{eq:stein_log_sobolev_inequality_general}
    \lambda\,\KL(\rho_t || \rho_{\infty}) \leq \DD^2(\rho_t||\rho_\infty). 
\end{equation}
We could alternatively prove algebraic rate of convergence, if there exists $\lambda>0$ and $\alpha> 1$ depending on $\rho_\infty$ such that
\begin{equation}
    \label{eq:stein_log_sobolev_inequality_general_alpha}
    \lambda\,\left[\KL(\rho_t || \rho_{\infty})\right]^\alpha \leq \DD^2(\rho_t||\rho_\infty). 
\end{equation}

\subsection{Main results}

Our main result states that there exist kernels such that the inequality \eqref{eq:stein_log_sobolev_inequality_general} holds for perturbed log-convex targets $\rho_\infty\in\Pro$ with 
\begin{equation}\label{eq:general_prop_target_potentialV_f}
\rho_\infty=\frac{1}{Z} e^{-V}, \qquad V(x)\geq \ngib+\frac{1}{2}(x-\mu)\cdot\Sigma^{-1}(x-\mu),
\end{equation} 
where $V\in L_\loc^\infty(\Rd)$ such that $V\geq 0$ a.e., $\ngib\in\R$ is a constant, $\mu\in\Rd$ is a mean and $\Sigma\in \R^{d\times d}$ is a strictly positive definite covariance matrix. We define $0<\sigma_1\leq\cdots\leq \sigma_d$ to be the eigenvalues of $\Sigma$, while $Z := \int_{\Rd} e^{-V} \diff x$ is the normalization constant. Note carefully that \eqref{eq:general_prop_target_potentialV_f} includes anisotropic Gaussian distributions and polynomial functions $V(x) = |x|^p$ with exponent $p \geq 2$. To simplify notation, we introduce the potential
$$
\pot(x):=\frac{1}{2}(x-\mu)\cdot \Sigma^{-1}(x-\mu).
$$
Inspired by \cite{MR4582478}, we assume that the kernel $K(x,y)$ is of the ansatz form
    \begin{align}
        \label{eq:ansatz_K}
        K(x,y):=e^{V(x)-\frac{\pot(x)}{2}} \, k(x-y) \, e^{V(y)-\frac{\pot(y)}{2}},
    \end{align}
    for a function $k$ to be determined. Then, the desired inequality \eqref{eq:stein_log_sobolev_inequality_general} reads
\begin{align}\label{eq:log_Sobolev_convolution_form}
    \begin{split}
        \lambda\,\KL(\rho || \rho_{\infty}) 
        &\leq  \int_{\Rd}\int_{\Rd}  \nabla\left(\rho(x)\,e^{V(x)}\right)\, e^{-\frac{\pot(x)}{2}}
        \cdot k(x-y)
        \nabla\left(\rho(y)\,e^{V(y)}\right)\,
        e^{-\frac{\pot(y)}{2}}\diff y\diff x \\
        & = \int_{\Rd} \nabla \left(\rho \,e^{V}\right)\, e^{-\frac{\pot}{2}} \,  k \ast \left(\nabla \left(\rho \,e^{V}\right)\, e^{-\frac{\pot}{2}}\right) \diff x,
   \end{split}
\end{align}
where $\ast$ is the convolution operator, see \eqref{eq:convolution_def}. Throughout this paper we will exclusively work with functions $\rho$ such that the RHS of \eqref{eq:log_Sobolev_convolution_form} is finite. Our strategy is to prove a sharper inequality that holds for all functions $\rho\in L^1(\Rd)$
\begin{align}\label{eq:l2_lower_bound_convolution_form}
	\lambda \int_{\Rd}\vert \rho e^{V-\frac{\pot}{2}}-\tau e^{-\frac{\pot}{2}}\vert^2\diff x
	\leq \int_{\Rd} \nabla \left(\rho \,e^{V}\right)\, e^{-\frac{\pot}{2}} \,  k \ast \left(\nabla \left(\rho \,e^{V}\right)\, e^{-\frac{\pot}{2}}\right) \diff x,
\end{align}
where $\tau$ is defined by the quotient
\begin{align}
    \label{eq:def_tau}
    \tau:={\displaystyle \int_{B_\varepsilon
    } \rF\big(\rho \, e^{V-\frac{\pot}{2}}\big)\diff \xi}\Bigg/{\displaystyle \int_{B_\varepsilon
    } \rF\big(e^{-\frac{\pot}{2}}\big)\diff \xi}
\end{align}
for some fixed $\varepsilon>0$ depending on $k$ and $d$, where $\rF(g)$ and $\hat{g}$ denote the Fourier transform of a function $g$. Observe that if $\rho=\rho_\infty=\frac{1}{Z}e^{-V}$, then $\tau=\frac{1}{Z}$, and thus both sides of the inequality \eqref{eq:l2_lower_bound_convolution_form} vanish. We have chosen the weight $e^{V-\frac{\pot}{2}}$ in the ansatz form in \eqref{eq:ansatz_K} so that it allows for a control of the Kullback-Leibler divergence in \eqref{eq:l2_lower_bound_convolution_form}. Notice that inequality \eqref{eq:log_Sobolev_convolution_form} is an immediate consequence of \eqref{eq:l2_lower_bound_convolution_form} and Lemma~\ref{lem:chi_squared_bounded_from_below_by_kullback_leibler}. Additionally, we note that the inequality \eqref{eq:l2_lower_bound_convolution_form} provides a bound on the $L^2 $-norm of $\rho \, e^{V-\frac{\pot}{2}}$ (see \eqref{eq:bound_L2_H-1_norm_by_dissipation}).

\smallskip

The main idea in our proof is to pass to Fourier variables, so that the inequality \eqref{eq:l2_lower_bound_convolution_form} simplifies. One reason for the choice of the weight in \eqref{eq:ansatz_K} is so that we can eliminate the term $e^{-V}$ when passing to Fourier variables in the dissipation. We construct $k$ such that $k\in L^1(\Rd)+L^2(\Rd)$ (i.e. $k = k_1+k_2$ with $k_1 \in L^1(\Rd)$ and $k_2 \in L^2(\Rd)$) and its Fourier transform satisfies
\begin{equation}\label{eq:condition_Fourier_transform_k}
	\frac{1}{\rkc_k} \frac{1}{1+|\xi|^2} \leq \hat{k}(\xi) \leq \rkc_k\,\frac{1}{1+|\xi|^2},
    \qquad \hat{k}\in W^{1,\infty}(\Rd),
\end{equation}
for some constant $\rkc_k\geq 1$ depending on $k$. We note Assumption \eqref{eq:condition_Fourier_transform_k} is equivalent to the Fourier transform of $k$ having quadratic decay (i.e. $\lim_{|\xi|\to \infty} \hat{k}(\xi) \, |\xi|^2 = c \in (0,\infty)$) and being locally bounded away from zero and infinity. Condition \eqref{eq:condition_Fourier_transform_k}  has two consequences for the dissipation:
\begin{enumerate}
	\item The dissipation is equivalent to the $H^{-1}$ norm on $\nabla \left(\rho \,e^{V}\right)\, e^{-\frac{\pot}{2}}$ (see \eqref{eq:definition_H-1} for the definition of $H^{-1}(\Rd)$). Thus, we may assume that $\nabla \left(\rho \,e^{V}\right)\, e^{-\frac{\pot}{2}} \in H^{-1}(\Rd)$.
	\item We get an operator $H^{-1}(\Rd)\ra H^1(\Rd); g\mapsto k\ast g$ (see Lemma \ref{lem:Plancherel_f_H-1}), so the dissipation can be interpreted as an action of the functional $\nabla (\rho\, e^{V}) \, e^{-\frac{\pot}{2}}$ on the function $k\ast~ \big(\nabla (\rho\, e^{V}) \, e^{-\frac{\pot}{2}}\big)$, denoted by the bracket $\langle \cdot, \cdot \rangle_{H^{-1}(\Rd), H^1(\Rd)}$.
\end{enumerate}
Therefore, we will assume that the following regularity conditions hold
\begin{equation}\label{eq:regularity_to_be_satisfied_ineqality_rho}
\rho \in L^1(\Rd), \qquad \rho \, e^{V-\frac{\pot}{2}} \in L^2(\Rd), \qquad  \rho \, e^{V-\frac{\pot}{2}} \, \nabla \pot \in H^{-1}(\Rd).
\tag{H}
\end{equation}
By observing that the following identity holds in the sense of distributions
\begin{align}
    \label{eq:split_stein_gradient}
    \nabla (\rho\, e^{V}) \, e^{-\frac{\pot}{2}} = \nabla (\rho \, e^{V-\frac{\pot}{2} }) + \frac{1}{2}\rho\, e^{V-\frac{\pot}{2}}\, \nabla \pot
\end{align}
we get $
\nabla (\rho\, e^{V}) \, e^{-\frac{\pot}{2}}  \in H^{-1}(\Rd)$. Thus, we may rewrite \eqref{eq:l2_lower_bound_convolution_form} equivalently as
\begin{equation}\label{eq:inequality_l2_lower_bound_formulated_H-1_H1}
	\begin{split}
		\lambda\int_{\Rd}\vert \rho& e^{V-\frac{\pot}{2}}-\tau e^{-\frac{\pot}{2}}\vert^2\diff x \\
		&\leq  \left \langle \nabla (\rho\, e^{V}) \, e^{-\frac{\pot}{2}}, k\ast \left(\nabla (\rho\, e^{V}) \, e^{-\frac{\pot}{2}}\right)  \right \rangle_{H^{-1}(\Rd), H^1(\Rd)} =: \DD_k^2(\rho||\rho_\infty),
	\end{split}
    \tag{B}
\end{equation}
where $\tau\in\R$ is given by \eqref{eq:def_tau}. We omit the subscript $\DD_k^2(\rho||\rho_\infty)=\DD^2(\rho||\rho_\infty)$ in case it is clear with which kernel $k$ we are working with. In case $\rho$ is also in $\Pro$, then we can show the following inequality holds
\begin{align}\label{eq:stein_log_sobolev_inequality}
    \lambda\, \KL(\rho||\rho_\infty) 
    \leq  \DD_k^2(\rho||\rho_\infty).
    \tag{SLSI}
\end{align}

To our knowledge, the validity of Stein-log-Sobolev inequality \eqref{eq:stein_log_sobolev_inequality} was not proven before for any kernel.

\begin{thm}[existence of kernels satisfying \eqref{eq:inequality_l2_lower_bound_formulated_H-1_H1} and \eqref{eq:stein_log_sobolev_inequality}]
    \label{thm:existence_of_maternised_convolution_kernel_that_satisfy_slsi}
    Let $\rho_\infty$, $V$ be as in \eqref{eq:general_prop_target_potentialV_f} with $\ngib\in\R$, $\mu\in\Rd$ and $\Sigma\in \R^{d\times d}$ a strictly positive definite and symmetric matrix. Then for any kernel $k\in L^1(\Rd)+L^2(\Rd)$ that satisfies the condition \eqref{eq:condition_Fourier_transform_k}, there exists a constant $\lambda>0$ dependent on $V$, $\Sigma$, $\ngib$ and $\rkc_k$ in \eqref{eq:condition_Fourier_transform_k}, such that the inequality \eqref{eq:inequality_l2_lower_bound_formulated_H-1_H1} is satisfied for all $\rho$ as in \eqref{eq:regularity_to_be_satisfied_ineqality_rho}. If additionally $\rho \in \Pro$, then the inequality \eqref{eq:stein_log_sobolev_inequality} also holds true. Moreover, there exists a constant $C>0$ depending on $k$ and $\rho_\infty$ such that for all $\rho$ satisfying \eqref{eq:regularity_to_be_satisfied_ineqality_rho} we have the inequality
    \begin{align}
        \label{eq:bound_L2_H-1_norm_by_dissipation}
        \norm{\rho e^{V-\frac{\pot}{2}}}_{L^2}^2+\norm{\rho e^{V-\frac{\pot}{2}}\nabla \pot}_{H^{-1}}^2\leq C\,\big(\norm{\rho}_{L^1}+\DD_k^2(\rho||\rho_\infty)\big).
    \end{align}
\end{thm}

\begin{rem}
    \label{rem:facts_main_thm}
    We emphasize some important facts about this result.
    \begin{itemize}
        \item The Stein-log-Sobolev constant $\lambda$ can be made explicit. In fact, by tracking all the dependencies in Section \ref{section:existence_kernels_satisfying_LSS_inequality}, the constant is given by 
        \begin{align*}
            \lambda=\frac{\lambda_{0,1}}{\rkc_k^{\,3} \,\rkc_{k_{0,d}}}\cdot\frac{e^{\ngib}}{Z}\cdot\frac{\sigma_1\wedge 1}{\sigma_d},
        \end{align*}
        where $\lambda_{0,d}>0$ is a constant defined in \eqref{eq:lambda_0_one_dimension} for $d=1$ and in \eqref{eq:formula_for_lambda_0} for $d\geq 2$, $k_{0,d}$ is the kernel defined in \eqref{eq:def_k01} for $d=1$ or is the kernel constructed in Section \ref{sect:exist_kernel_any_dimensio} for $d\geq 2$, the values $\sigma_1$, $\sigma_d$, $\ngib$, $Z$ are from \eqref{eq:general_prop_target_potentialV_f}, and $\rkc_k$ and $\rkc_{k_{0,d}}$ are the constants from \eqref{eq:condition_Fourier_transform_k} for $k$ and $k_{0,d}$ respectively.
        \item Let $\{V_n\}_{n\in\N}$ be a sequence of potentials satisfying assumptions of Theorem \ref{thm:existence_of_maternised_convolution_kernel_that_satisfy_slsi} and let $C_{V_n}$ be their constants in \eqref{eq:bound_L2_H-1_norm_by_dissipation}. Assume there exists a bounded open set $U\subseteq\Rd$ such that $\|V_n-V\|_{L^{\infty}(U)} \to 0$, then $C_{V_n} \to C_V$. This follows from \eqref{eq:bound_tau_on_bnd_set}. 
        \item We could generalize Theorem \ref{thm:existence_of_maternised_convolution_kernel_that_satisfy_slsi} for every kernel $k\in L^1(\Rd)+L^2(\Rd)$ that satisfies
        \begin{equation}\label{eq:weaker_condition_kernel}
            \frac{1}{\rkc_k}\frac{1}{1+|\xi|^2}\leq \hat{k}(\xi),\qquad \hat{k}\in W^{1,\infty}(\Rd),
        \end{equation}
        for some constant $\rkc_k\geq 1$ depending on $k$ at the price that $\DD_k^2(\rho||\rho_\infty)$ will not be a duality pairing between $H^{-1}(\Rd)$ and $H^1(\Rd)$. Since kernels satisfying \eqref{eq:weaker_condition_kernel} are less regular than \eqref{eq:condition_Fourier_transform_k}, we do not see any gain from using kernels satisfying \eqref{eq:weaker_condition_kernel}.
    \end{itemize}
\end{rem}

The proof of Theorem \ref{thm:existence_of_maternised_convolution_kernel_that_satisfy_slsi} is based on writing the RHS of \eqref{eq:inequality_l2_lower_bound_formulated_H-1_H1} in Fourier variables and a careful analysis of the resulting expression. The strategy is detailed at the beginning of Section \ref{section:existence_kernels_satisfying_LSS_inequality}.

\smallskip

It is not a priori clear that there exist solutions to \eqref{eq:stein_gradient_descent_kullback_leibler} for kernels constructed in Theorem \ref{thm:existence_of_maternised_convolution_kernel_that_satisfy_slsi}. In fact, even for the explicit kernel $k(x) = e^{-|x|}$ in one dimension, its second derivative is singular, preventing the application of classical arguments. The result below provides the existence of a weak solution with exponential decay to $\rho_{\infty}$ as $t \to \infty$.

\begin{thm}[existence of weak solutions to \eqref{eq:stein_gradient_descent_kullback_leibler}]
	\label{thm:existence_of_distributional_solutions_to_mfsvgd}
    Suppose that \eqref{eq:general_prop_target_potentialV_f} holds and $\nabla V$ is continuous, $V \in H^m_{loc}(\Rd)$ for some $m > \frac{d}{2}$. Let $\rho_0\in\Pro$ with $\KL(\rho_0||\rho_\infty)<\infty$ and $k\in L^1(\Rd)+L^2(\Rd)$ satisfying \eqref{eq:condition_Fourier_transform_k}. 
    Then, there exists a curve of probabiliy measures $[0,\infty) \ni t \ra \rho_t\in \mathscr{P}(\Rd)$
    that is a distributional solution to \eqref{eq:stein_gradient_descent_kullback_leibler} with initial value $\rho_0$, i.e. for all $T>0$ and for all test functions $\psi\in C_c^\infty([0,T]\times \Rd)$ we have
    \begin{align}
        \label{eq:weak_sol_mfsvgd}
    	\begin{split}
    	    \int_{\Rd}\psi(T,x)\, \rho_T(x) \diff x =& \int_{\Rd}\psi(0,x)\, \rho_0(x) \diff x + 
            \int_0^T \int_{\Rd}\prt_t\psi\; \rho \diff x\diff t\\
            &-\int_0^T\int_{\Rd} k\ast \left(\nabla(\rho e^{V})\, e^{-\frac{\pot}{2}}\right)\cdot \rho\,  e^{V-\frac{\pot}{2}} \nabla \psi\diff x\diff t.
    	\end{split}
    \end{align}
    The curve is absolutely continuous with respect to the Wasserstein metric $W_1$ and it enjoys the following regularity
    \begin{equation}\label{eq:regularity_weak_sol}
    \rho\, e^{V-\frac{\pot}{2}} \in L^2((0,T)\times\Rd), \quad   \nabla(\rho \, e^{V})e^{-\frac{\pot}{2}} \in L^2(0,T; H^{-1}(\Rd)) \mbox{ for all } T>0.
    \end{equation}
    Additionally, $\rho$ satisfies an energy dissipation type inequality and exponential decay in the Kullback-Leibler distance 
    \begin{equation}\label{eq:KL:dissipation_inequality_solution}
            \KL(\rho_t||\rho_\infty) + \int_0^t \DD^2(\rho_s||\rho_\infty)\diff s \leq \KL(\rho_0||\rho_\infty), \qquad \KL(\rho_t||\rho_\infty)
            \leq e^{-\lambda t}\, \KL(\rho_0||\rho_\infty).
        \end{equation}
\end{thm}

The formula \eqref{eq:weak_sol_mfsvgd} makes sense because the LHS is well-defined by the narrow continuity in time of $\rho$. The last term on the RHS is well-defined by the Hölder's inequality, regularity \eqref{eq:regularity_weak_sol} and Lemma \ref{lem:Plancherel_f_H-1}, this yields
\begin{align*}
    k\ast\left(\nabla(\rho \, e^{V})e^{-\frac{\pot}{2}}\right) \in L^2(0,T; H^{1}(\Rd)) \subset L^2((0,T)\times\Rd).
\end{align*}

The last result highlights several cases when the inequality \eqref{eq:stein_log_sobolev_inequality_general_alpha} cannot hold true.

\begin{thm}[Conditions for failure of \eqref{eq:stein_log_sobolev_inequality_general_alpha}]
\label{thm:conditions_for_slsi} 
    Assume that $\alpha\geq 1$. The inequality \eqref{eq:stein_log_sobolev_inequality_general_alpha} will fail in the following cases.
    \begin{enumerate}[label=(F\arabic*)]
        \item \label{thm:failure_of_slsi_for_regular_convolution_kernels}
    Let $V=\pot$ (so $\rho_\infty$ is Gaussian) and $k\in L^1(\Rd)+L^2(\Rd)$ be such that there exists $r> 1+\frac{d}{2}$ with
    \begin{align*}
        \hat{k}(\xi)\leq C\frac{1}{(1+|\xi|^2)^r}.
    \end{align*}
    The kernel is defined by $K(x,y):=e^{\frac{\pot(x)}{2}}k(x-y)e^{\frac{\pot(y)}{2}}$.
        \item\label{thm:failure of slsi for integrable convolution kernels}  
        Let $r\in [1,\infty]$, $\beta \in \left[0,\left(2-\frac{1}{2r}\right)d+1\right)$, $f \in L^{\infty}(\R^d)$ and $k \in L^r(\R^d)$. The kernel is defined by $K(x,y):=(1+|x|^\beta)f(x) k(x-y)(1+|y|^\beta)f(y)$.
    \item \label{thm:failure_of_slsi_for_homogeneous_convolution_kernels}
    Let $\gamma>0$, $\beta <0$, $k:\Rd\ra\R$ is a $\beta$-homogeneous positive definite kernel (in the sense of distributions), $V:\Rd\ra[0,\infty)$ is $\gamma$-homogeneous such that it is differentiable in the Sobolev sense, $\nabla V$ is $(\gamma-1)$-homogeneous and
    \begin{align}
        \label{eq:homogoneity_equation}
        \int_{\Rd}\int_{\Rd} \nabla V(x)\cdot k(x-y)\nabla V(y)\diff \rho_\infty(y)\diff\rho_\infty(x)<\infty.
    \end{align}
    The kernel is defined by $K(x,y):=k(x-y)$ and the following condition holds:
    \begin{align*}
        (\alpha-2)\,\gamma > \beta-2.
    \end{align*}
    \end{enumerate}
\end{thm}

We explain the implications for each case:
\begin{itemize}
    \item[\ref{thm:failure_of_slsi_for_regular_convolution_kernels}:] The inequality \eqref{eq:stein_log_sobolev_inequality_general_alpha} fails when $k$ in $C^s(\Rd)\cap L^2(\Rd)$ for $s>d+1$. Theorem \ref{thm:existence_of_maternised_convolution_kernel_that_satisfy_slsi} and Remark \ref{rem:facts_main_thm} show that the inequality \eqref{eq:stein_log_sobolev_inequality_general_alpha} may hold for $s\in [0,1]$, see \eqref{eq:condition_Fourier_transform_k}. There is a gap in regularity $C^s(\Rd)$ for $s\in (1,d+1]$, which remains an open problem.
    \item[\ref{thm:failure of slsi for integrable convolution kernels}:] We conjecture that the kernel must have exponential weights. The inequality \eqref{eq:stein_log_sobolev_inequality_general_alpha} fails for polynomial weights of the form $ (1 + |x|^\beta) f(x) $ for $\beta$ sufficiently small. There is an open gap in our results, and we conjecture that the kernel will also fail for any pair $ \beta \in [0, \infty) $ and $ r \in [1, \infty) $.
    \item[\ref{thm:failure_of_slsi_for_homogeneous_convolution_kernels}:] The inequality \eqref{eq:stein_log_sobolev_inequality_general_alpha} holds if the kernel is not too repulsive and if there is not too much mass is in the tails of $\rho_\infty$. 
\end{itemize}

These failure conditions highlight that the assumptions we make in Theorem \ref{thm:existence_of_maternised_convolution_kernel_that_satisfy_slsi}, are partially necessary. We comment that our failure results complement those by \cite{pmlr-v139-korba21a, MR4582478}. We leave it as an open problem to sharpen the conditions for which \eqref{eq:inequality_l2_lower_bound_formulated_H-1_H1} and \eqref{eq:stein_log_sobolev_inequality} hold.

\subsection{Structure of the paper}
In Section \ref{section:notation and preliminaries} we collect the notation used throughout the paper. Then, in Sections \ref{section:existence_kernels_satisfying_LSS_inequality}, \ref{section:weak solutions to the stein variational gradient descent equation} and \ref{section:necessary_properties_of_the_kernel}, we prove Theorems \ref{thm:existence_of_maternised_convolution_kernel_that_satisfy_slsi}, \ref{thm:existence_of_distributional_solutions_to_mfsvgd} and \ref{thm:conditions_for_slsi}, respectively. The Appendix is comprised of four parts. In Appendix \ref{appendix:incomplete_gamma_functions} we present results about the incomplete gamma function, in Appendix \ref{appendix:prop_H-1} we mention properties of the $H^{-s}$ Sobolev space, in Appendix \ref{appendix:existence_approximations_former_appendix} we provide an existence proof necessary for Theorem \ref{thm:existence_of_distributional_solutions_to_mfsvgd}, and in Appendix \ref{appendix:auxiliary} we state several auxiliary results.

\section{Preliminaries and Notation}\label{section:notation and preliminaries}

For the convenience of the reader, we collect here the notation used throughout the paper.

\smallskip

\underline{\it Function spaces.} We will always work in the spaces $\Rd$ and $[0,T]\times\Rd$ and their subsets. For the latter, we call $t$, $x$ the temporal- and spatial variable respectively for $(t,x)\in [0,T]\times\Rd$. For any $R>0$ we define the open ball of radius $R$ by $B_R:= \{x\in\Rd\;\vert\; |x|<R\}$. The Schwarz space and the space of tempered distributions are denoted with $\mathcal{S}(\Rd)$ and $\mathcal{S}'(\Rd)$, respectively. These function spaces are endowed with a duality pairing, which we denote by $\mathcal{S}'(\Rd)\times \mathcal{S}(\Rd)\ni (u,g)\mapsto \inner{u}{g}_{\mathcal{S}',\mathcal{S}}\in\R$. We let $\Lip_L(\Rd)$ denote the space of $L$-Lipshitz function for $L>0$. We denote by $C^\alpha(\R)$ the $\alpha$-Hölder continuous functions for $\alpha\in (0,1)$. 

\smallskip

Let $\mathcal{M}(\Rd)$ denote the space of signed Radon measures on $\Rd$ with finite total variation. The total variation of $\mu\in\mathcal{M}$ is defined by $\norm{\mu}_{\mathrm{TV}}:=\abs{\mu}(\Rd)$, where we have the Hahn-Jordan decomposition $\mu=\mu^+-\mu^-$ for non-negative measures $\mu^+$, $\mu^-$ and $\abs{\mu}:=\mu^++\mu^-$. Given two measures $\mu$, $\nu$ we write $\nu\ll \mu$ if there exists a measurable function $h:\Rd\ra\R$ such that $\diff \nu(x)=h(x)\diff \mu(x)$. In that case we write $\frac{\diff \nu}{\diff \mu}:= h$. We let $\Pro$ denote the space of probability measures and formally endow the space with the Kantorovich-Rubinstein metric
\begin{align*}
    W_1(\mu,\nu):=\sup_{f\in \Lip_1(\Rd)} \left\vert \int_{\Rd} f(x)\diff \mu(x)- \int_{\Rd} f(x)\diff \nu(x)\right\vert
    \qquad\forall \mu,\nu\in \Pro.
\end{align*}
We say a curve of probability measures $[0,T]\ni t\ra \mu_t \in \mathscr{P}(\Rd)$ is absolutely continuous if there exists a function $f\in L^1(0,T)$ with
\begin{align*}
    \sup_{t\in[0,T]} \int_{\Rd} |x|\diff \mu_t<\infty,\qquad
    W_1(\mu_{t_0},\mu_{t_1})\leq \int_{t_0}^{t_1} |f(s)|\diff s
    \qquad\forall t_0,t_1\in [0,T].
\end{align*}
For $p\in [1,\infty]$ the space of measurable functions with finite $L^p$-norm is denoted by $L^p(\Rd)$. In particular, $L^2(\Rd)$ is endowed with the complex inner product $\inner{f}{g}:=\int_{\Rd}f\cdot \overline{g}\diff x$. The Sobolev space $H^{s}(\Rd)$ for $s\in\R$ is the space of tempered distributions
\begin{equation}\label{eq:definition_H-1}
H^{s}(\Rd) = \{ u \in \mathcal{S}'(\Rd): (1+|\xi|^2)^{\frac{s}{2}}\, \hat{u}(\xi) \in L^2(\Rd) \},
\end{equation}
We remark that \eqref{eq:definition_H-1} in particular implies that $\hat{u}$, a priori an element of $\mathcal{S}'(\Rd)$, is in fact a function (see \cite[Proposition 9.16]{MR1681462} for the proof of equivalency of these definitions). For $s\geq 0$, we define a duality pairing between $H^s(\Rd)$ and $H^{-s}(\Rd)$ by
\begin{align*}
    \langle f, g \rangle_{H^{-s}, H^s}
    :=\inner{f}{g}.
\end{align*}
For any $s,L>0$ we define the space of Sobolev functions with Dirichlet boundary as the following completion under $H^s(\Rd)$ 
\begin{align*}
    H_0^s(B_L):=\overline{\{\vp\in C^\infty (\Rd)\;\vert\; \supp(\vp)\subseteq B_L\}}^{H^s(\Rd)}.
\end{align*}
Given a function space $X$ as above, we define $L^p(0,T;X)$ as the space of equivalence classes of measurable functions $f:[0,T]\ra X$ with finite norm 
\begin{align*}
    \norm{f}_{L^p(0,T;X)}
    :=\left(\int_0^T \norm{f(t)}_X^p \diff t\right)^{\frac{1}{p}}.
\end{align*}
In case $X=H^1(\Rd)$ we write the norm $\norm{f}_{L^p(0,T;X)}=\norm{f}_{L_t^p H_x^1}$. Additionally, we endow the space $L(0,T; H^1(\Rd))$ and $L^2(0,T; H^{-1}(\Rd))$ with a duality pairing
\begin{align*}
    \inner{f}{g}_{L_t^2 H_x^{-1}, L_t^2 H_x^1}:=\int_0^T \inner{f_t}{g_t}_{H^{-1},H^1}\diff t
    \qquad \forall f\in L^2(0,T; H^{-1}(\Rd))
    \quad \forall g\in L^2(0,T; H^1(\Rd)).
\end{align*}
Given a function space $X$ we denote its dual space by $X^\ast$, and for any sequence $\{f_n\}_{n\geq 1}\subseteq X$ and element $f\in X$ we use the following symbols
 $f_n\ra f$, $f_n\weak f$ and $f_n\wstar f$ as $n\ra\infty$ for strong, weak and weak$^\ast$ convergence, respectively.
  
\underline{\it Special operators, functions and sets.} For $p,q,r\in [1,\infty]$ with $1+\frac{1}{r}=\frac{1}{p}+\frac{1}{q}$ (with the convention $\frac{1}{\infty}=0$) and two functions $f\in L^p(\Rd)$, $g\in L^q(\Rd)$, the convolution $f\ast g$ is defined as 
\begin{equation}\label{eq:convolution_def}
f \ast g(x)= \int_{\Rd} f(x-y) \, g(y) \diff y\qquad\forall x\in\Rd.
\end{equation}
By Young's convolution inequality, this function is well-defined and belongs to $L^r(\Rd)$ with the estimate $\norm{f\ast g}_{L^r}\leq \norm{f}_{L^p}\norm{g}_{L^q}$. For a function $f\in L^p(0,T; L^q(\Rd))$ and functions $\vp\in C_c^\infty(\R)$, $\omega\in C_c^\infty(\Rd)$ we define mollification in time and space respectively by
\begin{align*}
    \begin{split}
        f\ast \vp(t,x)
        &:=\int_\R f(s,x)\vp(t-s)\diff s,\\
        f\ast \omega(t,x)
        &:=\int_{\Rd} f(t,y)\omega(x-y)\diff y
    \end{split}
    \qquad\forall (t,x)\in [0,T]\times\Rd.
\end{align*}
It will be clear from the context, which mollification type we use. Given a measurable set $E\subseteq \Rd$ we define the indicator function
    \begin{align*}
        \mathds{1}_E(x):=\begin{cases}
            1 & \text{ if }x\in E,\\
            0 & \text{ if }x\notin E
        \end{cases}
        \qquad\forall x\in\Rd.
    \end{align*}

\underline{\it Fourier transform.} Given an integrable function $f\in L^1(\Rd)$, we define the Fourier transform and inverse Fourier transform respectively by
    \begin{align*}
        \begin{split}
            \rF f(x)&:=\hat f(x)=\int_{\Rd} e^{-2\pi i x\cdot \xi} f(\xi)\diff\xi,\\
            \rF^{-1} f(x)&:=\check f(x)=\int_{\Rd} e^{2\pi i x\cdot \xi} f(\xi)\diff\xi
        \end{split}
        \qquad\forall x\in\Rd.
    \end{align*}
    We extend the transform to distributions by
    \begin{align*}
        \inner{\hat{u}}{g}
        :=\inner{\mu}{\hat{g}}
        \quad\text{ and }\quad
        \inner{\check{u}}{g}
        :=\inner{\mu}{\check{g}}
        \qquad\forall u\in \mathcal{S}'(\Rd)\quad\forall g\in \mathcal{S}(\Rd).
    \end{align*}

\underline{\it Radially symmetric functions.} Let  $g:\Rd\ra\R$ be a radially symmetric function, i.e. there exists $f:[0,\infty)\ra \R$ such that $g(\xi)=f(|\xi|)$. By a small abuse of notation, we do not distinguish between $f$ and $g$. It will be always clear from the context which function we are using.

\section{Existence of Kernels that satisfy the Stein-log-Sobolev-Inequality}\label{section:existence_kernels_satisfying_LSS_inequality}

The target of this section is to prove Theorem \ref{thm:existence_of_maternised_convolution_kernel_that_satisfy_slsi}. Let us sketch our arguments. We will assume that $k$ is radially symmetric. The first step is to rewrite the dissipation in \eqref{eq:inequality_l2_lower_bound_formulated_H-1_H1} as 
\begin{align*}
 \DD^2(\rho\,||\,\rho_\infty) &= \left \langle \nabla (\rho\, e^{V}) \, e^{-\frac{\pot}{2}}, k\ast \left(\nabla (\rho\, e^{V}) \, e^{-\frac{\pot}{2}}\right)  \right \rangle_{H^{-1}(\Rd), H^1(\Rd)} \\
 &= \left \langle \nabla ((\rho - \tau\, e^{-V})\, e^{V}) \, e^{-\frac{\pot}{2}}, k\ast \left(\nabla ((\rho - \tau\, e^{-V})\, e^{V}) \, e^{-\frac{\pot}{2}}\right)  \right \rangle_{H^{-1}(\Rd), H^1(\Rd)} 
\end{align*}
for any $\tau \in \R$. Letting $g := (\rho - \tau\, e^{-V})\, e^{V-\frac{\pot}{2}}$, we prove in Section \ref{subsection:proof strategy to construct kernels satisfying sls} for a particular value $\tau$ that
\begin{align*}
    \DD^2(\rho\,||\,\rho_\infty) &\geq C\, \int_{\Rd}\hat{k}(\xi)\, \left|(2\pi i \xi)\, \hat{g}(\xi) + \frac{i}{2\,(2\pi)}  \nabla \hat{g}(\xi)  \right|^2 \diff \xi\\ 
    &=C\,\left( \int_{\Rd} |\hat{g}(\xi)|^2\, q(\xi)\diff\xi+ \frac{1}{16\pi^2}\int_{\Rd} |\nabla \hat{g}(\xi)|^2 \hat{k}(\xi) \diff\xi\right),
\end{align*}
where $C>0$ is a constant that depends on $\Sigma$ and the coefficients in \eqref{eq:condition_Fourier_transform_k}, $q:\Rd\ra\R$ is a certain weight depending explicitly on $k$. We observe crucially that in the first term on the RHS, we can replace $\hat{k}$ with any other kernel that satisfies \eqref{eq:condition_Fourier_transform_k} at the cost of a smaller constant in \eqref{eq:inequality_l2_lower_bound_formulated_H-1_H1}. Hence, our strategy is to find a single kernel $k_{0,d}$ that satisfies the inequality \eqref{eq:inequality_l2_lower_bound_formulated_H-1_H1}. We will work with the second term on the RHS of the above inequality. To illustrate the difficulties, suppose that $q\geq \lambda>0$, then by the Plancherel theorem and assumption~\eqref{eq:general_prop_target_potentialV_f}
\begin{align}
\label{eq:L2_lower_bound_of_stein_fisher_info_intro}
\begin{split}
\DD^2(\rho\,||\,\rho_\infty)
	&\geq \lambda\int_{\Rd} |g(x)|^2 \diff x 
 =\int_{\R} \frac{\left|\rho - \tau\, e^{-V}\,\right|^2}{e^{-V}}\; e^{V-\pot} \diff x \\
        &\geq \lambda e^{\ngib}  \int_{\R}  \frac{\left|(\rho - \tau\, Z \rho_{\infty})\right|^2}{Z\, \rho_{\infty}} \diff x  
        \geq \frac{\lambda e^{\ngib}}{Z}  \int_{\R}\left\vert\frac{\rho}{\rho_\infty}-\pwc\,Z\right\vert^2 \diff \rho_\infty,
\end{split}	
\end{align}
where $Z = \int_{\Rd} e^{-V(x)}\diff x$. The conclusion would follow immediately by Lemma \ref{lem:chi_squared_bounded_from_below_by_kullback_leibler}. 

\smallskip

Unfortunately, a strict positive lower bound on $q$ would violate the positive definiteness of $k$, which can be seen in \eqref{eq:kernel_freq_bounded_at_origin_formula}. Luckily, we can state sufficient conditions on both $k$ and $q$ to satisfy \eqref{eq:L2_lower_bound_of_stein_fisher_info_intro}. We do this by compensating the non-positive values of $q$ with the information on the gradient $\nabla \hat{g}$ via a Poincaré-Wirtinger argument, see Section \ref{subsection:sufficient condition on q}. Our approach is to reconstruct $k$ from a given weight $q$ with sufficient properties as described above. This is performed in Section \ref{subsection:reconstructing kernels from weights}, where we derive a formula for $\hat{k}$ in terms of $q$. In Section \ref{sect:exist_kernel_any_dimensio} we further restrict our analysis to a particular step function $q$ in order to prove Theorem~\ref{thm:existence_of_maternised_convolution_kernel_that_satisfy_slsi} for dimension $d\geq 2$. But before we present our argument, in Section \ref{sect:proof_lSS_ineq_1D} we use the above methods in one dimension for a particular Matérn kernel $k(x)=e^{-|x|}$, thus proving Theorem \ref{thm:existence_of_maternised_convolution_kernel_that_satisfy_slsi} in 1D. This serves as a proof of concept.  

\subsection{Fisher information in Fourier variables}
\label{subsection:proof strategy to construct kernels satisfying sls}

In the following technical key result we pass to the Fourier variables. We observe that $\DD^2(\rho\,||\,\rho_\infty)$ is an action of $H^{-1}(\Rd)$ functionals on $H^1(\Rd)$ functions. Thus to pass to the Fourier variables, we need to adapt the classical Plancherel formula to this setting. This is done in Lemma \ref{lem:Plancherel_f_H-1}.

\begin{lem}
    \label{lem:passing_to_fourier}
    Let $\rho$ satisfy \eqref{eq:regularity_to_be_satisfied_ineqality_rho}, $\pwc\in\R$ and $k\in L^1(\Rd)+L^2(\Rd)$ satisfy \eqref{eq:condition_Fourier_transform_k}. Then, let $k_\Sigma(x):=\det(\Sigma^{-\frac{1}{2}})\,k(\Sigma^{-\frac{1}{2}}x)$ and let $h$ be either $k$ or $k_\Sigma$. Then, the dissipation $ \DD_h^2(\rho\,||\,\rho_\infty)$, defined in \eqref{eq:inequality_l2_lower_bound_formulated_H-1_H1}, satisfies
    \begin{align}
        \label{eq:dissipation_bnd_below_ft}
        \begin{split}
            \DD_h^2(\rho\,||\,\rho_\infty)
        &\geq C_h\, \int_{\Rd}\hat{k}(\xi)\, \left|(2\pi i \xi)\, \hat{g}(\xi) + \frac{i}{2\,(2\pi)}  \nabla \hat{g}(\xi)  \right|^2 \diff \xi\\
        & = C_h\left(\int_{\Rd} |\hat{g}(\xi)|^2 \, q(\xi) \diff\xi+\frac{1}{16\,\pi^2}\int_{\Rd} |\nabla \hat{g}(\xi)|^2 \,\hat{k}(\xi) \diff\xi\right),
        \end{split}
    \end{align}
     where we define the functions
     \begin{align}
        \label{eq:def_g}
          \begin{split}
              g_0(x)&: = \Big(\rho(x) - \tau\, e^{-V(x)}\Big)\, e^{V(x)-\frac{\pot(x)}{2}}, \qquad 
          z(x):=\Sigma^{\frac{1}{2}}x+\mu,\qquad
          g:=g_0\circ z,\\
        q(\xi)&:=4\pi^2\abs{\xi}^2 \hat{k}(\xi)-\frac{1}{2}\divv(\xi\, \hat{k}(\xi)).
          \end{split}
     \end{align}
     and the constants
     \begin{align}
     \label{eq:constants_of_passing_ft}
         C_k:=\frac{\sqrt{\det(\Sigma)}}{\sigma_d}\frac{(\sigma_1\wedge 1)}{\rkc_k^2},\qquad
         C_{k_\Sigma}:=\frac{\sqrt{\det(\Sigma)}}{\sigma_d},
     \end{align}
with $\rkc_k$ being the constant from \eqref{eq:condition_Fourier_transform_k} and $0<\sigma_1 \leq ... \leq \sigma_d$ being eigenvalues of $\Sigma$.
\end{lem}
\begin{proof}
It is straightforward to check
\begin{align*}
 \DD_h^2(\rho\,||\,\rho_\infty) &= \left \langle \nabla (\rho\, e^{V}) \, e^{-\frac{\pot}{2}}, h\ast \left(\nabla (\rho\, e^{V}) \, e^{-\frac{\pot}{2}}\right)  \right \rangle_{H^{-1}(\Rd), H^1(\Rd)} \\
 &= \left \langle \nabla ((\rho - \tau\, e^{-V})\, e^{V}) \, e^{-\frac{\pot}{2}}, h\ast \left(\nabla ((\rho - \tau\, e^{-V})\, e^{V}) \, e^{-\frac{\pot}{2}}\right)  \right \rangle_{H^{-1}(\Rd), H^1(\Rd)}. 
\end{align*}
We can represent the dissipation in terms of $g_0$ as
\begin{align*}
    \DD_h^2(\rho\,||\,\rho_\infty) &= \left \langle \nabla g_0+\frac{1}{2}\Sigma^{-1}(x-\mu)\; g_0, h\ast \Big(\nabla g_0+\frac{1}{2}\Sigma^{-1}(x-\mu)\; g_0\Big)  \right \rangle_{H^{-1}(\Rd), H^1(\Rd)}.
\end{align*}
We aim at applying the Plancherel formula in the $H^{-1}$ setting as in \ref{lem:Plancherel_f_H-1:plancherel} from Lemma \ref{lem:Plancherel_f_H-1}. Therefore, we have to compute the Fourier transforms of the terms above. We have
\begin{equation}
\label{eq:fourier_two_expressions}
\mathcal{F}(\nabla g_0) = (2\pi i\xi)\, \hat{g_0}, \quad  \mathcal{F}\left(\frac{1}{2}\Sigma^{-1}(x-\mu)\; g_0\right) = e^{-2\pi i \xi \mu} \,\frac{i}{2\,(2\pi)} \Sigma^{-1} \nabla\left( e^{2\pi i \xi \mu} \, \hat{g_0}\right).
\end{equation}
These formulas hold in the sense of distributions. But as a consequence of Lemma \ref{lem:representation_derivative_f_x_fourier}, we know $\nabla \hat{g_0}$ exists in the Sobolev sense. Let us write the formulas above in terms of $g$. Recall that $g_0(x) = g(\Sigma^{-\frac{1}{2}}(x-\mu))$. Hence, a direct computation gives
\begin{equation}
\label{eq:fourier_formula_g_0}
\hat{g_0}(\xi) = \det(\Sigma^{\frac{1}{2}}) \, e^{-2\pi i \xi \mu}\,  \hat{g}(\Sigma^{\frac{1}{2}} \xi).
\end{equation}
It follows that we can express 
\begin{equation}
\label{eq:fourier_second_expression_gradient_in_terms_of_g}
\begin{split}
\mathcal{F}\left(\frac{1}{2}\Sigma^{-1}(x-\mu)\; g_0\right) &= \det(\Sigma^{\frac{1}{2}}) \, e^{-2\pi i \xi \mu} \,\frac{i}{2\,(2\pi)} \Sigma^{-1} \nabla\left( \hat{g}(\Sigma^{\frac{1}{2}} \xi)\right) \\
&= \det(\Sigma^{\frac{1}{2}}) \, e^{-2\pi i \xi \mu} \,\frac{i}{2\,(2\pi)} \Sigma^{-\frac{1}{2}} \nabla \hat{g}(\Sigma^{\frac{1}{2}} \xi). 
\end{split}
\end{equation}
Next, we use \ref{lem:Plancherel_f_H-1:plancherel} in Lemma \ref{lem:Plancherel_f_H-1} and \eqref{eq:fourier_two_expressions}--\eqref{eq:fourier_second_expression_gradient_in_terms_of_g} to obtain
$$
\DD_h^2(\rho\,||\,\rho_\infty) = \det(\Sigma) \, \int_{\Rd}\hat{h}(\xi)\, \left|(2\pi i\xi)\, \hat{g}(\Sigma^{\frac{1}{2}} \xi) + \frac{i}{2\,(2\pi)} \Sigma^{-\frac{1}{2}} \nabla \hat{g}(\Sigma^{\frac{1}{2}} \xi) \right|^2 \diff \xi,
$$
where we used that $|e^{-2\pi i \xi \mu}| = 1$. We apply the change of variables $\eta = \Sigma^{\frac{1}{2}} \xi$ to get
\begin{align}
    \label{eq:first_lower_bnd_dis}
    \begin{split}
        \DD_h^2(\rho\,||\,\rho_\infty)  &= \sqrt{\det(\Sigma)} \, \int_{\Rd}\hat{h}(\Sigma^{-\frac{1}{2}}\eta)\, \left|\Sigma^{-\frac{1}{2}} \left((2\pi i \eta)\, \hat{g}(\eta) + \frac{i}{2\,(2\pi)}  \nabla \hat{g}(\eta) \right) \right|^2 \diff \eta\\
&\geq \frac{\sqrt{\det(\Sigma)}}{\sigma_d} \int_{\Rd}\hat{h}(\Sigma^{-\frac{1}{2}}\eta)\, \left|(2\pi i \eta)\, \hat{g}(\eta) + \frac{i}{2\,(2\pi)}  \nabla \hat{g}(\eta)  \right|^2 \diff \eta.
    \end{split}
\end{align}
We now transform the term $\hat{h}(\Sigma^{-\frac{1}{2}}\eta)$ and prove
\begin{equation}
\label{eq:dissipation_local_target_splitting_two_k}
\DD_h^2(\rho\,||\,\rho_\infty) \geq C_h\, \int_{\Rd}\hat{k}(\eta)\, \left|(2\pi i \eta)\, \hat{g}(\eta) + \frac{i}{2\,(2\pi)}  \nabla \hat{g}(\eta)  \right|^2 \diff \eta. 
\end{equation} 
Indeed, if $h(x) = k_{\Sigma}(x) = \det(\Sigma^{-\frac{1}{2}})k(\Sigma^{-\frac{1}{2}}x)$, we have $\hat{h}(\eta) = \hat{k}(\Sigma^{\frac{1}{2}}\eta)$ so $\hat{h}(\Sigma^{-\frac{1}{2}}\eta) = \hat{k}(\eta)$ and we arrive at \eqref{eq:dissipation_local_target_splitting_two_k}. On the other hand, if $h(x) = k(x)$, we estimate
$$
1+ |\Sigma^{-\frac{1}{2}}\eta|^2 \leq 1 + \frac{1}{\sigma_{1}}\, |\eta|^2 \leq \frac{1}{(\sigma_1\wedge 1)} \, (1+|\eta|^2),
$$
so using \eqref{eq:condition_Fourier_transform_k}, we obtain
$$
\hat{k}(\Sigma^{-\frac{1}{2}}\eta) \geq \frac{1}{\rkc_k} \frac{1}{1+ |\Sigma^{-\frac{1}{2}}\eta|^2 }  \geq  \frac{(\sigma_1\wedge 1)}{\rkc_k} \frac{1}{1+|\eta|^2} \geq \frac{(\sigma_1\wedge 1)}{\rkc_k^2} \hat{k}(\eta).
$$
and hence, we again deduce \eqref{eq:dissipation_local_target_splitting_two_k}.

\smallskip

Using \eqref{eq:dissipation_local_target_splitting_two_k} we can easily obtain the assertion of the lemma. First, we expand the square to get
\begin{align*}
\DD_h^2(\rho\,||\,\rho_\infty)
    \geq   C_h\Big(\inner{4\pi^2\hat{k}\, \abs{\eta}\hat{g}}{\abs{\eta}\hat{g}}
    &+\frac{1}{2}\inner{\hat{k}\, \eta\hat{g}}{\nabla\hat{g}}\\
    &+\frac{1}{2}\inner{\hat{k}\, \nabla\hat{g}}{\eta\hat{g}}
    +\frac{1}{16\pi^2}\inner{\hat{k}\,\nabla \hat{g}}{\nabla \hat{g}}\Big),
\end{align*}
where we recall that $\inner{\cdot}{\cdot}$ denotes the complex $L^2(\Rd)$ inner product. Finally, we write
$$
    \inner{\hat{k}\, \eta\hat{g}}{\nabla\hat{g}}
    +\inner{\hat{k}\, \nabla\hat{g}}{\eta\hat{g}}
    =-\inner{\divv(\hat{k}\, \eta)\hat{g}}{\hat{g}}
    -\inner{\hat{k}\, \eta\cdot \nabla\hat{g}}{\hat{g}}
    +\inner{\hat{k}\, \nabla\hat{g}}{\eta\hat{g}} =-\inner{\divv(\hat{k}\, \eta)\hat{g}}{\hat{g}},
$$
so that we have proven the inequality \eqref{eq:dissipation_bnd_below_ft}.
\end{proof}

\begin{rem}
    Note that we scale $k_\Sigma$ with $\det(\Sigma^{-\frac{1}{2}})$ to formally preserve the mass of $k$ and to ensure better stability of the constant $\lambda$ in \eqref{eq:inequality_l2_lower_bound_formulated_H-1_H1}.
\end{rem}

\subsection{Proof of Theorem \ref{thm:existence_of_maternised_convolution_kernel_that_satisfy_slsi} in 1D}\label{sect:proof_lSS_ineq_1D}

In this section, we prove the Stein-log-Sobolev inequality in 1D. As explained in the beginning of the next section, this reasoning works only in one space dimension.

\begin{proof}[Proof of Theorem \ref{thm:existence_of_maternised_convolution_kernel_that_satisfy_slsi} in 1D] We split the argument into four steps.

    \underline{\it Passing to a standard kernel} The first step is to derive the inequality
    \begin{align}
        \label{eq:dissipation_bounded_by_standard_kernel_1d}
        \DD_k^2(\rho\,||\,\rho_\infty)
        \geq \frac{C_k}{\rkc_k \rkc_{k_{0,1}}}\, \int_{\Rd}\hat{k}_{0,1}(\xi)\, \left|(2\pi i \eta)\, \hat{g}(\xi) + \frac{i}{2\,(2\pi)}  \nabla \hat{g}(\xi)  \right|^2 \diff \xi,
    \end{align}
    where $k_{0,1}\in L^1(\Rd)+L^2(\Rd)$ is a fixed kernel that satisfies \eqref{eq:condition_Fourier_transform_k}, $C_k$ is the constant in \eqref{eq:constants_of_passing_ft} and $\rkc_k$, $\rkc_{k_{0,1}}$ are the constants of $k$ and $k_{0,1}$ respectively in \eqref{eq:condition_Fourier_transform_k}. By assumption the given kernel $k$ satisfies all assumptions of Lemma \ref{lem:passing_to_fourier}. This means that the following lower estimate holds for all $\rho$ satisfying \eqref{eq:regularity_to_be_satisfied_ineqality_rho}, 
    \begin{align}
        \label{eq:application_dissipation_bound_1d}
        \DD_k^2(\rho\,||\,\rho_\infty)
        \geq C_k\, \int_{\Rd}\hat{k}(\xi)\, \left|(2\pi i \eta)\, \hat{g}(\xi) + \frac{i}{2\,(2\pi)}  \nabla \hat{g}(\xi)  \right|^2 \diff \xi,
    \end{align}
    where $C_k$ is the constant in \eqref{eq:constants_of_passing_ft}. Inspired by \cite{MR4582478}, we choose a particular Matérn kernel
    \begin{align}
        \label{eq:def_k01}
        k_{0,1}(x):=e^{-|x|}.
    \end{align}
    One can show that this kernel has the Fourier transform $\hat{k}_{0,1}(\xi)=\frac{1}{1+4\pi^2 \xi^2}$. Thus as both kernels $k$ and $k_{0,1}$ satisfy \eqref{eq:condition_Fourier_transform_k}, we can show $\rkc_k \rkc_{k_{0,1}}\,\hat{k}\geq \hat{k}_{0,1}$. Then combining this inequality with \eqref{eq:application_dissipation_bound_1d}, we immediately see that \eqref{eq:dissipation_bounded_by_standard_kernel_1d} holds true.
    
    \underline{\it Dissipation dominates $L^2$-norm} As in \eqref{eq:dissipation_bnd_below_ft}, we can rewrite the inequality \eqref{eq:dissipation_bounded_by_standard_kernel_1d} as
    \begin{align*}
        \DD_k^2(\rho||\rho_\infty)
        \geq \frac{C_k}{\rkc_k \rkc_{k_{0,1}}}\left(\int_{\R} |\hat{g}|^2 q\diff\xi+\frac{1}{16\pi^2}\int_{\R} |\nabla \hat{g}|^2 \hat{k}_{0,1}\diff\xi\right)
    \end{align*}
    where $g$ and $q$ are defined in \eqref{eq:def_g} using $k_{0,1}$. The next step is to establish the following lower bound 
    \begin{align}
        \label{thm:existence_of_maternised_convolution_kernel_that_satisfy_slsi:eq1}
        \int_{\R} |\hat{g}|^2 q\diff\xi+\frac{1}{16\pi^2}\int_{\R} |\nabla \hat{g}|^2 \hat{k}_{0,1}\diff\xi
        \geq \lambda_{0,1} \int_{\R} |g|^2\diff\xi
    \end{align}
    for some $\lambda_{0,1}>0$. To this end, we explicitly compute the weight $q$. This yields
    \begin{align*}
        q(\xi)
        &=4\pi^2 |\xi|^2 \hat{k}_{0,1}(\xi)
        -\frac{1}{2}\divv(\xi \hat{k}_{0,1}(\xi))
        =\frac{4\pi^2 \xi^2}{1+4\pi^2 \xi^2}
        +\frac{4\pi^2\xi^2}{(1+4\pi^2\xi^2)^2}
        -\frac{1}{2}\frac{1}{1+4\pi^2\xi^2} = \\
        &= \frac{(4\pi^2 \xi^2-\frac{1}{2})(1+4\pi^2 \xi^2)+4\pi^2 \xi^2}{(1+4\pi^2 \xi^2)^2}
        =\frac{(w-\frac{1}{2})(1+w)+w}{(1+w)^2},
    \end{align*}
    where we write $w=4\pi^2 \xi^2$. The numerator is a convex quadratic function with critical point at  $w_c=-\frac{3}{4}$ and only one positive root at $w_\ast=\frac{\sqrt{17}-3}{4}$. Then $q$ has negative values only in the ball $B_\varepsilon$ with radius $\varepsilon:=\frac{\sqrt{\sqrt{17}-3}}{2\pi}$. Given $\delta>0$ we define $\beta:=q(\varepsilon+\delta)$ and $\alpha:=-q(0)=\frac{1}{2}$. By a direct computation $q'(\xi) = \frac{w+5}{(1+w)^3}\,4 \pi^2 \xi \geq 0$ so that $q$ is increasing for $\xi \geq 0$. Then, we have the lower bound
    \begin{align*}
        q(\xi)
        \geq \begin{cases}
            -\alpha & \text{ if }|\xi|<\varepsilon+\delta,\\
            \phantom{-}\beta & \text{ if }|\xi|\geq \varepsilon+\delta.
        \end{cases}
    \end{align*}
    We use this to estimate the first term in the LHS of \eqref{thm:existence_of_maternised_convolution_kernel_that_satisfy_slsi:eq1} as follows
	\begin{multline*}
		\int_{\R \setminus B_{\varepsilon+\delta}}|\hat g(\xi)|^2 q(\xi)\diff \xi
		+\int_{B_{\varepsilon+\delta}}|\hat g(\xi)|^2 q(\xi)\diff \xi
		\geq \, \beta\int_{\R \setminus B_{\varepsilon+\delta}}|\hat g(\xi)|^2\diff \xi
		-\alpha\int_{B_{\varepsilon+\delta}}|\hat g(\xi)|^2 \diff \xi.
    \end{multline*}
    We will choose $\tau\in\R$, in the definition of $g$, such that
    $\int_{B_{\varepsilon+\delta}}\hat g(\xi)\diff \xi=0$ (note that this is analogue to \eqref{eq:def_tau}). This allows us to use the Poincaré-Wirtinger inequality and get
    \begin{align*}
        \int_{B_{\varepsilon+\delta}}|\hat g|^2\diff \xi
        \leq \left(\frac{\varepsilon+\delta}{p_{1/2,1}}\right)^2 \,\int_{B_{\varepsilon+\delta}}|\nabla \hat g|^2\diff \xi
        \leq \frac{1}{\hat{k}_{0,1}(\varepsilon+\delta)}\left(\frac{\varepsilon+\delta}{p_{1/2,1}}\right)^2 \,\int_{B_{\varepsilon+\delta}}|\nabla \hat g|^2\, \hat{k}_{0,1}\diff \xi.
    \end{align*}
    Above, we used the monotonicity of $\hat{k}_{0,1}$ and the optimal constant in the Poincaré-Wirtinger inequality from Lemma \ref{lem:poincare-wirtinger-optimal}. In order for \eqref{thm:existence_of_maternised_convolution_kernel_that_satisfy_slsi:eq1} to hold, it now suffices to find $\delta$ such that
    \begin{align*}
        \frac{\alpha}{\hat{k}_{0,1}(\varepsilon+\delta)}\left(\frac{\varepsilon+\delta}{p_{1/2,1}}\right)^2
        <\frac{1}{16\pi^2}.
    \end{align*}
    By continuity, such $\delta$ exists if
    \begin{align}
    	\label{thm:existence of maternised convolution kernel that satisfy slsi:eq2}
        \frac{\alpha}{\hat{k}_{0,1}(\varepsilon)}\left(\frac{\varepsilon}{p_{1/2,1}}\right)^2
        <\frac{1}{16\pi^2},
    \end{align}
which is true because $\frac{\alpha}{\hat{k}_{0,1}(\varepsilon)}\left(\frac{\varepsilon}{p_{1/2,1}}\right)^2 \approx 0.00335$ and $\frac{1}{16\pi^2} \approx 0.00633$. The optimal value for $\lambda_{0,1}$ is given by
    \begin{align}\label{eq:lambda_0_one_dimension}
    \lambda_{0,1} =    \sup_{\delta>0}\min\left\{q(\varepsilon+\delta), \;\frac{\hat{k}_{0,1}(\varepsilon+\delta)}{16\pi^2}\left(\frac{p_{1/2,1}}{\varepsilon+\delta}\right)^2-\alpha \right\}.
    \end{align}

\underline{\it $L^2$-norm dominates Kullback-Leibler} We use the definition of function~$g$ in \eqref{eq:def_g} and  the change of variables $y = z(x) = \Sigma^{\frac{1}{2}}x + \mu$ 
    \begin{align*}
        \DD_k^2(\rho\,||\,\rho_\infty)
        \geq \frac{C_k \lambda_{0,1}}{\rkc_k \rkc_{k_{0,1}}} \int_{\R}|g_0(z(x))|^2\diff x
        = \frac{C_k \lambda_{0,1}}{\rkc_k \rkc_{k_{0,1}}\sqrt{\det(\Sigma)}} \int_{\R}|g_0(y)|^2\diff y.
        \end{align*}
Finally, we exploit the definition of $g_0$ and use property \eqref{eq:general_prop_target_potentialV_f} to obtain
    \begin{align*}
    &\DD_k^2(\rho\,||\,\rho_\infty)
        \geq \frac{C_k\, \lambda_{0,1}}{\rkc_k \rkc_{k_{0,1}}\sqrt{\det(\Sigma)}}  \int_{\R} \left|(\rho - \tau\, e^{-V})\, e^{V-\frac{\pot}{2}}\right|^2 \diff x \\
        &\quad\geq \frac{C_k\, \lambda_{0,1}\, e^\ngib}{\rkc_k \rkc_{k_{0,1}}\sqrt{\det(\Sigma)}}  \int_{\R}  \frac{\left|(\rho - \tau\, Z \rho_{\infty})\right|^2}{Z\, \rho_{\infty}} \diff x 
        \geq \frac{C_k\, \lambda_{0,1}\, e^\ngib}{\rkc_k \rkc_{k_{0,1}}\sqrt{\det(\Sigma)} Z}  \int_{\R}\left\vert\frac{\diff\rho}{\diff\rho_\infty}-\pwc\,Z\right\vert^2 \diff \rho_\infty,
    \end{align*}
    where $Z = \int_{\R} e^{-V} \diff x$. Using Lemma \ref{lem:chi_squared_bounded_from_below_by_kullback_leibler}, we can bound the RHS by the Kullback-Leibler distance to deduce
    \begin{align*}
        \DD_k^2(\rho\,||\,\rho_\infty)
        \geq \lambda\int_{\R}\left\vert\frac{\diff\rho}{\diff\rho_\infty}-\pwc\,Z\right\vert^2 \diff \rho_\infty
        \geq \lambda\,\KL(\rho||\rho_\infty),
    \end{align*}
 where the Stein-log-Sobolev constant is given by
 \begin{align*}
     \lambda = \frac{C_k\, \lambda_{0,1}\, e^\ngib}{\rkc_k \rkc_{k_{0,1}}\sqrt{\det(\Sigma)} Z}=\frac{\lambda_{0,1}}{\rkc_k^{\,3} \,\rkc_{k_{0,1}}}\cdot\frac{e^{\ngib}}{Z}\cdot\frac{\sigma_1\wedge 1}{\sigma_d}.
 \end{align*}
\underline{\it Proof of $L^2$ and $H^{-1}$ estimate} In order to derive the $L^2$ estimate in \eqref{eq:bound_L2_H-1_norm_by_dissipation} we use the inequality \eqref{eq:inequality_l2_lower_bound_formulated_H-1_H1} that we proved above. We must somehow estimate the value $\tau$ (given by \eqref{eq:def_tau}) in terms of $\DD^2(\rho||\rho_\infty)$ and $\norm{\rho}_{L^1}$. An estimate on $\norm{\rho e^{V-\frac{\pot}{2}}}_{L^2}$ will immediately follow from this. Take any bounded open interval $I\subseteq \R$ and note that
\begin{align*}
        \lambda\int_{I} \left|\rho e^{V-\frac{\pot}{2}} - \tau\, e^{-\frac{\pot}{2}} \right|^2 \diff x
        \leq \lambda\int_{\R} \left|\rho e^{V-\frac{\pot}{2}} - \tau\, e^{-\frac{\pot}{2}} \right|^2 \diff x
        \leq \DD_k^2(\rho\,||\,\rho_\infty).
\end{align*}
We expand the square, use Cauchy-Schwarz and estimate
\begin{align}
    \label{eq:bound_tau_on_bnd_set}
    \begin{split}
        \lambda\int_{I} &\left|\rho e^{V-\frac{\pot}{2}}\right|^2 \diff x
    +\lambda N \tau^2 
    \leq \DD_k^2(\rho\,||\,\rho_\infty)+2\lambda \tau \int_I \rho\, e^{V-\pot}\diff x\\
    &\leq \DD_k^2(\rho\,||\,\rho_\infty)+2\lambda e^m \tau\norm{\rho}_{L^1}
    \leq \DD_k^2(\rho\,||\,\rho_\infty)+\frac{\lambda N}{2}\tau^2+\frac{2\lambda e^{2m}}{N}\norm{\rho}_{L^1}^2,
    \end{split}
\end{align}
where we define $m:=\norm{V-\pot}_{L^\infty(I)}$, $N:=\int_{I}e^{-\pot} \diff x$ and we used that $\tau\in \R$. We thus derive a bound on $\tau$ by
\begin{align*}
    \tau^2 \leq \frac{2}{\lambda N}\DD_k^2(\rho\,||\,\rho_\infty)+\frac{4e^{2m}}{N^2}\norm{\rho}_{L^1}^2.
\end{align*}
Now by the reverse triangle inequality we bound the $L^2$ norm of $\rho e^{V-\frac{\pot}{2}}$ by
\begin{align*}
    \norm{\rho e^{V-\frac{\pot}{2}}}_{L^2}
    \leq \sqrt{\frac{1}{\lambda}\DD_k^2(\rho||\rho_\infty)}
    +\tau\norm{e^{-\pot}}_{L^1}^{\frac{1}{2}}
    \leq C\big(\norm{\rho}_{L^1}+\sqrt{\DD_k^2(\rho||\rho_\infty)}\;\big),
\end{align*}
where $C>0$ is a constant independent of $\rho$. Using \ref{lem:Plancherel_f_H-1:plancherel} from Lemma \ref{lem:Plancherel_f_H-1} and \eqref{eq:condition_Fourier_transform_k} we observe that
\begin{align*}
    \frac{1}{\rkc_k}\norm{\nabla (\rho e^V)e^{-\frac{\pot}{2}}}_{H^{-1}}^2
    \leq \DD^2(\rho||\rho_\infty)
    \leq \rkc_k \norm{\nabla (\rho e^V)e^{-\frac{\pot}{2}}}_{H^{-1}}^2.
\end{align*}
As a consequence of this inequality and the identity \eqref{eq:split_stein_gradient}, we find the additional bound
\begin{align*}
    \norm{\rho e^{V-\frac{\pot}{2}}\nabla V_0}_{H^{-1}}
    &\leq \norm{\nabla (\rho e^{V-\frac{\pot}{2}})}_{H^{-1}}+\norm{\nabla(\rho e^{V})e^{-\frac{\pot}{2}}}_{H^{-1}}\\
    &\leq C\big(\norm{\rho e^{V-\frac{\pot}{2}}}_{L^2}+\sqrt{\DD^2(\rho||\rho_\infty)}\big)
    \leq C\big(\norm{\rho}_{L^1}+\sqrt{\DD^2(\rho||\rho_\infty)}\big),
\end{align*}
where $C>0$ is a varying constant independent of $\rho$. This concludes the proof.
\end{proof}

\begin{rem}
    The proof above works also for the kernel $k_{\Sigma}(x) = k(\Sigma^{-\frac{1}{2}}x)$ because it also satisfies Lemma \ref{lem:passing_to_fourier}. Its advantage is that it yields a better constant $\lambda = \frac{\lambda_{0,1}}{\rkc_k^3 \rkc_{k_{0,1}}}\cdot\frac{e^{\ngib}}{Z}\cdot~\frac{1}{\sigma_d}$. However, the kernel depends on $\Sigma$ and involves multiplication $\Sigma^{-\frac{1}{2}}x$, which increases the computational cost. We elected to concentrate the analysis on kernels $k$ that are independent of the distribution $\rho_{\infty}$. 
\end{rem}

\subsection{Sufficient condition on $\hat{k}$ and $q$}
\label{subsection:sufficient condition on q}
The proof in Subsection \ref{sect:proof_lSS_ineq_1D} will fail in higher dimensions, since the condition \eqref{thm:existence of maternised convolution kernel that satisfy slsi:eq2} will not hold for $d \geq 2$, see Figure \ref{fig:failure_of_matern_kernel_in_higher_dim}. This impels us to construct new kernels $k$ that overcome the failure of this condition. In this section, for a general kernel $k$, we formulate sufficient conditions on $\hat{k}$ and $q$ to obtain the lower bound \eqref{thm:existence_of_maternised_convolution_kernel_that_satisfy_slsi:eq1} as in the proof of Theorem~\ref{thm:existence_of_maternised_convolution_kernel_that_satisfy_slsi} in 1D. 

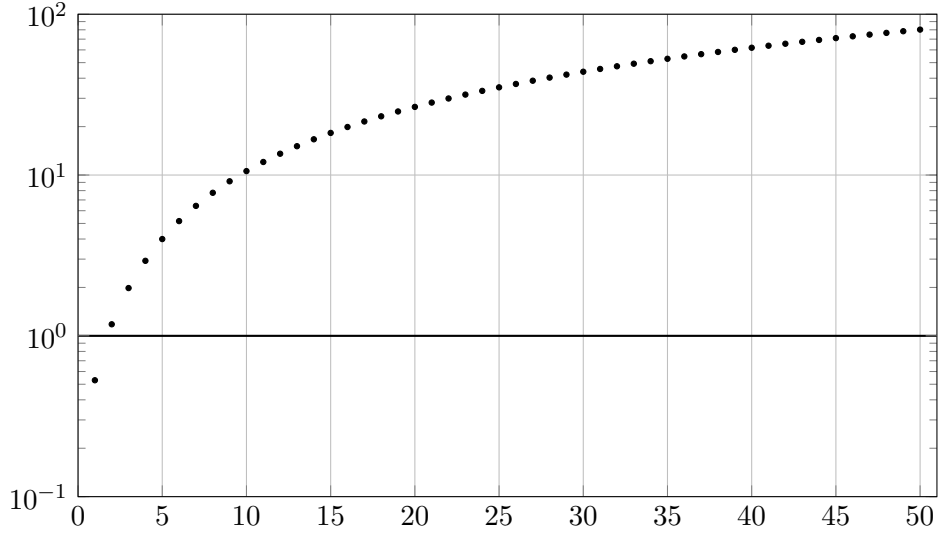
\begin{figure}[H]
    \centering
    \begin{tikzpicture}
    \begin{semilogyaxis}[
        width=13cm,
        height=8cm,
        xlabel={},
        ylabel={},
        grid=both,
        minor grid style={gray!25},
        major grid style={gray!50},
        ymajorgrids=true,
        yminorgrids=false,
        ymin=0.1, ymax=100,
        xmin=0, xmax=51,
        extra y ticks={1},
        extra y tick labels={},
        extra y tick style={grid=major, major grid style={black, thick}}
    ]
        % Example data points
        \addplot[
            only marks,
            mark=*,
            mark options={scale=0.5, fill=black}
        ] coordinates {
            (1, 0.529412) 
            (2, 1.17996)
            (3, 1.98208)
            (4, 2.92803)
            (5, 3.99794)
            (6, 5.17092)
            (7, 6.42876)
            (8, 7.7567)
            (9, 9.14299)
            (10,10.5784)
            (11, 12.0554)
            (12, 13.5683)
            (13, 15.1121)
            (14, 16.6831)
            (15, 18.2779)
            (16, 19.8939)
            (17, 21.5287)
            (18, 23.1805)
            (19, 24.8477)
            (20, 26.5287)
            (21, 28.2224)
            (22, 29.9277)
            (23, 31.6437)
            (24, 33.3694)
            (25, 35.1043)
            (26, 36.8476)
            (27, 38.5987)
            (28, 40.3572)
            (29, 42.1224)
            (30, 43.8941)
            (31, 45.6719)
            (32, 47.4552)
            (33, 49.244)
            (34, 51.0377)
            (35, 52.8363)
            (36, 54.6394)
            (37,56.4468)
            (38, 58.2582)
            (39, 60.0736)
            (40, 61.8926)
            (41, 63.7152)
            (42, 65.5412)
            (43, 67.3704)
            (44, 69.2027)
            (45, 71.038)
            (46,72.8761)
            (47, 74.7169)
            (48, 76.5604)
            (49, 78.4064)
            (50, 80.2548)
        };
    \end{semilogyaxis}
\end{tikzpicture}
    \caption{The value  $\frac{\alpha}{\hat{k}_{0,1}(\varepsilon)}\left(\frac{4 \pi \varepsilon}{p_{d/2,1}}\right)^2$ ($y$-axis, log-scale) vs the dimension $d$ ($x$-axis). The constraint \eqref{thm:existence of maternised convolution kernel that satisfy slsi:eq2} is only satisfied for $d=1$.}
    \label{fig:failure_of_matern_kernel_in_higher_dim}
\end{figure}

\begin{lem}[sufficient conditions on $\hat{k}$ and $q$]
	\label{thm:stein_log_sobolev_inequality_via_poincaré_wirtinger}
    Let $k$ and $q$ be as in Lemma \ref{lem:passing_to_fourier}. Furthermore, assume that $k$ and $q$ are radially symmetric and that there exists constants $\alpha,\beta,\varepsilon,\theta>0$ such that
	\begin{align*}
		q(r)\geq\begin{cases}
			-\alpha& \text{ if }r<\varepsilon,\\
			\phantom{-}\beta & \text{ if }r\geq \varepsilon
		\end{cases}
            \quad\mbox{for } r\in[0,\infty)
             \qquad\text{and}\qquad
            \hat{k}(r)\geq \theta \quad
		  \mbox{for } r\in[0,\varepsilon].
	\end{align*}
	If the following strict inequality holds
    \begin{align}
        \label{thm:bounded kernel frequency with step function weight:eq1}
		\frac{\alpha}{\theta}\left(\frac{\varepsilon}{p_{d/2,1}}\right)^2 < \frac{1}{16\, \pi^2},
    \end{align}
    where $p_{d/2,1}$ is defined in Lemma \ref{lem:poincare-wirtinger-optimal}, then it holds
    \begin{align}
        \label{thm:bounded kernel frequency with step function weight:eq3}
        \int_{\Rd} |\hat{g}|^2 q\diff\xi+\frac{1}{16\pi^2}\int_{\Rd} |\nabla \hat{g}|^2 \hat{k}\diff\xi
        \geq \nu \int_{\Rd} |g|^2\diff\xi,
    \end{align}
    where $g$ is defined in \eqref{eq:def_g}, $\pwc\in\R$ is defined in \eqref{eq:def_tau}, $\rho$ satisfies \eqref{eq:regularity_to_be_satisfied_ineqality_rho} and $V$ satisfies \eqref{eq:general_prop_target_potentialV_f}. Moreover, the parameter $\nu$ is given by
    	\begin{align}
		\label{thm:bounded kernel frequency with step function weight:eq2}
		\nu=\min\left\{\beta,\; \frac{\theta}{16\, \pi^2}\left(\frac{p_{d/2,1}}{\varepsilon}\right)^2-\alpha\right\}.
	\end{align}
\end{lem}
\begin{proof}
	Consider the ball $B_{\varepsilon} = \{\xi: |\xi| \leq \varepsilon\}$. We can estimate
$$
	\int_{\R^d \setminus B_{\varepsilon}}|\hat g(\xi)|^2 q(\xi)\diff \xi
		+\int_{B_{\varepsilon}}|\hat g(\xi)|^2 q(\xi)\diff \xi
		\geq \, \beta\int_{\R^d \setminus B_{\varepsilon}}|\hat g(\xi)|^2\diff \xi
		-\alpha\int_{B_{\varepsilon}}|\hat g(\xi)|^2 \diff \xi.
$$	
As in the proof of Theorem \ref{thm:existence_of_maternised_convolution_kernel_that_satisfy_slsi} in 1D, the negative value can be compensated by the information on the gradient. We choose $\pwc\in\R$ such that $\int_{B_{\varepsilon}} \hat{g} \diff \xi = 0$ (see \eqref{eq:def_tau}) so that by the Poincaré-Wirtinger inequality in Lemma \ref{lem:poincare-wirtinger-optimal} we get
	\begin{align*}
		\int_{B_{\varepsilon}}|\hat g(\xi)|^2\diff \xi
		\leq \left(\frac{\varepsilon}{p_{d/2,1}}\right)^2 \,\int_{B_{\varepsilon}}|\nabla \hat g(\xi)|^2\diff \xi \leq  \frac{1}{\theta }\left(\frac{\varepsilon}{p_{d/2,1}}\right)^2 \, \int_{\Rd}|\nabla \hat g(\xi)|^2 \hat{k}(\xi)\diff \xi.
	\end{align*}
    Hence, to conclude the lower bound \eqref{thm:bounded kernel frequency with step function weight:eq3} we see that \eqref{thm:bounded kernel frequency with step function weight:eq1} needs to hold. From this, we deduce that the optimal value $\nu$ can be expressed as the minimum \eqref{thm:bounded kernel frequency with step function weight:eq2}.
\end{proof}

\subsection{Reconstructing kernels from weights}
\label{subsection:reconstructing kernels from weights}
For any given radially symmetric weight $q$, we explicitly reconstruct the kernel frequency $\hat{k}$ from the PDE stated in \eqref{eq:def_g}. We also study the differentiability of $
\hat{k}$.

\begin{lem}[formula for the kernel frequency $\hat{k}$]
    \label{Lemma:bounded kernel frequency near origin}
    Let $q\in L_{\mathrm{loc}}^1(0,\infty)$ have a representative, denoted by the same symbol, that is differentiable at $r=0$. Then, there exists the unique radially symmetric kernel frequency $\hat{k}\in W_{\mathrm{loc}}^{1,1}\cap L_{\loc}^\infty([0,\infty))$ that solves the ODE
    \begin{align*}
        \left(4\pi^2r^2-\frac{d}{2}\right) \hat{k}(r)-\frac{r}{2}\,\hat{k}'(r)
        =q(r).
    \end{align*}
    It is explicitly given by
    \begin{equation} 
\label{eq:kernel_freq_bounded_at_origin_formula}
\begin{split}
        \hat{k}(r)
        &=-\frac{2}{r^d}e^{4\pi^2 r^2}\int_0^r q(s)s^{d-1}e^{-4\pi^2 s^2}\diff s =-2 \int_0^1 q(u r) u^{d-1}e^{4\pi^2 r^2 (1-u^2)}\diff u\\
        &= -\frac{1}{(2\pi r)^d}e^{4\pi^2 r^2} \int_0^{4\pi^2 r^2} q\left( \frac{\sqrt{u}}{2\pi}\right)\, u^{\frac{d}{2}-1}e^{-u}\diff u.
   \end{split}
    \end{equation}
	Furthermore, its derivative reads
	\begin{align}
            \label{Lemma:bounded kernel frequency near origin:eq1}
		\hat{k}'(r)
		=-2\left(8\pi^2 r-\frac{d}{r}\right)\frac{e^{4\pi^2 r^2}}{r^d}\int_0^r q(s)s^{d-1}e^{-4\pi^2 s^2}\diff s
        -\frac{2}{r}q(r)
	\end{align}
    and it is bounded at the origin.
\end{lem}
\begin{proof}
    Suppose we know the solution $\hat{k}$ satisfies  $\hat{k}(\delta)=\hat{k}_\delta$ for some $\delta$, $\hat{k}_\delta>0$. Then by standard linear ODE methods we can explicitly solve for $\hat{k}$ by 
    \begin{align*}
        \hat{k}(r)
        =\frac{e^{4\pi^2 r^2}}{r^d}\left(\delta^d e^{-4\pi^2\delta^2}\hat{k}_\delta-2\int_\delta^r q(s)s^{d-1}e^{-4\pi^2 s^2}\diff s\right).
    \end{align*}
    In order for $\hat{k}$ to be bounded near the origin, the term inside the brackets must converge to zero as $r\ra 0$. In particular, taking the limit inside the brackets yields
    \begin{align*}
        \delta^d e^{-4\pi^2\delta^2}\hat{k}_\delta+2\int_0^\delta q(s)s^{d-1}e^{-4\pi^2 s^2}\diff s=0.
    \end{align*}
    As $\delta$ is arbitrary, we derive a formula for $\hat{k}$ by solving for $\hat{k}_\delta$ in the previous equation and substituting $\delta$ with any $r>0$. This yields the unique solution
    \begin{align*}
        \hat{k}(r)
        =-\frac{2}{r^d}e^{4\pi^2 r^2}\int_0^r q(s)s^{d-1}e^{-4\pi^2 s^2}\diff s.
    \end{align*}
    Then, the second formula in \eqref{eq:kernel_freq_bounded_at_origin_formula} follows by the change of variables $u=s/r$. The last expression in \eqref{eq:kernel_freq_bounded_at_origin_formula} follows by the change of variables $u = 4\pi^2 s^2$ so that $\diff u = 8 \pi^2 s \diff s$ and
    \begin{align*}
    \int_0^r q(s)s^{d-1}e^{-4\pi^2 s^2}\diff s &= \frac{1}{8\pi^2} \int_0^{4\pi^2 r^2} q\left( \frac{\sqrt{u}}{2\pi}\right)\, \left( \frac{\sqrt{u}}{2\pi}\right)^{d-2}e^{-u}\diff u \\
    &= \frac{1}{2\, (2\pi)^{d}}\int_0^{4\pi^2 r^2} q\left( \frac{\sqrt{u}}{2\pi}\right)\, u^{\frac{d}{2}-1}e^{-u}\diff u.
    \end{align*}
    We proceed to further properties of $\hat{k}$. Since $q$ is continuous at the origin, $\hat{k}$ is bounded at the origin by
    \begin{align*}
        \limsup_{r\ra 0}|\hat{k}(r)|
        \leq\limsup_{r\ra 0}\frac{2}{r}\int_0^r |q(s)|\left(\frac{s}{r}\right)^{d-1}\diff s
        \leq\limsup_{r\ra 0}\frac{2}{r}\int_0^r |q(s)|\diff s
        <\infty.
    \end{align*}
    Let us now show the differentiability of $\hat{k}$. From the formula \eqref{Lemma:bounded kernel frequency near origin:eq1}, we see that the only difficulty is at $r = 0$. 
    Since $q$ is differentiable at the origin, there exists a function $h:~[0,\infty)\ra\R$ such that
    \begin{align*}
        q(r)
        =q(0)+r\, q'(0)+h(r),\qquad
        h(r)=o(r)
        \text{ whenever }r\ra 0.
    \end{align*}
    Directly differentiating and using the above Taylor expansion we obtain
    \begin{align}
            \hat{k}'&(r)= -2\left(8\pi^2 r-\frac{d}{r}\right)\frac{e^{4\pi^2  r^2}}{r^d}\int_0^r q(s) s^{d-1}e^{-4\pi^2 s^2}\diff s
            -\frac{2}{r}q(r)\notag\\
            &=\frac{q(0)}{r}\left(d\frac{e^{4 \pi ^2 r^2}}{(2\pi r)^d}\,\gamma \left(\frac{d}{2},4 \pi ^2 r^2\right)-2\right)-2 q'(0)\label{Lemma:bounded_kernel_frequency_near_origin:eq2}\\
            &\qquad-2\left(8\pi^2 r-\frac{d}{r}\right)\frac{e^{4\pi^2 r^2}}{r^d}\int_0^r h(s) s^{d-1}e^{-4\pi^2 s^2}\diff s-\frac{2}{r}h(r)\label{Lemma:bounded_kernel_frequency_near_origin:eq3}\\
            &\qquad-8 \pi ^2 r\frac{e^{4 \pi ^2 r^2}}{(2\pi r)^d}q(0)\,\gamma \left(\frac{d}{2},4 \pi ^2 r^2\right)
            -\left(8 \pi ^2 r-\frac{d}{r}\right)\frac{e^{4 \pi ^2 r^2}}{(2\pi r)^d}\frac{q'(0)}{2\pi }\gamma \left(\frac{d+1}{2},4 \pi ^2 r^2\right),\label{Lemma:bounded_kernel_frequency_near_origin:eq4}
    \end{align}
where $\gamma$ is the incomplete gamma function defined in \eqref{eq:def_incomplete_gamma_lower}. To control the singularities and the unknown terms, we notice for all $a,z>0$ that
\begin{equation}\label{eq:lower_and_upper_bound_lower_gamma_incomplete}     
\frac{1}{a}z^a e^{-z} = e^{-z}\int_0^z s^{a-1}\diff s \leq \gamma(a,z) \leq \int_0^z s^{a-1}\diff s= \frac{1}{a}z^a.
\end{equation}
Hence, we can estimate the singular terms in \eqref{Lemma:bounded_kernel_frequency_near_origin:eq2} as follows, keeping in mind $q(0)<0$,
$$
    0
    =\frac{q(0)}{r}\left(2-2\right)
    \geq 
    \frac{q(0)}{r}\left(d\,\frac{e^{4 \pi ^2 r^2}}{(2\pi  r)^d} \gamma \left(\frac{d}{2},4 \pi ^2 r^2\right)-2\right)
    \geq 2q(0)\,\frac{e^{4 \pi ^2 r^2}-1}{r}.
$$
The other terms in \eqref{Lemma:bounded_kernel_frequency_near_origin:eq3} are easy to control because \eqref{eq:lower_and_upper_bound_lower_gamma_incomplete} implies that the following functions
\begin{equation}\label{eq:proof_functions_which_are_bounded}
r\mapsto\frac{1}{(2\pi r)^d} \, \gamma \left(\frac{d}{2},4 \pi ^2 r^2\right) \qquad \mbox{and} \qquad r\mapsto\frac{1}{(2\pi r)^{d+1}} \,\gamma \left(\frac{d+1}{2},4 \pi ^2 r^2\right)
\end{equation}
converge when $r \to 0$. Finally, the terms with function $h$ in \eqref{Lemma:bounded_kernel_frequency_near_origin:eq4} can be controlled by using the fact $h(r)=o(r)$~whenever $r\to 0$. More precisely, for all $\eta>0$, there exists $\varepsilon>0$ such that $|h(r)|<\eta\, r$ for all $r\in [0,\varepsilon]$. Then we estimate
    \begin{align*}
        \left\vert \frac{1}{r^{d+1}} \int_0^r h(s) s^{d-1}e^{-4\pi^2 s^2}\diff s\right\vert
        \leq \frac{\eta}{2\,(2\pi r)^{d+1}}\, \gamma\left(\frac{d+1}{2},4\pi^2 r^2\right),
    \end{align*}
and thus we conclude using \eqref{eq:proof_functions_which_are_bounded} again and the fact that $\eta>0$ is arbitrarily small.
\end{proof}

\subsection{Proof of the existence of kernel in higher dimensions}\label{sect:exist_kernel_any_dimensio}

In what follows we assume the simplest form of the weight $q$ being piecewise constant and we prove that it leads to a kernel satisfying the Stein-log-Sobolev inequality. The construction will extensively use the notions of the incomplete gamma functions $\gamma(s,r)$ and $\Gamma(s,r)$, summarized in Appendix~\ref{appendix:incomplete_gamma_functions}. It is worth comparing the Fourier transform of the kernel constructed in this section with the one from Section~\ref{sect:proof_lSS_ineq_1D}, see Figure~\ref{fig:development_of_knicks}. 

\begin{proof}[Proof of Theorem \ref{thm:existence_of_maternised_convolution_kernel_that_satisfy_slsi} (Case 2)]
    We split the proof into six steps.

    \smallskip
    
    \underline{\it A candidate for function $\hat{k}_{0,d}$.} Let $\alpha,\varepsilon>0$, $r\geq 0$ and define the weight
\begin{equation} \label{eq:step_function_q}
q(r)=\begin{cases}
            -\alpha & \text{ if } r\leq\varepsilon,\\
            \alpha\beta & \text{ if }r>\varepsilon,
        \end{cases}
        \qquad\beta:=\frac{\gamma \left(\frac{d}{2},4 \pi ^2 \varepsilon ^2\right)}{\Gamma \left(\frac{d}{2},4 \pi ^2 \varepsilon ^2\right)}.
\end{equation}
    We have chosen $\beta$ such that
    \begin{equation}\label{eq:final_proof_0_integral_condition_two_weights}
    \int_0^{\infty} q\left( \frac{\sqrt{u}}{2\pi}\right)\, u^{\frac{d}{2}-1}e^{-u}\diff u = 0.
    \end{equation}
    This condition, together with the form of $q$ in \eqref{eq:step_function_q}, guarantees that $\hat{k}_{0,d}$ given by \eqref{eq:kernel_freq_bounded_at_origin_formula} satisfies $\hat{k}_{0,d} > 0$ (so that $k$ is formally positive definite by Bochner's Theorem \cite{MR152834}). Using \eqref{eq:final_proof_0_integral_condition_two_weights} and the third formula in \eqref{eq:kernel_freq_bounded_at_origin_formula} we can compute $\hat{k}_{0,d}$ explicitly as
    \begin{equation}\label{eq:explicit_formula_k_two_weights}
        \hat{k}_{0,d}(r)=\begin{cases}
            \displaystyle\alpha\frac{\exp(4\pi^2 r^2)}{(2\pi r)^d}\gamma\left(\frac{d}{2},4\pi^2 r^2\right) & \text{ if }r\leq \varepsilon,\\[.5cm]
            \displaystyle\alpha\beta\frac{\exp(4\pi^2 r^2)}{(2\pi r)^d}\Gamma\left(\frac{d}{2},4\pi^2 r^2\right) & \text{ if }r> \varepsilon.
        \end{cases}
    \end{equation}
    \underline{\it Monotonicity of $\hat{k}_{0,d}$.} We claim that the function $\hat{k}_{0,d}$ is increasing on $[0,\eps]$ and decreasing on $[\eps, \infty)$ with $\hat{k}_{0,d}(0) = \frac{2\alpha}{d}$ and $\lim_{r \to \infty} \hat{k}_{0,d}(r) = 0$. Indeed, using the second formula in \eqref{eq:kernel_freq_bounded_at_origin_formula} we have
    \begin{align*}
        \lim_{r\to 0} \hat{k}_{0,d}(r)
         = 2\alpha   \int_0^1 u^{d-1} \diff u
        =\frac{2\alpha }{d}.
    \end{align*} 
    All other assertions follow directly from Lemma \ref{Lemma:kernel frequency of constant is monotone increasing} and formula \eqref{eq:explicit_formula_k_two_weights}.

    \smallskip
    
\underline{\it The Fourier inversion and the regularity of $k_{0,d}$.} We now prove that there exists function $k_{0,d} \in L^1(\Rd)+L^2(\Rd)$ with Fourier transform $\hat{k}_{0,d}$. 

\smallskip

Fix $m \in \N$ and write $\hat{k}_{0,d} = \hat{k}_{0,d}\, \mathds{1}_{4\pi^2 r^2 \leq 1} + \hat{k}_{0,d}\, \mathds{1}_{4\pi^2 r^2 > 1}$. Concerning the term $\hat{k}_{0,d}\, \mathds{1}_{4\pi^2 r^2 > 1}$, using Lemma \ref{lem:expansion_into_matern_kernels} and the explicit formula \eqref{eq:explicit_formula_k_two_weights}, we can write for $n \in \N$ to be chosen later, constants $C_0$, ..., $C_n\in\R$ and an error function $\varepsilon_n:[0,\infty)\ra\R$  
\begin{align*}
\hat{k}_{0,d}(r) &= \sum_{i=1}^n \frac{C_i}{(1+ 4\pi^2 r^2)^i} \, \mathds{1}_{4\pi^2 r^2 > 1} + \varepsilon_n(r)\, \mathds{1}_{4\pi^2 r^2 > 1} + \hat{k}_{0,d}\, \mathds{1}_{4\pi^2 r^2 \leq 1}\\
&= \sum_{i=1}^n \frac{C_i}{(1+ 4\pi^2 r^2)^i} - \sum_{i=1}^n \frac{C_i}{(1+ 4\pi^2 r^2)^i} \, \mathds{1}_{4\pi^2 r^2 \leq 1} + \varepsilon_n(r)\, \mathds{1}_{4\pi^2 r^2 > 1} + \hat{k}_{0,d}\, \mathds{1}_{4\pi^2 r^2 \leq 1},
\end{align*}
where $|\varepsilon_n(r)|\leq C_{0} r^{-(n+1)}$. We let
$$
\widehat{k}_1(r) := \sum_{i=1}^n \frac{C_i}{(1+ 4\pi^2 r^2)^i}, \quad \widehat{k}_2(r):= \left(\hat{k}_{0,d}-\sum_{i=1}^n \frac{C_i}{(1+ 4\pi^2 r^2)^i}\right) \, \mathds{1}_{4\pi^2 r^2 \leq 1} + \varepsilon_n(r)\, \mathds{1}_{4\pi^2 r^2 > 1}.
$$
There exists a function $k_1$ with Fourier transform $\widehat{k}_1$, namely
\begin{equation}\label{eq:representation_k_1_main_proof}
k_1 = \sum_{i=1}^n C_i\, \underbrace{\spk \ast \ldots \ast \spk}_{i \text{ times}},
\end{equation}
where $\spk$ is the solution of \eqref{eq:screened_poisson_equation} in Lemma \ref{lem:fundamental_solution_to_screened_poisson_equation}. Since the Fourier transform is a bijection on the space of tempered distributions, the function $k_1$ is uniquely determined. By the Young convolutional inequality each summand in \eqref{eq:representation_k_1_main_proof}, and consequently also $k_1$, shares the same regularity as $\spk$. In particular, we have proven $k_1\in L^1(\Rd)$.

\smallskip

Concerning $\widehat{k}_2$, we choose $n$ sufficiently large so that 
\begin{equation}\label{eq:choice_m_construction_k2}
\int_{4\pi^2 |\xi|^2>1} |\varepsilon_n(|\xi|)|^2 \,(1+|\xi|^2)^m \diff \xi < \infty.
\end{equation}  
First, since the Fourier transform is an isometric isomorphism on $L^2(\Rd)$ and $\widehat{k}_2 \in L^2(\Rd)$, there exists a unique function $k_2$ with Fourier transform $\widehat{k}_2$. Moreover, by \eqref{eq:choice_m_construction_k2},
$$
\int_{\Rd} |\widehat{k}_2(\xi)|^2 \, (1+|\xi|^2)^m \diff \xi < \infty,
$$
so that $k_2 \in H^m(\Rd)$, and hence $k_2\in L^2(\Rd)$, as desired. 

\underline{\it Kernel satisfies upper- and lower bounds in  \eqref{eq:condition_Fourier_transform_k}.} By using Corollary \ref{cor:the_asymptotic_beh_Gamma_fcn} we have
$$
\lim_{r\to \infty} \alpha\beta\frac{\exp(4\pi^2 r^2)}{(2\pi r)^{d-2}}\Gamma\left(\frac{d}{2},4\pi^2 r^2\right) = \alpha \beta >0.
$$
Hence, there exists large $R > \max(1,\eps)$ such that for $r \in [R, \infty)$ we have
$$
\frac{\alpha \beta}{2} \leq \alpha\beta\frac{\exp(4\pi^2 r^2)}{(2\pi r)^{d-2}}\Gamma\left(\frac{d}{2},4\pi^2 r^2\right) \leq 2 \alpha \beta. 
$$
In particular, 
\begin{equation}\label{eq:lower_and_upper_bound_k_large_interval}
\frac{\alpha \beta}{2\,(1+r^2)} \leq \frac{\alpha \beta}{2\,r^2} \leq \alpha\beta\frac{\exp(4\pi^2 r^2)}{(2\pi r)^{d-2}}\Gamma\left(\frac{d}{2},4\pi^2 r^2\right)\frac{1}{r^2} \leq \frac{2 \alpha \beta}{r^2} \leq \frac{4 \alpha \beta}{(1+r^2)}, 
\end{equation}
which proves the bounds in \eqref{eq:condition_Fourier_transform_k} for $r \in [R, \infty)$. To conclude, it suffices to notice that since functions $\frac{1}{1+r^2}$ and $\hat{k}_{0,d}(r)$ are bounded away from 0 on $[0,R]$ (by Step 2), there is always a constant $C>0$ depending on $R$ such that
\begin{equation}\label{eq:lower_and_upper_bound_k_small_interval}
\frac{1}{C}\, \frac{1}{1+r^2} \leq \hat{k}_{0,d}(r) \leq C\, \frac{1}{1+r^2} \mbox{ for } r\in [0,R].
\end{equation}
Combining \eqref{eq:lower_and_upper_bound_k_large_interval} and \eqref{eq:lower_and_upper_bound_k_small_interval}, the proof of the upper- and lower bound in \eqref{eq:condition_Fourier_transform_k} is concluded.

\smallskip

\underline{\it Boundedness of $\hat{k}_{0,d}'$.} By Lemma \ref{Lemma:bounded kernel frequency near origin}, we know that $\hat{k}_{0,d}'(r)$ is locally bounded and it remains to prove that it remains bounded when $r \to \infty$. In fact, we will prove $\lim_{r\to \infty} \hat{k}_{0,d}'(r) = 0$. Indeed, by a direct computation for $r>\varepsilon$ and the chain rule
\begin{equation}\label{eq:derivative_hat_k_radial_particular_case}
\hat{k}_{0,d}'(r) = \alpha \, \beta\, \partial_s \left[e^s s^{-\frac{d}{2}} \Gamma\left(\frac{d}{2}, s\right)\right] \Bigg|_{s = 4\pi r^2} \, (8 \pi r). 
\end{equation}
Lemma \ref{lem:Gamma_asymptotics_derivative} implies that the RHS of \eqref{eq:derivative_hat_k_radial_particular_case} converges to 0 when $r \to \infty$. 

\smallskip

\underline{\it Proof of the inequality \eqref{eq:inequality_l2_lower_bound_formulated_H-1_H1}.}
    We know that $k_{0,d} \in L^1(\Rd)+L^2(\Rd)$, $\hat{k}_{0,d}$ satisfies \eqref{eq:condition_Fourier_transform_k}. Hence, all assumptions of Lemmas \ref{lem:passing_to_fourier} and \ref{thm:stein_log_sobolev_inequality_via_poincaré_wirtinger} are satisfied. By choosing $\varepsilon$ small enough, we see that \eqref{thm:bounded kernel frequency with step function weight:eq1} in Lemma~\ref{thm:stein_log_sobolev_inequality_via_poincaré_wirtinger} holds. Hence, there exists $\lambda_{0,d}$ such that 
     \begin{equation}\label{eq:l^2_lower_bound_final_proof}
        \int_{\Rd} |\hat{g}|^2 q\diff\xi+\frac{1}{16\pi^2}\int_{\Rd} |\nabla \hat{g}|^2 \hat{k}_{0,d}\diff\xi
        \geq \lambda_{0,d} \int_{\Rd} |g|^2\diff\xi
    \end{equation}
    for $g$ as in Lemma~\ref{thm:stein_log_sobolev_inequality_via_poincaré_wirtinger}. In this case the value $\lambda_{0,d}$ is given by
    \begin{align}\label{eq:formula_for_lambda_0}
        \lambda_{0,d}= \alpha\min\left\{\frac{\gamma \left(\frac{d}{2},4 \pi ^2 \epsilon ^2\right)}{\Gamma \left(\frac{d}{2},4 \pi ^2 \epsilon ^2\right)},\,\frac{1}{8\pi^2 d}\left(\frac{p_{d/2,1}}{\varepsilon}\right)^2-1\right\}.
    \end{align}
    Thus, we can optimize the coefficient $\lambda_{0,d}$, by choosing $\varepsilon$ such that
    \begin{align*}
    	\varepsilon:=\argmax_{\delta\in\left(0,\sqrt{\frac{2}{d}}\,\frac{p_{d/2,1}}{2\pi}\right)}\;\min\left\{\frac{\gamma \left(\frac{d}{2},4 \pi ^2 \delta ^2\right)}{\Gamma \left(\frac{d}{2},4 \pi ^2 \delta ^2\right)},\,\frac{1}{8\pi^2 d}\left(\frac{p_{d/2,1}}{\delta}\right)^2-1\right\}
    \end{align*}
    
    We conclude by following the same steps as in the proof of Theorem \ref{thm:existence_of_maternised_convolution_kernel_that_satisfy_slsi} in 1D, in Section \ref{sect:proof_lSS_ineq_1D}.
\end{proof}

\begin{rem}
It may be worth commenting optimality of the choice of the weight $q$ in \eqref{eq:step_function_q}. It is an interesting open problem to understand if introducing more complicated shapes of $q$ (for example, oscillations) leads to better regularity properties of the kernel $k$. We leave this as an open question.
\end{rem}

\begin{figure}
    \centering
    \begin{tikzpicture}
    \begin{axis}[
        title={$\alpha=1$ and $\varepsilon=0.1$},
        legend pos=north east,
        width=0.49\linewidth,
        height=5cm,
        legend style={font=\tiny},
        scaled x ticks=false,
        xticklabel style={
            /pgf/number format/fixed,
            /pgf/number format/precision=3
        },
        xmin=0, xmax=0.2
    ]
    \addplot[black, solid] coordinates {(0, 0.500494) (0.01, 0.500494)  (0.02, 0.501979)  (0.03, 0.504468)  (0.04, 0.507979)  (0.05, 0.512542)  (0.06, 0.518194)  (0.07, 0.524979)  (0.08, 0.532956)  (0.09, 0.54219)  (0.1, 0.552762)  (0.11, 0.456828)  (0.12, 0.383862)  (0.13, 0.327078)  (0.14, 0.282021)  (0.15, 0.245672)  (0.16, 0.215923)  (0.17, 0.191267)  (0.18, 0.170605)  (0.19, 0.15312)  (0.2, 0.13819)};
    \addlegendentry{$d=2$}
    \addplot[black, dashed] coordinates {(0, 0.200113) (0.01, 0.200113) (0.02, 0.200452) (0.03, 0.201019) (0.04, 0.201817) (0.05, 0.202851) (0.06, 0.204126) (0.07, 0.205648) (0.08, 0.207426) (0.09, 0.209471) (0.1, 0.211793) (0.11, 0.136721) (0.12, 0.0922714) (0.13, 0.0646591) (0.14, 0.0467927) (0.15, 0.0348202) (0.16,0.0265495) (0.17, 0.0206811) (0.18, 0.0164179) (0.19, 0.0132548) (0.2, 0.0108634)};
    \addlegendentry{$d=5$}
    \addplot[black, dash dot] coordinates {(0, 0.125064) (0.01, 0.125064) (0.02, 0.125198) (0.03, 0.125445) (0.04, 0.125794) (0.05, 0.126244) (0.06, 0.126798) (0.07, 0.127458) (0.08, 0.128226) (0.09, 0.129106) (0.1, 0.130102) (0.11, 0.0632586) (0.12, 0.0329976) (0.13, 0.01827) (0.14, 0.0106483) (0.15, 0.00648992) (0.16, 0.00411423) (0.17, 0.00270095) (0.18, 0.00182947) (0.19, 0.00127455) (0.2, 0.000910872)};
    \addlegendentry{$d=8$}
    \end{axis}
\end{tikzpicture}
\begin{tikzpicture}
    \begin{axis}[
        title={$\alpha=1$ and $\varepsilon=0.05$},
        legend pos=north east,
        width=0.49\linewidth,
        height=5cm,
        legend style={font=\tiny},
        scaled x ticks=false,
        xticklabel style={
            /pgf/number format/fixed,
            /pgf/number format/precision=3
        },
        xmin=0, xmax=0.2
    ]
    \addplot[black, solid] coordinates {(0, 1.00198) (0.01, 1.00198) (0.02, 1.00794) (0.03, 1.01798) (0.04, 1.03226) (0.05, 1.05101) (0.06, 0.72987) (0.07, 0.536231) (0.08, 0.410552) (0.09, 0.324387) (0.1, 0.262753) (0.11, 0.217151) (0.12, 0.182467) (0.13, 0.155475) (0.14, 0.134058) (0.15, 0.116779) (0.16, 0.102638) (0.17, 0.090918) (0.18, 0.0810966) (0.19, 0.0727848) (0.2, 0.0656883)};
    \addlegendentry{$d=2$}
    \addplot[black, dashed] coordinates {(0, 0.400452) (0.02, 0.401811) (0.03, 0.404093) (0.04, 0.407321) (0.05, 0.411531) (0.06, 0.172516) (0.07, 0.0838312) (0.08, 0.0454488) (0.09, 0.0268137) (0.1, 0.0169207) (0.11, 0.011279) (0.12, 0.00786734) (0.13, 0.00570055) (0.14, 0.00426604) (0.15, 0.00328189) (0.16, 0.00258553) (0.17, 0.00207934) (0.18, 0.00170256) (0.19, 0.00141613) (0.2, 0.0011943)};
    \addlegendentry{$d=5$}
    \addplot[black, dash dot] coordinates {(0, 0.250196) (0.01, 0.250196) (0.02, 0.250792) (0.03, 0.251787) (0.04, 0.253192) (0.05, 0.255017) (0.06, 0.0619404) (0.07, 0.0189965) (0.08, 0.00692496) (0.09, 0.00288574) (0.1, 0.00133844) (0.11, 0.000677858) (0.12, 0.000369576) (0.13, 0.000214563) (0.14, 0.000131509) (0.15, 0.0000845032) (0.16, 0.000056599) (0.17, 0.0000393234) (0.18, 0.0000282225) (0.19, 0.0000208486) (0.2, 0.0000158025)};
    \addlegendentry{$d=8$}
    \end{axis}
\end{tikzpicture}
    \caption{The shape of the kernel frequency $\hat{k}_{0,d}$ with $\alpha=1$ and $\varepsilon=0.05, 0.1$ for different dimensions $d$. On the $x$ axis we have the radius $r$, and on the $y$ axis we have the value $\hat{k}_{0,d}(r)$. Contrary to the kernel which satisfies \eqref{eq:inequality_l2_lower_bound_formulated_H-1_H1} in one dimension $\hat{k}_{0,d}(\xi) = \frac{1}{1+C\,|\xi|^2}$, the ones constructed for higher dimensions first increase and then decrease to 0 at infinity.}
    \label{fig:development_of_knicks}
\end{figure}
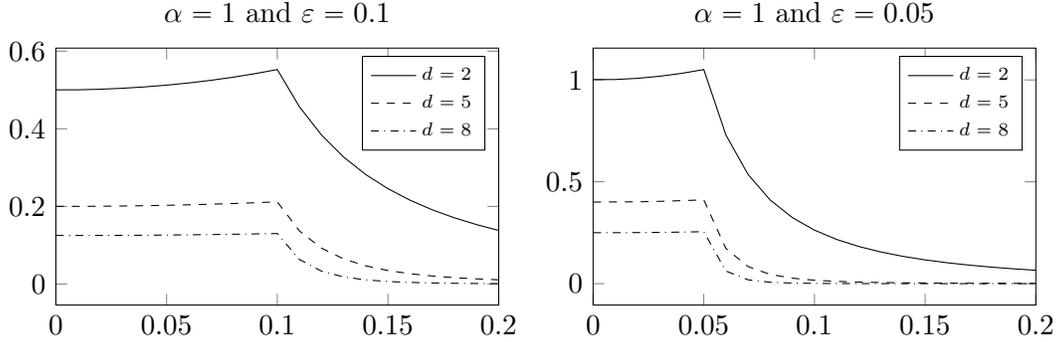

\section{Weak Solutions to the Stein Gradient Flow}
\label{section:weak solutions to the stein variational gradient descent equation}

\noindent In this section, we prove Theorem \ref{thm:existence_of_distributional_solutions_to_mfsvgd}, i.e. the existence of distributional solutions $\rho = \rho(t,x)$ to the PDE \eqref{eq:stein_gradient_descent_kullback_leibler} with initial value $\rho_0\in\Pro$, which we recall for convenience
\begin{align}\label{eq:wellposed_pde_mf_svgd}
     \begin{split}
     \prt_t \rho&=\divv\Big(\rho\; e^{V-\frac{\pot}{2}}\; k\ast \Big(\nabla \left( \rho\, e^{V} \right) e^{-\frac{\pot}{2}}\Big)\Big)\\
     \rho(0,x) &= \rho_0(x)
     \end{split}
     \begin{split}
        & \mbox{ on } (0,\infty) \times \Rd, \\
        &\mbox{ on }\Rd. 
 \end{split}
\end{align}
Let us explain the strategy of our proof. The starting point is to consider an approximation of the PDE \eqref{eq:wellposed_pde_mf_svgd}
\begin{align}
    \begin{split}
         \prt_t \rho=& \,\divv\Big(\rho\; \omega_\eta\ast\Big[e^{V_\eta-\frac{\pot}{2}}\; k\ast \Big(
\nabla ((\rho \ast \omega_\eta) e^{V_\eta}) \, e^{-\frac{\pot}{2}}
\Big)\Big]\Big) \\
&\,+\frac{1}{R}\,\divv\left(\nabla \rho+\rho\nabla V_\eta\right)
     \end{split}
\qquad &&\mbox{ on } (0,\infty) \times B_R, \label{eq:regularised_weak_pde} \\
 \rho=&\,0 &&\text{ on }(0,\infty)\times \prt B_R,\label{eq:Dirichlet_approximation_sys}\\
\rho(0,\cdot) = & \,\rho_0^{R,\eta}:=(\rho\, \mathds{1}_{B_{R-1}}) \ast\omega_\eta &&\text{ on } B_R,  \label{ass:mfsvgd:ini} 
\end{align}
where we regularise the potential and the target probability measure  
\begin{equation}\label{ass:mfsvgd:f}
V_\eta:= V\ast \omega_\eta, \qquad \rho_\infty^\eta=e^{-V_\eta}.
\end{equation}
The approximating problem uses the mollification kernel $\{\omega_\eta\}_{\eta\in (0,1)}$ with parameter $\eta$ defined by $\omega_\eta(x)=\frac{1}{\eta^d}\tilde\omega\left(\frac{x}{\eta}\right)$ for a smooth function $\tilde\omega$ such that 
        $$
        \tilde\omega>0\text{ on }B_1,
        \qquad 
        \tilde\omega=0\text{ on }\Rd\setminus B_1,
        \qquad 
        \int_{\Rd}\tilde\omega\diff x=1,
        \qquad 
        \tilde\omega(x)=\tilde\omega(|x|).
        $$
To make the convolutions in space well-defined, we always extend $\rho(t,x)=0$ for $x \notin B_R$. The existence of classical solutions $\rho=: \rho^{R,\eta}$ to \eqref{eq:regularised_weak_pde}--\eqref{ass:mfsvgd:ini} follows from the standard theory of second-order parabolic PDEs. For the sake of completeness, we present the proof in Appendix \ref{appendix:existence_approximations_former_appendix}.

\smallskip

In Section \ref{sect:uniform_estimates} we obtain uniform bounds with respect to $R$ and $\eta$ for on the following terms
\begin{align*}
    (\rho^{R,\eta}\ast\omega_\eta) e^{V_\eta-\frac{\pot}{2}},\qquad
    \nabla( (\rho^{R,\eta}\ast\omega_\eta) e^{V_\eta}) e^{-\frac{\pot}{2}},\qquad
    \prt_t \rho^{R,\eta}.
\end{align*}
These estimates are a consequence of a selection of $\eta$ in terms of $R$, the fact that the kernel $k$ satisfies Theorem \ref{thm:existence_of_maternised_convolution_kernel_that_satisfy_slsi}, as well as the careful construction of the approximating problem \eqref{eq:regularised_weak_pde} so that the inequality \eqref{eq:inequality_l2_lower_bound_formulated_H-1_H1} can be applied. 

\smallskip

In Section \ref{sect:passage_limit_proof_main} we use these bounds to pass to the limit $R\to\infty$ and $\eta \to 0$ in \eqref{eq:regularised_weak_pde}--\eqref{ass:mfsvgd:ini}, hence proving the existence of weak solution to 
\eqref{eq:wellposed_pde_mf_svgd}.

\smallskip

We point out that the approximating scheme leading to the existence of a solution to \eqref{eq:wellposed_pde_mf_svgd} does not require us to normalize the initial condition in \eqref{ass:mfsvgd:ini}.  Furthermore, we observe that by~\eqref{eq:general_prop_target_potentialV_f} 
    \begin{equation}\label{eq:estimate_V_eta_from_below}
    \begin{split}
        \ngib\leq \int_{\Rd} \Big(V(x-y)-&\frac{1}{2}(x-y-\mu)\cdot \Sigma^{-1}(x-y-\mu)\Big)\;\omega_\eta(y)\diff y\\
        &=V_\eta(x)-\pot(x)-\frac{1}{2}\int_{\Rd} y\cdot \Sigma^{-1}y\;\omega_\eta(y)\diff y
        \leq V_\eta(x)-\pot(x),
    \end{split}
    \end{equation}
so that $V_{\eta}$ satisfies \eqref{eq:general_prop_target_potentialV_f} with the same constant $\ngib$ as $V$. Thus, we can employ the bound \eqref{eq:inequality_l2_lower_bound_formulated_H-1_H1} with the same constant $\lambda$ for all $R$ and $\eta$.

\subsection{Uniform estimates for \eqref{thm:existence_of_distributional_solutions_to_mfsvgd}}\label{sect:uniform_estimates}

In this section, we prove uniform estimates that are required for the proof of Theorem \ref{thm:existence_of_distributional_solutions_to_mfsvgd}. We start by selecting $\eta$ in terms of $R$. As $\nabla V$ is continuous, there exists $\eta=\eta(R)\in (0,\frac{1}{R})$ sufficiently small such that
\begin{equation}\label{eq:choice_eta_R}
\mbox{ for all } x,y\in B_{R+2} \mbox{ with }|x-y|<\eta \mbox{ it holds } |\nabla V(x)-\nabla V(y)|<\frac{1}{R}.
\end{equation}
We write
\begin{align*}
    \rho^R : = \rho^{R,\eta(R)},\qquad
    \rho_0^R := \rho_0^{R,\eta(R)},\qquad
    \rho_\infty^R := \rho_\infty^{\eta(R)}\qquad
    V_R:=V_{\eta(R)},\qquad
    \omega_R:=\omega_{\eta(R)}.
\end{align*}

\begin{lem}\label{lem:uniform_estimates}
    There exists $R_0\geq 2$ such that following sequences are uniformly bounded:
    \begin{enumerate}[label=(E\arabic*),itemsep=.1em]
        \item \label{lem:uniform_estimates:KL}
        $\{\KL(\rho_0^R|| \rho_\infty^R))\}_{R\geq R_0}$ in $\R$, 
        \item \label{lem:uniform_estimates:ini_H-1_est}
        $\{\rho_0^R\}_{R\geq R_0}$ in $H^{-m}(\Rd)$ for any fixed $m>\frac{d}{2}$, 
        \item \label{lem:uniform_estimates:L2_est}
        $\{(\rho^R\ast\omega_R) e^{V_R-\frac{\pot}{2}}\}_{R\geq R_0}$ in $L^2((0,T)\times \Rd)$ for all fixed $T>0$,
        \item \label{lem:uniform_estimates:H-1_est}
        $\{\nabla( (\rho^R\ast\omega_R) e^{V_R}) e^{-\frac{\pot}{2}}\}_{R\geq R_0}$ in $L^2(0,T; H^{-1}(\Rd))$ for all fixed $T>0$,
        \item \label{lem:uniform_estimates:KL_all_times} $\{\KL(\rho_T^R||\rho_\infty^R)\}_{R\geq R_0}$ for all fixed $T>0$,
        \item\label{lem:uniform_estimates:H1_est} $\left\{k\ast \left(\nabla ( (\rho^R\ast\omega_R)e^{V_R})e^{-\frac{\pot}{2}}\right)\right\}_{R\geq R_0}$ in $L^2(0,T; H^{1}(\Rd))$ for all fixed $T>0$,
        \item\label{lem:uniform_estimates:L^2_est_diff} $\left\{\frac{1}{\sqrt{R}}  \frac{|\nabla \rho^R+\rho^R\nabla V_R|}{\sqrt{\rho^R}}\right\}_{R\geq R_0}$ in $L^2((0,T)\times\Rd)$ for all fixed $T>0$,
        \item \label{lem:uniform_estimates:der_est}
        $\{\prt_t \rho^R\}_{R\geq R_0}$ in $(L^{2d}(0,T; H^m_0(B_L))\raisebox{3pt}{$\ast$}$ for any fixed $L>0$, $T>0$, $q= 3\vee d$ and $m>\frac{d}{2}$. 
    \end{enumerate}
Moreover,
\begin{equation}\label{eq:conv_KL_time_0}
\lim_{R\to \infty} \KL(\rho_0^R|| \rho_\infty^R) = \KL(\rho_0 || \rho_{\infty}).
\end{equation}
\end{lem}

\begin{proof}
We split the proof into several steps. For simplicity, in the proof of \ref{lem:uniform_estimates:L2_est}--\ref{lem:uniform_estimates:der_est}, we write $\rho$ for $\rho^R$.

\underline{\it Estimate \ref{lem:uniform_estimates:KL}.} Let $\Phi(\rho) := \rho \mapsto \rho \log(\frac{\rho}{\rho_\infty^R}) = \rho \log \rho + \rho \,V_R$. First, by \eqref{eq:pointwise_lower_bound} and \eqref{eq:estimate_V_eta_from_below}, 
    \begin{equation}\label{eq:estimate_negative_part_mollifier_Phi}
  - \frac{2}{e} \, e^{-\ngib/2} \, e^{-\pot(x)/2}  \leq - \frac{2}{e} \, e^{-V_R(x)/2} \leq \Phi(\rho_0^R),
    \end{equation}
    so the negative part of $\Phi(\rho_0^R)$ is bounded in $L^1(\Rd)$. Second, $\Phi$ is convex as a sum of convex and linear functions. Applying Jensen's inequality
    \begin{equation}\label{eq:estimate_positive_part_mollifier_Phi}
    \Phi(\rho_0^R)  \leq \Phi(\rho_0\mathds{1}_{B_{R-1}}) \ast \omega_{R} = (\Phi(\rho_0) \mathds{1}_{B_{R-1}}) \ast \omega_{R}.
    \end{equation}
    Combining the two facts we can integrate to get
    $$
    \int_{\Rd} \Phi(\rho_0^R) \diff x \leq 
    \int_{\Rd} \int_{B_{R-1}} \Phi(\rho_0(y)) \, \omega_{R}(x-y) \diff y \leq \int_{B_{R-1}} \Phi(\rho_0(y)) \diff y \leq \int_{\Rd} |\Phi(\rho_0(y))| \diff y.
    $$
    The Fubini theorem can be applied because, by $\KL(\rho_0|| \rho_{\infty}) < \infty$ and from Lemma~\ref{lem:control_neg_log} we deduce $\Phi(\rho_0) \in L^1(\Rd)$. This also allows to obtain \ref{lem:uniform_estimates:KL}.
    
    \underline{\it Convergence \eqref{eq:conv_KL_time_0}.} By \eqref{eq:estimate_negative_part_mollifier_Phi} and \eqref{eq:estimate_positive_part_mollifier_Phi} we know that  
    $$
    | \Phi(\rho_0^R)| \leq  \frac{2}{e} \, e^{-\ngib/2} \, e^{-\pot(x)/2} + | \Phi(\rho_0) | \ast \omega_{R},
    $$
    where the RHS is convergent in $L^1(\Rd)$ since $\Phi(\rho_0) \in L^1(\Rd)$. By a variant of Vitali convergence theorem \cite[Corollary A.1]{2024arXiv240802345C} we deduce that $\Phi(\rho_0^R) \to \Phi(\rho_0)$ in $L^1(\Rd)$, which implies in particular the claim.
    
    \underline{\it Estimate \ref{lem:uniform_estimates:ini_H-1_est}.}
    For all $\psi\in  W^{1,\infty}(\Rd)$ we compute 
    \begin{align*}
        \abs{\int_{\Rd} \psi\diff \rho_0^R}
        =\abs{\int_{B_R} \psi\ast\omega_R\diff \rho_0}
        \leq \norm{\psi\ast\omega_R}_{W^{1,\infty}}
        \leq \norm{\psi}_{W^{1,\infty}} \leq C\,\norm{\psi}_{H^m}
    \end{align*}
    for some constant $C>0$, where in the last step we used the Sobolev embedding theorem $H^m(\Rd)\hookrightarrow W^{1,\infty}(\Rd)$ and Young convolution inequality. Taking the supremum over $\norm{\psi}_{H^m}\leq 1$ proves the claim.

    \underline{\it Energy dissipation type inequality to prove \ref{lem:uniform_estimates:L2_est}--\ref{lem:uniform_estimates:der_est}.}
    First, we show that there exists a constant $C>~0$ not depending on $R$ such that the following estimate holds
    \begin{equation}
        \label{lem:uniform_estimates:step1:eq1}
        \begin{multlined}
            \int_0^T \DD^2(\rho_t\ast\omega_R||\rho_\infty^R)\diff t
            -\frac{C}{R}\Big(\norm{(\rho\ast \omega_R)e^{V_R-\frac{\pot}{2}}}_{L_{t,x}^2}^2
            +\norm{\nabla\big((\rho\ast \omega_R)e^{V_R}\big)e^{-\frac{\pot}{2}}}_{L_t^2 H_x^{-1}}^2\Big)\\
            +\frac{1}{R} \int_0^T \int_{B_R} \frac{|\nabla \rho_t+\rho_t\nabla V_R|^2}{\rho_t}\diff x\diff t
            \leq \KL(\rho_0^R||\rho_\infty^R)-\KL(\rho_T||\rho_\infty^R).
        \end{multlined}
    \end{equation}
    We multiply \eqref{eq:regularised_weak_pde} with $\ln(\rho_t \, e^{V_R})$, integrate over $\Rd$ and apply integration by parts with Lemma \ref{lem:chng_var} (passing $\kappa\ra 0$) to derive
	\begin{equation}
        \label{lem:uniform_estimates:step1:eq2}
		\begin{split}
			&\prt_t \KL(\rho_t||\rho_\infty^{\eta})+\frac{1}{R} \int_{\Rd} \frac{|\nabla \rho_t+\rho_t \nabla V_R|^2}{\rho_t}\diff x\\
			&+\int_{\Rd}  k\ast \left( \nabla \Big( (\rho_t\ast\omega_R)e^{V_R}\Big)e^{-\frac{\pot}{2}}\right)\cdot 
			\Big(\nabla (\rho_t\ast\omega_R)+\omega_R\ast \Big[\rho_t \nabla V_R\Big]\Big) e^{V_R-\frac{\pot}{2}}\diff x
			\leq 0.
		\end{split}
	\end{equation}
	To get the claim, we need to replace $\omega_R\ast \Big[\rho_t \nabla V_R\Big]$ with $(\omega_R\ast\rho_t) \nabla V_R$. The corresponding error equals
	\begin{align*}
		&\left\vert\int_{\Rd} k\ast\left( \nabla \Big( (\rho_t\ast\omega_R)e^{V_R}\Big)e^{-\frac{\pot}{2}}\right)\cdot \Big(\omega_R\ast\Big[\rho_t\nabla V_R\Big]-(\rho_t\ast\omega_R)\nabla V_R\Big)e^{V_R-\frac{\pot}{2}}\diff x\right\vert\\
		&\leq \int_{\Rd} \left\vert k\ast \left(\nabla \Big( (\rho_t\ast\omega_R)e^{V_R}\Big)e^{-\frac{\pot}{2}}\right)\right\vert
		 \left\vert\int_{\Rd} \rho_t(y) \, \omega_R(x-y)\Big( \nabla V_R(y)-\nabla V_R(x)\Big)\diff y \right\vert\,e^{V_R-\frac{\pot}{2}}\diff x.
    \end{align*}
    We estimate the difference $|\nabla V_R(y)-\nabla V_R(x)|$. We can restrict to $x, y\in B_{R+1}$. Indeed, if $y \notin B_{R}$, then the integrand is zero by the support of $\rho_t$. Consequently, if $x \notin B_{R+1}$, then $\omega_{R}(x-y)=0$. For $x, y \in B_{R+1}$, using the choice \eqref{eq:choice_eta_R}, we compute
    $$
            |\nabla V_R(y)-\nabla V_R(x)|
        \leq \int_{\Rd}\omega_R(z)|\nabla V(y-z)-\nabla V(x-z)|\diff z \leq \frac{1}{R}\int_{\Rd}\omega_R(z)\diff z
        =\frac{1}{R}.
    $$
    We can thus bound the term above by
    \begin{align*}
        \int_{\Rd} &\left\vert k\ast\left( \nabla \Big( (\rho_t\ast\omega_R)e^{V_R}\Big)e^{-\frac{\pot}{2}}\right)\right\vert
		\cdot \left\vert\int_{\Rd} \rho_t(y) \, \omega_R(x-y)\Big( \nabla V_R(y)-\nabla V_R(x)\Big)\diff y \right\vert\,e^{V_R-\frac{\pot}{2}}\diff x\\
		&\,\leq \frac{1}{R} \int_{\Rd} \left\vert k\ast \left(\nabla \Big( (\rho_t\ast\omega_R)e^{V_R}\Big)e^{-\frac{\pot}{2}}\right)\right\vert
		\cdot (\rho_t\ast\omega_R) \, e^{V_R-\frac{\pot}{2}} \diff x.
    \end{align*}
    We use the Hölder inequality, the Cauchy-Schwarz inequality and \ref{lem:Plancherel_f_H-1:wd} from Lemma \ref{lem:Plancherel_f_H-1} to estimate
	\begin{align*}
		\int_{\Rd} \left\vert k\ast \left(\nabla \Big( (\rho_t\ast\omega_R)e^{V_R}\Big)e^{-\frac{\pot}{2}}\right)\right\vert
		&\cdot (\rho_t\ast\omega_R) \, e^{V_R-\frac{\pot}{2}} \diff x\\
		&\leq 
		\left\|k\ast \left(\nabla ( (\rho_t\ast\omega_R)e^{V_R})e^{-\frac{\pot}{2}}\right)\right\|_{H_x^1}
		\left\|(\rho_t\ast\omega_R)e^{V_R-\frac{\pot}{2}}\right\|_{L_x^2}\\
		&\leq C\, \left(\left\|\nabla ( (\rho_t\ast\omega_R)e^{V_R})e^{-\frac{\pot}{2}}\right\|_{H_x^{-1}}^2+\left\|(\rho_t\ast\omega_R)e^{V_R-\frac{\pot}{2}}\right\|_{L_x^2}^2\right)
	\end{align*}
    for some constant $C\geq 1$ independent of $R$. Thus, we can transform \eqref{lem:uniform_estimates:step1:eq2} to get
	\begin{multline*}
		\prt_t \KL(\rho_t||\rho^R_\infty)+\frac{1}{R} \int_{B_R} \frac{|\nabla \rho_t+\rho_t \nabla V_R|^2}{\rho_t}\diff x
		+\DD^2(\rho_t\ast\omega_R||\rho_\infty^R)\\
        -\frac{C}{R}\Big(\norm{(\rho_t\ast \omega_R)e^{V_R-\frac{\pot}{2}}}_{L_x^2}^2
        +\norm{\nabla\big((\rho_t\ast \omega_R)e^{V_R}\big)e^{-\frac{\pot}{2}}}_{ H_x^{-1}}^2\Big)
		\leq 0,
	\end{multline*}
	where $C>0$ is some constant independent of $R$. Integrating in time yields the result.
    
    \underline{\it Estimate \ref{lem:uniform_estimates:L2_est}.}
    Using \eqref{eq:condition_Fourier_transform_k} and \ref{lem:Plancherel_f_H-1:plancherel} from Lemma \ref{lem:Plancherel_f_H-1} we have
    \begin{equation}\label{eq:equivalence_FI_H-1_proof_estimates}
\frac{1}{\rkc_k}\, \norm{(\rho_t\ast\omega_R)e^{V_R})e^{-\frac{\pot}{2}}}_{H_x^{-1}}^2\leq \DD^2(\rho_t\ast\omega_R||\rho_\infty^R)\leq \rkc_k\, \norm{(\rho_t\ast\omega_R)e^{V_R})e^{-\frac{\pot}{2}}}_{H_x^{-1}}^2.
    \end{equation}
    Using \eqref{eq:pointwise_lower_bound} from Lemma \ref{lem:control_neg_log} and \eqref{eq:estimate_V_eta_from_below} we can bound from below 
    \begin{align}
    \label{eq:bound_kl_from_below_unormalized}
    \KL(\rho_T||\rho_\infty^R)
        \geq -\frac{2}{e} \int_{B_R} e^{-\frac{V_R}{2}}\diff x
        \geq -\frac{2}{e} e^{-\frac {\ngib}{2}}\int_{\Rd} e^{-\frac{\pot}{2}}\diff x
        >-\infty.
    \end{align}
    Then, using \eqref{lem:uniform_estimates:step1:eq1} and \eqref{eq:equivalence_FI_H-1_proof_estimates} we get
    \begin{align}
        \label{eq:bound_dissipation_l2_intermediate_step}
            \left(1-\frac{C}{R}\right)\int_0^T \DD^2(\rho_t\ast\omega_R||\rho_\infty^R)\diff t
            -\frac{C}{R} \norm{(\rho\ast \omega_R)e^{V_R-\frac{\pot}{2}}}_{L_{t,x}^2}^2
            \leq \KL(\rho_0^R||\rho_\infty^R)
            +C,
    \end{align}
    for some constant $C>0$ not depending on $R$. Additionally, using \eqref{eq:inequality_l2_lower_bound_formulated_H-1_H1} and \eqref{eq:bound_L2_H-1_norm_by_dissipation} we derive an estimate of the $L^2$-norm by
    \begin{align*}
        \left[\frac{\lambda}{C_0} \left(1-\frac{C}{R}\right)-\frac{C}{R}\right]&\norm{(\rho\ast \omega_R)e^{V_R-\frac{\pot}{2}}}_{L_{t,x}^2}^2\\
        &\leq \frac{\lambda}{C_0} \left(1-\frac{C}{R}\right)\int_0^T C_0(1+\DD^2(\rho_t\ast\omega_R||\rho_\infty^R))\diff t-\frac{C}{R} \norm{(\rho\ast \omega_R)e^{V_R-\frac{\pot}{2}}}_{L_{t,x}^2}^2\\
            &\leq \KL(\rho_0^R||\rho_\infty^R)
            +C
            +\lambda T \left(1-\frac{C}{R}\right),
    \end{align*}
    where $C_0$ is the constant in \eqref{eq:bound_L2_H-1_norm_by_dissipation}, which we can bound independently from $R$ by Remark \ref{rem:facts_main_thm} and by the fact $V_R\ra V$ in $L_\loc^\infty$. We thus see that \ref{lem:uniform_estimates:L2_est} holds by \ref{lem:uniform_estimates:KL} and by choosing $R$ sufficiently large.

    \underline{\it Estimates \ref{lem:uniform_estimates:H-1_est}-- \ref{lem:uniform_estimates:L^2_est_diff}.}
    The estimate \ref{lem:uniform_estimates:H-1_est} is an immediate consequence of \ref{lem:uniform_estimates:KL}, \ref{lem:uniform_estimates:L2_est}, the estimate \eqref{eq:bound_dissipation_l2_intermediate_step} and the equivalence \eqref{eq:equivalence_FI_H-1_proof_estimates}. Consequently, the estimate \ref{lem:uniform_estimates:KL_all_times} follows from \eqref{lem:uniform_estimates:step1:eq1}, \eqref{eq:bound_kl_from_below_unormalized}, the equivalence \eqref{eq:equivalence_FI_H-1_proof_estimates} and \ref{lem:uniform_estimates:KL}, \ref{lem:uniform_estimates:L2_est}, \ref{lem:uniform_estimates:H-1_est}. Then, \ref{lem:uniform_estimates:H1_est} is derived from \ref{lem:uniform_estimates:H-1_est} and \ref{lem:Plancherel_f_H-1:wd} in Lemma~\ref{lem:Plancherel_f_H-1}, while \ref{lem:uniform_estimates:L^2_est_diff} follows directly from \eqref{lem:uniform_estimates:step1:eq1} and the previous bounds.

    \underline{\it Estimate \ref{lem:uniform_estimates:der_est}.}
    Choose any $q= 3\vee d$. Let $\psi \in C_c^{\infty}((0,T)\times B_L)$ and extend it with 0 to the whole of $\R^d$. We use \eqref{eq:regularised_weak_pde}--\eqref{ass:mfsvgd:ini} and compute
    \begin{align}
        \label{thm:existence_of_distributional_solutions_to_mfsvgd:step4:eq2}
        \begin{split}
        &\Big\vert\int_0^T \int_{B_L} \rho \cdot \prt_t \psi \diff x \diff t\Big\vert
        \leq \frac{1}{R} \left\vert\int_0^T \int_{B_L} (\nabla \rho
			+\rho\nabla V_R)\cdot \nabla\psi\diff x\diff t\right\vert\\
        &\qquad+\left\vert\int_0^T \int_{B_{L+\eta}} k\ast \nabla\big((\rho\ast\omega_R) e^{V_R}\big)e^{-\frac{\pot}{2}}\cdot \big((\rho \nabla\psi)\ast\omega_R\big) e^{V_R-\frac{\pot}{2}}\diff x\diff t\right\vert.
        \end{split}
    \end{align}
	To bound the first term in \eqref{thm:existence_of_distributional_solutions_to_mfsvgd:step4:eq2} we use the Hölder inequality twice, \ref{lem:uniform_estimates:L^2_est_diff} and conservation of mass as follows 
	\begin{align*}
		\frac{1}{R}\left\vert\int_0^T \int_{B_L} (\nabla \rho+\rho\nabla V_R)\cdot \nabla\psi\diff x\diff t\right\vert &\leq \frac{1}{\sqrt{R}} \left\|\frac{1}{\sqrt{R}}\,\frac{\nabla\rho+\rho\nabla V_R}{\sqrt{\rho}}\right\|_{L^2_{t,x}} \, \left\| \rho |\nabla\psi|^2 \right \|_{L^1_{t,x}}^{\frac{1}{2}} \\
			&\leq C\, \norm{\psi}_{L_t^2 W_x^{1,\infty}} \leq C\, T^{\frac{1}{2}-\frac{1}{2q}} \norm{\psi}_{L_t^{2q} W_x^{1,\infty}},
	\end{align*}
    where $C>0$ is a constant independent of $R$ and $L$. Next, we bound the second term in \eqref{thm:existence_of_distributional_solutions_to_mfsvgd:step4:eq2}. For this we use the Hölder inequality and estimate
    \begin{align}
        \label{thm:existence_of_distributional_solutions_to_mfsvgd:step4:eq3}
        \begin{split}
        &\left\vert\int_0^T \int_{B_{L+1}} k\ast \nabla\big((\rho\ast\omega_R) e^{V_R}\big)e^{-\frac{\pot}{2}}\cdot \big((\rho \nabla\psi)\ast\omega_R\big) e^{V_R-\frac{\pot}{2}}\diff x\diff t\right\vert\\
        &\qquad\leq \norm{k\ast \nabla\big((\rho\ast\omega_R) e^{V_R}\big)e^{-\frac{\pot}{2}}}_{L_t^2 L_x^{\frac{2q}{q-2}}(B_{L+1})} \norm{\big((\rho \nabla\psi)\ast\omega_R\big) e^{V_R-\frac{\pot}{2}}}_{L_t^2 L_x^{\frac{2q}{q+2}}(B_{L+1})}.
        \end{split}
    \end{align}
   By the Sobolev embedding theorem we have the continuous embedding $H^1(B_{L+1})\hookrightarrow L^{\frac{2q}{q-2}}(B_{L+1})$ (to see this for $d=1$, note that the Sobolev embedding theorem implies $H^1(B_{L+1})\hookrightarrow C^{\frac{1}{2}}(B_{L+1})\hookrightarrow L^{\frac{2q}{q-2}}(B_{L+1})$; to see this for $d=2$, note the embedding implies $H^1(B_{L+1})\hookrightarrow W^{1,\frac{3}{2}}(B_{L+1})\hookrightarrow L^{\frac{2q}{q-2}}(B_{L+1})$). Thus, we bound the first term in \eqref{thm:existence_of_distributional_solutions_to_mfsvgd:step4:eq3} using these embeddings and \ref{lem:uniform_estimates:H1_est}. The left-over term in \eqref{thm:existence_of_distributional_solutions_to_mfsvgd:step4:eq3} can be estimated using the Hölder inequality twice
   \begin{align*}
       \norm{\big((\rho \nabla\psi)\ast\omega_R\big) e^{V_R-\frac{\pot}{2}}}&_{L_t^2 L_x^{\frac{2q}{q+2}}(B_{L+1})} 
        \leq \norm{\big(\rho\ast\omega_R\big) e^{V_R-\frac{\pot}{2}}\cdot \norm{\nabla\psi}_{W_x^{1,\infty}}}_{L_t^2 L_x^{\frac{2q}{q+2}}(B_{L+1})}\\
        &\leq |B_{L+1}|^{\frac{1}{2d}}\norm{\big(\rho\ast\omega_R\big) e^{V_R-\frac{\pot}{2}}}_{L_t^{\frac{2q}{q-1}} L_x^{\frac{2q}{q+1}}(B_{L+1})}\norm{\psi}_{L_t^{2q}W_x^{1,\infty}(B_{L})}.
   \end{align*}
    Using \ref{lem:uniform_estimates:L2_est} and the conservation of mass, we observe that we have the uniform estimate on $\big(\rho\ast\omega_R\big) e^{V_R-\frac{\pot}{2}}\in L_t^\infty L_x^1(B_{L+1})\cap L^2_{t} L^2_{x}(B_{L+1})$. Thus, by interpolation of $L^p$-norms, we know that
	\begin{align*}
		\rho \in L_t^{p_t}L_x^{p_x},
		\qquad \frac{1}{p_t}=\frac{1-\theta}{2},
		\qquad \frac{1}{p_x}=\theta+\frac{1-\theta}{2}
		\qquad\forall \theta\in (0,1).
	\end{align*}
	If we choose $\theta=\frac{1}{q}$, then we get $p_x=\frac{2q}{q+1}$ and $p_t=\frac{2q}{q-1}$. In particular, this means that we can uniformly bound the second term in \eqref{thm:existence_of_distributional_solutions_to_mfsvgd:step4:eq3}. Finally, by the Sobolev embedding $H^m_0(B_{L})\hookrightarrow W^{1,\infty}_0(B_{L})$, there exists a constant $C>0$, that does not depend on $R$, such that
	\begin{align*}
		\left\vert\int_0^T \int_{\Rd}  \rho\cdot \prt_t\psi\diff x\diff t\right\vert
		\leq C \norm{\psi}_{L_t^{2q} H_x^m(B_L)}.
	\end{align*}
    Taking the supremum in $\norm{\psi}_{L_t^{2q} H_x^m(B_L)}\leq 1$ proves \ref{lem:uniform_estimates:der_est}, thus concluding the proof. 
\end{proof}

\subsection{Passing to the limit $R\ra\infty$}\label{sect:passage_limit_proof_main}

We first establish compactness of the sequence $\{\rho^R\}_{R\geq R_0}$, where $R_0$ is the value from Lemma \ref{lem:uniform_estimates}.

\begin{lem}\label{lem:weak_convergences_mfsvgd}
    Let $m>\frac{d}{2}$ be fixed. There exists a subsequence (not relabelled) $\{\rho^R\}_{R\geq R_0}$ and function $\rho:[0,\infty) \times \R^d\to \R^+$ with $\rho_t \in \Pro$ for all $t\in[0,\infty)$ such that  
\begin{equation}\label{eq:compactness_uniform_in_time_neg_Sobolev}
    \rho^R_t \to \rho_t \mbox{ in } C([0,T]; H^{-m}(B_L)) \, \mbox{ for all } T, L >0,  
    \end{equation}
    \begin{equation}\label{eq:weak_L1_compactness_each_time}
    \rho^R_t \rightharpoonup \rho_t \mbox{ weakly in } L^1(\Rd) \mbox{ for all } t\in [0,\infty). 
    \end{equation}
    Moreover, for all $L>0$, $T>0$ and $\vp\in C_c^\infty([0,T]\times B_L)$ we have
    \begin{enumerate}[label=(L\arabic*)]
        \item \label{lem:weak_convergences_mfsvgd:dissipation}
        $\frac{1}{R}(\nabla\rho^R+\rho^R\nabla V_R)\to 0$ strongly in $L^1(\Rd)$,
        \item \label{lem:weak_convergences_mfsvgd:L2} $((\rho^R \nabla \vp)\ast \omega_R) e^{V_R-\frac{\pot}{2}}\weak (\rho\nabla \vp) e^{V_R-\frac{\pot}{2}}$ weakly in $L^2((0,T)\times \Rd)$ and strongly in $L^2(0,T; H^{-1}(\Rd))$, 
        \item \label{lem:weak_convergences_mfsvgd:H-1}
        $\nabla((\rho^R\ast\omega_R) e^{V_R})e^{-\frac{\pot}{2}}\weak \nabla(\rho e^{V})e^{-\frac{\pot}{2}}$ weakly in $L^2(0,T;H^{-1}(\Rd))$,
        \item \label{lem:weak_convergences_mfsvgd:H1}
        $k\ast \nabla((\rho^R\ast\omega_R) e^{V_R})e^{-\frac{\pot}{2}}\weak k\ast \nabla(\rho e^{V})e^{-\frac{\pot}{2}}$ weakly in $L^2(0,T;H^1(\Rd))$.
    \end{enumerate}
\end{lem}
\begin{rem}
    In particular, from \eqref{eq:compactness_uniform_in_time_neg_Sobolev} we have that the map $[0,\infty) \ni t \mapsto \int_{\Rd} \psi(x) \, \rho_t(x) \diff x$ is continuous for all $\psi \in H^m_0(B_L)$.
\end{rem}
\begin{proof}[Proof of Lemma \ref{lem:weak_convergences_mfsvgd}]
   \underline{\it Proof of \eqref{eq:compactness_uniform_in_time_neg_Sobolev}.} Let $q= d\vee 3$ and let $T>0$, $L>0$ be fixed. We argue by an Arzéla-Ascoli argument, i.e. we show that the map $t\mapsto \rho_t^R$ is uniformly bounded and equicontinuous in $H^{-m}(B_L)$ as well as $\{\rho^R\}_{R\geq R_0}$ is compact in $H^{-m}(B_L)$ for each $t$.
   
   Using \ref{lem:uniform_estimates:der_est}, there exists a constant $C>0$ independent of $R$ such that for all $\psi\in H^m_0(B_L)$ and $t_1,t_2 \in [0,T]$
    \begin{align*}
        \left\vert \int_{\Rd}\psi\, \rho_{t_1}^R \diff x -\int_{\Rd}\psi \, \rho_{t_2}^R \diff x \right\vert
        \leq \norm{\prt_t \rho^R}_{(L_t^{2q} H_x^{m})\raisebox{3pt}{$\ast$}}\norm{\chi_{[t_1,t_2]}\psi}_{L_t^{2q} H_x^m}
        \leq C|t_1-t_2|^{\frac{1}{2q}}\norm{\psi}_{H^m_x(B_L)}.
    \end{align*}
    Taking the supremum over $\norm{\psi}_{H^m_x(B_L)} \leq 1$ we obtain
    \begin{equation}\label{eq:uniform_continuity_H^{-m}}
    \| \rho_{t_1}^R - \rho_{t_2}^R \|_{H^{-m}(B_L)} \leq C|t_1-t_2|^{\frac{1}{2q}}. 
    \end{equation}
    Using \ref{lem:uniform_estimates:ini_H-1_est} we see the curves are uniformly bounded in $C^{\frac{1}{2q}}([0,T];\, H^{-m}(B_L))$ with the upper bound $\norm{\rho_t^R}_{H^{-m}}\leq \norm{\rho_0^R}_{H^{-m}}+C T^{\frac{1}{2q}}$. Moreover, $L^1(B_L)$ is compactly embedded in $H^{-m}(B_L)$ (Lemma \ref{lem:lem_compactness_measures_H-m}) so that for each fixed $t\in[0,T]$, $\{\rho_t^R\}_{R \geq R_0}$ has a subsequence converging strongly in $H^{-m}(B_L)$. Let $\Lambda\subseteq [0,T]$ be a countable dense subset. By a diagonal argument we may assume $\rho_t^R\to \rho_t$ in $H^{-m}(B_L)$ as $R\ra\infty$ for all $t\in \Lambda$. Finally, using the Hölder continuity in \eqref{eq:uniform_continuity_H^{-m}} and completeness of $H^{-m}(B_L)$, there is a unique way of extending $\rho_t$ from $\Lambda$ to $[0,T]$ such that the convergence $\rho_t^R \to \rho_t$ is true for all $t \in [0,T]$ and $\rho\in C^{\frac{1}{2q}}([0,T];\, H^{-m}(B_L))$. Finally, we use a diagonal argument once again to extend the result for each ball $B_L$ and each interval of time $[0,T]$.

    \underline{\it Proof of \eqref{eq:weak_L1_compactness_each_time}.}
    Fix $t \in [0,\infty)$. There exists $R_0\geq 2$ such that by estimate \ref{lem:uniform_estimates:KL_all_times}, condition \eqref{eq:general_prop_target_potentialV_f} and estimate \eqref{eq:lower_bound_negative_log} in Lemma \ref{lem:control_neg_log} we have the uniform bounds
    \begin{align*}
        \sup_{R\geq R_0}
        \int_{\Rd} \rho_t^R |\log(\rho_t^R)| \diff x
        +\sup_{R\geq R_0}
        \int_{\Rd} |x|^2\,\rho_t^R\diff x
        <\infty.
    \end{align*}
    In particular, the set $\{\rho_t^R\}_{R \geq R_0}$ is equi-integrable and tight so by the Dunfard-Pettis theorem, we have that $\{\rho_t^R\}_{R\geq R_0}$ is weakly precompact in $L^1(\Rd)$. Up to passing to a subsequence, we have $\rho_t^R\weak \mu$ in $L^1(\Rd)$ as $R\ra\infty$ for some $\mu\in L^1(\Rd)$. Let $\vp\in C_c^\infty(\Rd)$ with $\supp(\vp)\subseteq B_L$ for some $L>0$. By \eqref{eq:compactness_uniform_in_time_neg_Sobolev} we have
    \begin{align*}
        \int_{\Rd}\vp\, \mu\diff x
        =\lim_{R\ra\infty}\int_{\Rd}\vp\, \rho_t^R\diff x
        =\langle \rho_t, \vp \rangle_{H^{-m}(B_L), H^{m}_0(B_L)}.
    \end{align*}
    It follows that in the sense of distributions $\rho_t=\mu$ so that $\rho_t$, a priori an element of $H_{\loc}^{-m}(\Rd)$, is in fact an $L^1(\Rd)$ function. Finally, the reasoning above shows that every subsequence of $\{\rho_t^R\}_{R\geq R_0}$ has a further subsequence converging to $\rho_t$, so that we deduce \eqref{eq:weak_L1_compactness_each_time}. 

    \underline{\it Convergence \ref{lem:weak_convergences_mfsvgd:dissipation}.}    
    Using \ref{lem:uniform_estimates:L^2_est_diff}, the fact that $\|\rho^R\|_{L^{\infty}_t L^1_x} \leq 1$ by conservation of mass, and by multiplying and dividing $\rho_t^R$, we obtain
	\begin{align*}
		 \frac{1}{R} \int_0^T \int_{B_R} |\nabla\rho^R+\rho^R\nabla V_R| \diff x\diff t
		\leq \frac{1}{\sqrt{R}} \left\| \frac{1}{\sqrt{R}} \frac{\nabla\rho^R+\rho^R\nabla V_R}{ \sqrt{\rho^{R}}}\right\|_{L^2_{t,x}}\, \sqrt{T} \to 0.
	\end{align*}

    \underline{\it Convergence \ref{lem:weak_convergences_mfsvgd:L2}.} First we show that the sequence $\{((\rho^R \nabla \vp)\ast \omega_R) e^{V_R-\frac{\pot}{2}}\}_{R\geq R_0}$ is bounded in $L^2((0,T)\times \Rd)$. This follows immediately from \ref{lem:uniform_estimates:L2_est} and the estimate
    \begin{align}\label{eq:L2_bound_sequence_that_conv_strongly_in_H-1_F2}
        \norm{((\rho^R \nabla \vp)\ast \omega_R) e^{V_R-\frac{\pot}{2}} }_{L_{t,x}^2}
        \leq \norm{\nabla\vp}_\infty\norm{(\rho^R\ast \omega_R) e^{V_R-\frac{\pot}{2}} }_{L_{t,x}^2}.
    \end{align} 
    By the Banach-Alaoglu theorem, up to passing to a subsequence, $((\rho^R \nabla \vp)\ast \omega_R) e^{V_R-\frac{\pot}{2}} \weak \xi$ weakly in $L^2((0,T)\times\Rd)$ for some $\xi$. To identify $\xi$, let $\psi\in C_c^\infty([0,T]\times \Rd)$. We have
    $$
    \int_0^T \int_{\Rd} ((\rho^R \nabla \vp)\ast \omega_R) e^{V_R-\frac{\pot}{2}} \, \psi \diff x \diff t = \int_0^T \int_{\Rd} ((\psi e^{V_R-\frac{\pot}{2}})\ast \omega_R)\nabla \vp \, \rho^R \diff x \diff t.
    $$
    Using convergence $((\psi e^{V_R-\frac{\pot}{2}})\ast \omega_R)\nabla \vp\ra \psi e^{V-\frac{\pot}{2}} \nabla \vp$ in $L^\infty((0,T)\times\Rd)$ and \eqref{eq:weak_L1_compactness_each_time} we obtain
    $$
    \int_0^T \int_{\Rd} \xi \, \psi \diff x \diff t  = \int_0^T \int_{\Rd} \psi e^{V-\frac{\pot}{2}} \nabla \vp \, \rho \diff x \diff t
    $$
    which implies $\xi =  e^{V-\frac{\pot}{2}} \nabla \vp \, \rho$. Standard subsequence argument as in the proof of \eqref{eq:weak_L1_compactness_each_time} above concludes the proof of weak convergence in $L^2((0,T)\times\Rd)$.

    \smallskip

    We proceed to the proof of strong convergence for the sequence $\{((\rho^R \nabla \vp)\ast \omega_R) e^{V_R-\frac{\pot}{2}}\}_{R \geq R_0}$ in $L^2(0,T; H^{-1}(\Rd))$ and we recall that $\vp\in C_c^\infty([0,T]\times B_L)$. Clearly, by \eqref{eq:L2_bound_sequence_that_conv_strongly_in_H-1_F2}, the sequence is bounded in $L^2(0,T; H^{-1}(\Rd))$. Moreover, the sequence is supported only on $(0,T)\times B_{L+1}$, so it is sufficient to prove strong convergence in $L^2(0,T; H^{-1}(B_{L+1}))$. We will argue by the Aubin-Lions lemma and for this we only need to prove that the sequence of time derivatives $\{ \partial_t ((\rho^R \nabla \vp)\ast \omega_R) e^{V_R-\frac{\pot}{2}}\}_{R\geq R_0}$ is uniformly bounded in a certain negative Sobolev space. Let $\psi \in C_c^{\infty}((0,T)\times B_{L+1})$. We compute
    \begin{align*}
    \int_0^T &\int_{\Rd} ((\rho^R \nabla \vp)\ast \omega_R) e^{V_R-\frac{\pot}{2}} \, \partial_t \psi \diff x \diff t = \int_0^T \int_{\Rd} \rho^R \, \nabla \vp \,  \partial_t \left[ \omega_R \ast (e^{V_R-\frac{\pot}{2}} \,  \psi) \right]
    \diff x \diff t \\
    &= \int_0^T \int_{\Rd} \rho^R \,   \partial_t \left[ \nabla \vp \, \omega_R \ast (e^{V_R-\frac{\pot}{2}} \,  \psi) \right]
    \diff x \diff t  - \int_0^T \int_{\Rd} \rho^R \,   \partial_t\nabla \vp \,    \, \omega_R \ast (e^{V_R-\frac{\pot}{2}} \,  \psi) 
    \diff x \diff t.
    \end{align*}
    Let $q = 3\vee d$, then the first term can be bounded by the estimate \ref{lem:uniform_estimates:der_est} by the upper bound $ C\, \|\psi\|_{L_t^{2q} H_x^m(B_{L+1})}$ for some constant $C>0$ depending on $\vp$, $T$, $L$ and $\norm{V}_{H^m(B_{L+1})}$. The second term is bounded by $C\, \|\psi\|_{L_t^{2q} L_x^\infty(B_{L+1})}$, which by the Sobolev embedding can be estimated by $ C\, \|\psi\|_{L_t^{2q} H_x^m(B_{L+1})}$ for constants $C>0$ depending on $\varphi$, $T$ and $L$. Combining both estimates we conclude
    $$
    \left| \int_0^T \int_{\Rd} ((\rho^R \nabla \vp)\ast \omega_R) e^{V_R-\frac{\pot}{2}} \, \partial_t \psi \diff x \diff t \right| \leq C\, \| \psi \|_{L^{2q}_t H^m_x(B_{L+1})},
    $$
    for some constant $C>0$ depending on $\vp$, $T$, $L$ and $\norm{V}_{H^m(B_{L+1})}$. Taking a supremum over all $\psi \in C_c^{\infty}((0,T)\times B_{L+1})$ with $ \| \psi \|_{L^{2q}_t H^m_x(B_{L+1})} \leq 1$, we obtain that the squence $\{ \partial_t ((\rho^R \nabla \vp)\ast \omega_R) e^{V_R-\frac{\pot}{2}}\}_{R\geq R_0}$ is bounded in $L^{2q}(0,T; H^m_0(B_{L+1}))^*$. The strong convergence follows by the Aubin-Lions lemma.

    \underline{\it Convergences \ref{lem:weak_convergences_mfsvgd:H-1} and \ref{lem:weak_convergences_mfsvgd:H1}.} We first prove \ref{lem:weak_convergences_mfsvgd:H-1} arguing as above. By \ref{lem:uniform_estimates:H-1_est}, the sequence $\{\nabla((\rho^R\ast\omega_R) e^{V_R})e^{-\frac{\pot}{2}}\}_{R\geq R_0}$ is uniformly bounded in $L^2(0,T;H^{-1}(\Rd))$ so it has a subsequence converging weakly to some $\xi \in L^2(0,T;H^{-1}(\Rd))$. Since for $\psi \in C_c^\infty([0,T]\times \Rd)$, we have $(\nabla(\psi e^{-\frac{\pot}{2}})e^{V_R})\ast\omega_R\ra \nabla(\psi e^{-\frac{\pot}{2}})e^{V}$ in $L^\infty((0,T)\times\Rd)$ it is easy to identify the limit $\xi = \nabla(\rho e^{V})e^{-\frac{\pot}{2}}$ and deduce \ref{lem:weak_convergences_mfsvgd:H-1} by the subsequence argument. Finally, \ref{lem:weak_convergences_mfsvgd:H1} follows by \ref{lem:Plancherel_f_H-1:wd} in Lemma \ref{lem:Plancherel_f_H-1} and the fact that linear operators preserve weak convergence.
\end{proof}

\begin{proof}[Proof of Theorem \ref{thm:existence_of_distributional_solutions_to_mfsvgd}]   
   \underline{\it Passing to the limit $R\ra\infty$.} Fix $L>0, T>0$, $\vp\in C_c^\infty([0,T]\times~B_L)$. For each $R\geq R_0$ we have
    \begin{align*}    	    \int_{\Rd}\psi(T,x)\, \rho^R_T(x) \diff x =& \int_{\Rd}\psi(0,x)\, \rho^R_0(x) \diff x + 
            \int_0^T \int_{\Rd}\prt_t\psi\; \rho^R \diff x\diff t\\
            &-\int_0^T\int_{\Rd}  ((\rho^R \nabla \psi )\ast \omega_R) e^{V_R-\frac{\pot}{2}} \, k\ast \nabla((\rho^R\ast\omega_R) e^{V_R})e^{
     -\frac{\pot}{2}}\diff x\diff t\\
     & - \frac{1}{R}\, \int_0^T \int_{\Rd}(\nabla\rho^R+\rho^R\nabla V_R) \cdot \nabla \psi \diff x\diff t 
    \end{align*}
    By \eqref{eq:weak_L1_compactness_each_time}, we can easily pass to the limit in the first three terms. Also, the last term converges to 0 by \ref{lem:weak_convergences_mfsvgd:dissipation}. It remains to establish the limit of the penultimate term. Let
    $$
    f_R := ((\rho^R \nabla \psi )\ast \omega_R) e^{V_R-\frac{\pot}{2}}, \qquad  g_R : =k\ast \nabla((\rho^R\ast\omega_R) e^{V_R})e^{
     -\frac{\pot}{2}}.
    $$
    From \ref{lem:weak_convergences_mfsvgd:L2} and \ref{lem:weak_convergences_mfsvgd:H1} in Lemma \ref{lem:weak_convergences_mfsvgd} we know that the sequences $\{f_R\}_{R\geq R_0}$ and $\{g_R\}_{R\geq R_0}$ converge strongly to $f$ in $L^2(0,T; H^{-1}(\Rd))$ and weakly to $g$ in $L^2(0,T; H^{1}(\Rd))$, respectively. We have
    $$
    \lim_{R \to \infty} \int_0^T \int_{\Rd} f_n\, g_n \diff x \diff t = \lim_{R \to \infty} \langle f_n, g_n \rangle_{L^2_t H^{-1}_x, L^2_t H^1_x} = \langle f, g \rangle_{L^2_t H^{-1}_x, L^2_t H^1_x} = \int_0^T \int_{\Rd} f\, g \diff x \diff t,
    $$
    where the last step follows from the fact that $f$ and $g$ are in fact functions in $L^2((0,T)\times\Rd)$. This concludes the proof of the weak formulation \eqref{eq:weak_sol_mfsvgd}.

    \underline{\it Energy dissipation inequality \eqref{eq:KL:dissipation_inequality_solution}.} We want to send $R \to \infty$ in \eqref{lem:uniform_estimates:step1:eq1}. Using \cite[Theorem 2.34]{MR1857292} and \eqref{eq:weak_L1_compactness_each_time}, we obtain 
    \begin{align*}
        \KL(\rho_t || \rho_\infty)
        \leq \liminf_{R\ra \infty}\KL(\rho_t^R || \rho_\infty^R).
    \end{align*}
    From \ref{lem:weak_convergences_mfsvgd:H-1} in Lemma \ref{lem:weak_convergences_mfsvgd} we have weak convergence $\nabla((\rho^R\ast\omega_R) e^{V_R})e^{-\frac{\pot}{2}}\weak \nabla(\rho e^{V})e^{-\frac{\pot}{2}} $ in $L_t^2 H_x^{-1}(\Rd)$. Thanks to \eqref{eq:condition_Fourier_transform_k} and \ref{lem:Plancherel_f_H-1:plancherel} from Lemma \ref{lem:Plancherel_f_H-1}, we know that the dissipation term $\DD^2(\rho_s^R \ast \omega_R ||\rho_\infty^R)$ is equivalent to the $H^{-1}(\Rd)$ norm of $\nabla((\rho_s^R\ast\omega_R) e^{V_R})e^{-\frac{\pot}{2}}$. Thus by weak lower semicontinuity of this equivalent norm, we have the inequality
    \begin{align*}
        \int_0^t \DD^2(\rho_s||\rho_\infty)\diff s
        \leq \liminf_{R\ra \infty} \int_0^t \DD^2(\rho_s^R\ast \omega_R||\rho_\infty^R)\diff s.
    \end{align*}
Finally, by \eqref{eq:conv_KL_time_0}, $\KL(\rho_0^R ||\rho_\infty^R)\ra \KL(\rho_0 ||\rho_\infty)$. Thus, we can pass to the limit $R \to \infty$ in \eqref{lem:uniform_estimates:step1:eq1} and deduce the first statement in \eqref{eq:KL:dissipation_inequality_solution}, which together with the Grönwall inequality and the \eqref{eq:stein_log_sobolev_inequality} implies the second statement in \eqref{eq:KL:dissipation_inequality_solution}.

\underline{\it Absolute continuity of solution with respect to $W_1$.}
Take any test function $\psi\in C_c^\infty(\Rd)$ and observe we have the integrability
\begin{multline*}
    \int_0^T \int_{\Rd} \big\vert k\ast \Big(\nabla \left( \rho\, e^{V} \right) e^{-\frac{\pot}{2}}\Big) \, \rho \, e^{V-\frac{\pot}{2}} \, \nabla \psi\big\vert \diff x \diff t
    \\ \leq \norm{\rho e^{V-\frac{\pot}{2}}}_{L_t^2 H_x^{-1}}
    \norm{k\ast \big(\nabla \left( \rho\, e^{V} \right) e^{-\frac{\pot}{2}}\big)}_{L_t^2 H_x^1} \norm{\nabla \psi}_{L^{\infty}_x}
\end{multline*}
Let $0\leq t_1 \leq t_2$, then from \eqref{eq:weak_sol_mfsvgd} we easily deduce
\begin{equation}\label{eq:weak_form_simplified_only_space_all_times}
    	    \int_{\Rd}\psi(x) \rho_{t_2}(x) \diff x - \int_{\Rd}\psi(x) \rho_{t_1}(x) \diff x=  
             -\int_{t_1}^{t_2} \int_{\Rd} k\ast \left(\nabla(\rho e^{V})\, e^{-\frac{\pot}{2}}\right)\cdot \rho\,  e^{V-\frac{\pot}{2}} \nabla \psi\diff x\diff s.
\end{equation}
Indeed, \eqref{eq:weak_form_simplified_only_space_all_times} can be proved by considering in \eqref{eq:weak_sol_mfsvgd} a test function $\psi(x) \, \vp_n(t) $, where $\vp_n\in C_c^\infty(\R)$ is a function with $0\leq \vp_n\leq 1$, $\supp(\vp_n)\subseteq [t_1-\frac{1}{n},t_2+\frac{1}{n}]$ and $\vp_n=1$ on $[t_1,t_2]$ and passing to the limit $n\ra\infty$. The identity can be easily extended to $\psi$ being 1-Lipschitz by considering a sequence $\{\psi_n\}_{n\geq 1} \subset C_c^{\infty}(\Rd)$ such that $\psi_n\ra\psi$ pointwisely. In particular, we may assume $|\psi_n(x)| \leq 2\, |\psi(0)| + 2|x|$ so the limit $\lim_{n\to \infty} \int_{\Rd}\psi_n(x)\, \rho_{t}(x) \diff x = \int_{\Rd}\psi(x)\, \rho_{t}(x) \diff x$ can be justified by the dominated convergence.

\smallskip

We can now take the absolute value in \eqref{eq:weak_form_simplified_only_space_all_times} and pass to the supremum over all $\psi\in \Lip(\Rd)$ with $\Lip(\psi)\leq 1$ to obtain the absolute contiuity of the curve
\begin{align*}
    W_1(\rho_{t_1},\rho_{t_2})
    \leq \int_{t_1}^{t_2} \int_{\Rd} \big\vert k\ast \Big(\nabla \left( \rho\, e^{V} \right) e^{-\frac{\pot}{2}}\Big) \, \rho \, e^{V-\frac{\pot}{2}} \big\vert \diff x \diff t \qquad \mbox{ for all } t_1, t_2 \in [0,T].
\end{align*}
With this we conclude the proof.
\end{proof}

\section{Necessary Properties of the Kernel}\label{section:necessary_properties_of_the_kernel}

\subsection{Proof of Theorem \ref{thm:conditions_for_slsi}, Case 
\ref{thm:failure_of_slsi_for_regular_convolution_kernels}}

\begin{lem}
    \label{lem:upper_bnd_dis}
     Let $r>1$ and assume $\rho$ satisfies
     \begin{align}
        \label{eq:condition_rho_general_r}
         \rho\in L^1(\Rd),\qquad
         \rho e^{V-\frac{\pot}{2}}\in L^2(\Rd),\qquad
         \rho e^{V-\frac{\pot}{2}}\nabla\pot\in H^{-r}(\Rd),
         \tag{$\text{H}_r$}
     \end{align}
     and $k\in L^1(\Rd)+L^2(\Rd)$ is such that
    \begin{equation}\label{eq:upper_bound_failure_case}
        |\hat{k}(\xi)|\leq \rkc_k\frac{1}{(1+|\xi|^2)^r}
        \qquad \forall \xi\in\Rd.
    \end{equation}
    Then, the Fisher information $ \DD^2(\rho\,||\,\rho_\infty)$, which like in \eqref{eq:inequality_l2_lower_bound_formulated_H-1_H1} can be defined by $H^{-r}$ and $H^r$ duality, satisfies
    \begin{align*}
        \DD^2(\rho\,||\,\rho_\infty)
        \leq \rkc_k\sqrt{\det(\Sigma)}\frac{(\sigma_d\vee 1)^r}{\sigma_1\; }\left(\int_{\Rd} |\hat{g}(\xi)|^2 \, q(\xi) \diff\xi+\frac{1}{16\,\pi^2}\int_{\Rd}  \frac{|\nabla \hat{g}(\xi)|^2}{(1+|\xi|^2)^r} \diff\xi\right),
    \end{align*}
     where $\rkc_k$ is the constant form \eqref{eq:condition_Fourier_transform_k}, $0<\sigma_1 \leq ... \leq \sigma_d$ are the eigenvalues of $\Sigma$ and we define the functions
     \begin{equation}\label{eq:def_g_0_failure_case}
        \begin{split}
              g_0(x)&: = \rho(x)e^{V(x)-\frac{\pot(x)}{2}}, \qquad 
          z(x):=\Sigma^{\frac{1}{2}}x+\mu,\qquad
          g:=g_0\circ z,\\
        q(\xi)&:=\left(4\pi^2\abs{\xi}^2-\frac{d}{2}+r\frac{|\xi|^2}{1+|\xi|^2}\right)\frac{1}{(1+|\xi|^2)^r}.
          \end{split}
     \end{equation}
\end{lem}
\begin{proof}
    We follow the proof of Lemma \ref{lem:passing_to_fourier}, making use of Lemma \ref{lem:Plancherel_f_H-1} for the case of duality between $H^{-r}$ and $H^r$, the exact same way up to \eqref{eq:first_lower_bnd_dis} and choosing $\tau=0$. Now, we control the length of the vector by the maximal eigenvalue of $\Sigma^{-\frac{1}{2}}$ which is $\frac{1}{\sqrt{\sigma_1}}$ to get the upper bound
    \begin{align*}
        \DD^2(\rho\,||\,\rho_\infty) \leq \frac{\sqrt{\det(\Sigma)}}{\sigma_1} \int_{\Rd}\hat{k}(\Sigma^{-\frac{1}{2}}\eta)\, \left|(2\pi i \eta)\, \hat{g}(\eta) + \frac{i}{2\,(2\pi)}  \nabla \hat{g}(\eta)  \right|^2 \diff \eta. 
    \end{align*}
    We estimate from below
    $$
    (1+ |\Sigma^{-\frac{1}{2}}\eta|^2)^r \geq (1 + \sigma_d^{-1}\, |\eta|^2)^r \geq (\sigma_d\vee 1)^{-r} \, (1+|\eta|^2)^r
    $$
    and use the upper bound in \eqref{eq:upper_bound_failure_case} to get
    \begin{align*}
        \DD^2(\rho\,||\,\rho_\infty) \leq 
        \rkc_k\sqrt{\det(\Sigma)}\frac{(\sigma_d\vee 1)^r}{\sigma_1\; }
        \int_{\Rd}\frac{1}{(1+|\eta|^2)^r}\, \left|(2\pi i \eta)\, \hat{g}(\eta) + \frac{i}{2\,(2\pi)}  \nabla \hat{g}(\eta)  \right|^2 \diff \eta. 
    \end{align*}
    At this stage, the proof is exactly the same as in Lemma \ref{lem:passing_to_fourier}, except that $\hat{k}(\eta)$ is substituted with the function $(1+|\eta|^2)^{-r}$.
\end{proof}

The idea of the proof is to use Gaussians with vanishing variance. Under the Fourier transform the functions will converge to the constant one function. By the choice of $r$, this will show that the dissipation is bounded. However, the Kullback-Leibler divergence will blow up, since we are using a sequence of mollifiers that weakly converge to the Dirac measure.

\begin{proof}[Proof of Theorem \ref{thm:conditions_for_slsi}, Case \ref{thm:failure_of_slsi_for_regular_convolution_kernels}]
    Let $\{\rho_n\}_{n\geq 1}$ be a sequence such that each element satisfies \eqref{eq:condition_rho_general_r}. From Lemma \ref{lem:upper_bnd_dis} we know there exists a constant $C>0$ independent of $n$ such that
    \begin{align}
        \label{eq:bnd_dis_above_for_ce}
        \DD^2(\rho_n\,||\,\rho_\infty)\leq C \left(\int_{\Rd} |\hat{g}_n(\xi)|^2 \, q(\xi) \diff\xi+\frac{1}{16\,\pi^2}\int_{\Rd}  \frac{|\nabla \hat{g}_n(\xi)|^2}{(1+|\xi|^2)^r} \diff\xi\right), 
    \end{align}
    where we define $g_n$ and $q$ as in \eqref{eq:def_g_0_failure_case} with $\rho=\rho_n$. Thus it suffices to find a sequence that satisfies
    \begin{align}
        \label{eq:cond_fail_slsi_kreg}
        \sup_{n\geq 1}\int_{\Rd} |\hat{g}_n(\xi)|^2 \, q(\xi) \diff\xi+\frac{1}{16\,\pi^2}\int_{\Rd}  \frac{|\nabla \hat{g}_n(\xi)|^2}{(1+|\xi|^2)^r} \diff\xi<\infty,
        \quad \lim_{n\ra\infty} \KL(\rho_n||\rho_\infty)=\infty.
    \end{align}
    We choose the probability distribution $\rho_n=\mathcal{N}(\mu,\frac{1}{n}\Sigma)$, this means 
    \begin{align*}
        \rho_n(x)=\left(\frac{n}{2\pi}\right)^{\frac{d}{2}}\frac{1}{\sqrt{\det(\Sigma)}}\exp\left(-\frac{n}{2} (x-\mu)\cdot \Sigma^{-1}(x-\mu) \right),
    \end{align*}
    and compute the function $g_n$ explicitly as
    \begin{align*}
        g_n(x)
        =\rho_n(\Sigma^{\frac{1}{2}}x+\mu)e^{\frac{|x|^2}{4}}
        =\left(\frac{n}{2\pi}\right)^{\frac{d}{2}}\frac{1}{\sqrt{\det(\Sigma)}}e^{-|x|^2\left(\frac{2n-1}{4}\right)}.
    \end{align*}
    It is not difficult to verify that the Fourier transform satisfies
    \begin{align*}
        \hat{g}_n(x)
        =\frac{1}{\sqrt{\det(\Sigma)}}
        \left(\frac{1}{1-\frac{1}{2n}}\right)^{\frac{d}{2}}
        e^{-|\pi \xi|^2\left(\frac{4}{2n-1}\right)},
    \end{align*}
    so that $\hat{g}_n$, $\nabla \hat{g}_n$ are uniformly bounded in $L^{\infty}(\Rd)$. Moreover, we have chosen $r$ sufficiently large such that $q\in L^1(\Rd)$, $\int_{\Rd} \frac{1}{(1+|\xi|^2)^r}\diff\xi < \infty$ and the first condition in \eqref{eq:cond_fail_slsi_kreg} is satisfied.
    Next, we use Lemma~\ref{lem:kl_gaussians} to compute the Kullback-Leibler divergence
    \begin{align*}
        \KL(\rho_n||\rho_\infty)
        =\frac{d}{2}\left(\frac{1}{n}-1+\ln(n)\right)
    \end{align*}
    so the second condition \eqref{eq:cond_fail_slsi_kreg} is satisfied and the proof is concluded.
\end{proof}

\subsection{Proof of Theorem \ref{thm:conditions_for_slsi}, Case \ref{thm:failure of slsi for integrable convolution kernels}}

Let us first explain the proof strategy. First we bound the dissipation in terms of $L^p$ norms of $(1+|x|^\beta)\nabla\rho(x)$ and $(1+|x|^\beta)\rho(x)\nabla V(x)$. We choose a sequence such that the Kullback-Leibler divergence blows up, but the aforementioned estimates remain bounded. It turns we can do this by choosing $\rho_\infty$ as a Gaussian and $\rho$ to be an approximation of a polynomial tail.

\smallskip

Let $p:=\frac{2r}{2r-1}$ and $q:=2r$ so that
\begin{align*}
    \frac{1}{p}+\frac{1}{q}=1,\qquad
    1+\frac{1}{q}=\frac{1}{p}+\frac{1}{r}.
\end{align*}
Then by using the Hölder inequality and the Young convolution inequality we can estimate the Fisher information, defined by 
\begin{align*}
    \DD^2(\rho||\rho_\infty)=\int_{\Rd} (F+G)\cdot \;  k \ast (F+G) \diff x,
    \quad
    \begin{split}
        F(x)&:=(1+|x|^\beta) f(x)\nabla \rho(x),\\
    G(x)&:=(1+|x|^\beta) f(x)\rho(x)\nabla V(x),
    \end{split}
\end{align*}
from above by the following term
\begin{align}
    \label{thm:failure of slsi for integrable convolution kernels:eq1}
    2\,\|f\|_{L^{\infty}(\R^d)}^2 \, \|k\|_{L^r(\R^d)}\,   \left(\|(1+|x|^\beta)\nabla \rho\|^2_{L^p(\R^d)} + \|(1+|x|^\beta)\rho\, \nabla V\|^2_{L^p(\R^d)} \right). 
\end{align}
Define $V(x)=|x|^2$ and the probability measure $\rho\in \Pro$ by
\begin{align*}
    \rho\,\propto\, \frac{1}{1+|x|^{d+2}}.
\end{align*}
Furthermore, we note that the Kullback-Leibler divergence can be expanded to
\begin{align*}
    \KL(\rho||\rho_\infty)
    =\int_{\Rd}\rho\, \log(\rho)\diff x+\int_{\Rd} \rho\,V\diff x-Z,
\end{align*}
for the normalization constant $Z:=\int_{\Rd}e^{-V}\diff x$. One can show that there exists a constant $C>0$ such that
\begin{align*}
    |(1+|x|^\beta)\nabla \rho(x)|^p 
    &\leq C \left(\frac{1}{1+|x|}\right)^{p(d+3-\beta)},\\
    |(1+|x|^\beta)\rho(x)\, \nabla V(x)|^p 
    &\leq C \left(\frac{1}{1+|x|}\right)^{p(d+1-\beta)}.
\end{align*}
One can show that if $\beta <d+1-\frac{d}{p}$ then $(1+|x|^\beta)\nabla \rho\in L^p(\Rd)$ and $(1+|x|^\beta)\rho\, \nabla V\in L^p(\Rd)$. It is immediate to see that $\rho \, \log \rho \in L^1(\R^d)$. However,
$$
\int_{\R^d} \rho\, V \diff x = \int_{\R^d} \frac{|x|^2}{1+|x|^{d+2}} \diff x \geq \frac{1}{2} \int_{|x|>1} |x|^{-d} = \infty.
$$
Hence, this function acts as a counterexample to the inequality \eqref{eq:stein_log_sobolev_inequality_general_alpha} for all $\alpha\geq 1$. This concludes the proof.

\subsection{Proof of Theorem \ref{thm:conditions_for_slsi}, Case \ref{thm:failure_of_slsi_for_homogeneous_convolution_kernels}}
    For any scalar $\lambda>0$, define the dilation function $S_\lambda(x):=\lambda x$ for $x\in\Rd$. We have $\diff(S_\lambda)_\#\rho_\infty(x)=h(x)\diff\rho_\infty(x)$ where
    \begin{equation*}
        h(x)
        =\frac{e^{-V(x/\lambda)}}{e^{-V(x)} \lambda^d}.
    \end{equation*}
    We compute the Kullback-Leibler divergence of $(S_\lambda)_\#\rho_\infty$
    \begin{align*}
        \KL((S_\lambda)_\#\rho_\infty || \rho_\infty)
        =\int_{\Rd} V(\lambda x)-V(x)\diff \rho_\infty(x)
        -d\ln(\lambda).
    \end{align*}
    We also evaluate the dissipation in \eqref{eq:stein_fisher_information} as
    \begin{align*}
        &\frac{\lambda^{-2d}}{Z^2}\int_{\Rd}\int_{\Rd} \nabla\big(e^{V(x)-V(x/\lambda)}\big)\cdot k(x-y)\nabla\big(e^{V(y)-V(y/\lambda)}\big)\;e^{-V(x)}e^{-V(y)}\diff y \diff x \\
        &\qquad=\frac{\lambda^{-2d}}{Z^2}\int_{\Rd}\int_{\Rd} \left(\nabla V(x)-\frac{1}{\lambda}\nabla V\left(\frac{x}{\lambda}\right)\right)\\
        &\qquad\qquad\qquad\qquad\qquad\qquad\cdot k(x-y)\left(\nabla V(y)-\frac{1}{\lambda}\nabla V\left(\frac{y}{\lambda}\right)\right) e^{-V(x/\lambda)}e^{-V(y/\lambda)}\diff y\diff x.
    \end{align*}
    First, we use the substitution $\tilde{x}=x/\lambda$ and $\tilde{y}=y/\lambda$, and then the fact that $\nabla V$ is $(\gamma-1)$-homogeneous and that $k$ is $\beta$-homogeneous to rewrite the integral as
    \begin{align*}
        &\frac{1}{Z^2}\int_{\Rd}\int_{\Rd} \left(\nabla V(\lambda x)-\frac{1}{\lambda}\nabla V\left(x\right)\right)\cdot k(\lambda(x-y))\left(\nabla V(\lambda y)-\frac{1}{\lambda}\nabla V\left(y\right)\right) e^{-V(x)}e^{-V(y)}\diff y\diff x\\
        &\qquad=(\lambda^\gamma-1)^2\lambda^{\beta-2}\int_{\Rd}\int_{\Rd} \nabla V(x)\cdot k(x-y)\nabla V(y)\diff \rho_\infty(y)\diff \rho_\infty(x).
    \end{align*}
    Suppose for a contradiction there exists a constant $C >0$, for which \eqref{eq:stein_log_sobolev_inequality_general_alpha} holds. Combining both representations and using homogeneity, we derive the bound
    \begin{align*}
        &\left[(\lambda^\gamma-1)\int_{\Rd} V(x)\diff\rho_\infty(x)-d\ln(\lambda)\right]^\alpha\\
        &\qquad\leq C (\lambda^\gamma-1)^2\lambda^{\beta-2}\int_{\Rd}\int_{\Rd} \nabla V(x)\cdot k(x-y)\nabla V(y)\diff \rho_\infty(y)\diff\rho_\infty(x).
    \end{align*}
    When $\lambda \to \infty$, the LHS has growth $\alpha\gamma $, while RHS has growth $\beta+2(\gamma-1)$ so we deduce the inequality $\alpha\gamma \leq \beta+2 (\gamma-1)$. We see we have failure if $\alpha\geq 2$, and otherwise we have the admissible set
    \begin{align*}
        \left\{\left.(\beta,\gamma)\in (-\infty,0)\times \left(\frac{2}{2-\alpha},\,\infty\right)\;\right\vert\; 2+(\alpha-2)\gamma\leq \beta<0\right\}.
    \end{align*}

\begin{Ex}
    To illustrate the applicability of Theorem \ref{thm:conditions_for_slsi}, Case \ref{thm:failure_of_slsi_for_homogeneous_convolution_kernels}, we take $V=|x|^\alpha$ and $k(z)=|z|^{-s}$ for any $\alpha>0$ and $s\in (0,d)$. Then to show that the RHS of \eqref{eq:homogoneity_equation} is finite, we use the Hölder inequality and the Hardy-Littlewood-Sobolev theorem \cite[Theorem 6.1.3]{grafakos_modern_2009} to estimate
    \begin{align*}
        \int_{\Rd}\int_{\Rd} \nabla V(x)\cdot k(x-y)\nabla V(y)\diff \rho_\infty(y)\diff\rho_\infty(x)
        &\leq \norm{\nabla V e^{-V}}_{L^p(\Rd)}\norm{k\ast \nabla V e^{-V}}_{L^q(\Rd)}\\
        &\leq C \norm{\nabla V e^{-V}}_{L^p(\Rd)}^2
        <\infty,
    \end{align*}
    where $p\in (1,2)$ and $q\in (2,\infty)$ with $\frac{1}{p}+\frac{1}{q}=1$ and $\frac{1}{p}-\frac{1}{q}=\frac{s}{d}$ and $C>0$ is some constant depending on $p$, $s$ and $d$.
\end{Ex}

\appendix 

\section{The incomplete gamma functions}\label{appendix:incomplete_gamma_functions}

We collect here several properties of the incomplete gamma functions which are used in Section \ref{section:existence_kernels_satisfying_LSS_inequality}. Given $s > 0$ and $r \geq 0$, we define the upper incomplete gamma function as 
$$
\Gamma(s,r) = \int_r^{\infty} z^{s-1} \, e^{-z} \diff z 
$$
and the lower one as
\begin{equation}\label{eq:def_incomplete_gamma_lower}
\gamma(s,r) = \int_0^r z^{s-1}\, e^{-z} \diff z.  
\end{equation}

\begin{lem}
    \label{Lemma:kernel frequency of constant is monotone increasing}
    Let $f(r) = r^{-s}\, e^{r}\, \gamma(s, r)$ and $g(r) = r^{-s}\, e^{r}\, \Gamma(s,r)$. Then, $f$ is increasing while $g$ is decreasing. Moreover, $\lim_{r\to \infty} g(r) = 0$.
\end{lem}
\begin{proof}
We first focus on the function $f$. We compute the derivative and find
    \begin{align*}
        f'(r)
        =\left(1-\frac{s}{r}\right)r^{-s}e^r \,\gamma\left(s, r\right)
        +r^{-1}.
    \end{align*}
    We show the derivative is positive. Clearly, we only need to study the case $r \in (0,s]$. Hence, we must show for $r \in (0,s]$
    \begin{align*}
        0 < \left(r-s\right)e^r\gamma\left(s, r\right)
        +r^{s} \iff \gamma\left(s, r\right)
        < \frac{r^{s}}{s-r}e^{-r}.
    \end{align*}
    Since we have equality at $r=0$, it suffices to compare derivatives. A direct computation leads to the inequality
    \begin{align*}
        r^{s-1}e^{-r}
        <  r^{s-1} e^{-r} \frac{(s-r)^2 + r}{(s-r)^2}
        = e^{-r} \frac{s(s-r)r^{s-1} + r^s }{(s-r)^2} - e^{-r} \frac{r^s}{s-r},
    \end{align*}
    which is clearly satisfied and so, the proof is concluded.\\

    Concerning function $g$, we again compute the derivative
    $$
    g'(r) = \left(1-\frac{s}{r}\right)r^{-s}e^r \, \Gamma\left(s, r\right)
        -r^{-1}
    $$
    and we have to show that $g'(r) < 0$. The inequality $g'(r) <  0$ is equivalent with
    $$
   0 < r^{s} e^{-r} -  \left(r-s\right) \Gamma\left(s, r\right) =:\Phi(r).
    $$
    By a direct computation
    $$
    \Phi'(r) = s\, r^{s-1}e^{-r} - r^{s}e^{-r} - \Gamma\left(s, r\right) + \left(r-s\right) r^{s-1}e^{-r}
    =-\Gamma(s,r)< 0.
    $$
    Furthermore, $\lim_{r \to \infty} \Phi(r) = 0$ (this can be proved by estimating $\Gamma\left(s, r\right) \leq C\,e^{-r/2}$, where $C= \int_0^{\infty} z^{s-1}e^{-z/2}\diff z$) so that $\Phi(r) > 0$ and the proof of monotonicity of $g$ is concluded. Finally, the limit $\lim_{r\to \infty} g(r) = 0$ follows from Corollary \ref{cor:the_asymptotic_beh_Gamma_fcn} below.
\end{proof}

We also recall a well-known asymptotic expansion of the upper incomplete gamma function.
\begin{lem}\label{lem:expansion_series_Gamma}
Let $n \in \N$ be such that $n \geq s-2$. Then 
\begin{multline*}
\Gamma(s,r) = e^{-r}\,r^{s-1}\Big(1 + {(s-1)}\,{r^{-1}} + {(s-1)(s-2)}\,{r^{-2}} \, + \\ + 
\ldots + {(s-1) \ldots (s-n)}\,{r^{-n}} \Big) + \epsilon_n(r),
\end{multline*}
where
$$
|\epsilon_n(r)| \leq |(s-1) \ldots (s-n-1)| \, e^{-r}\,{r^{-n+s-2}}. 
$$
\end{lem}
For a simple proof, based on iterating the identity 
$$
\Gamma(s,r) = e^{-r} r^{s-1} + (s-1)\,\Gamma(s-1,r),
$$
we refer to \cite[Chapter 3]{MR1429619}. The expansion implies in particular the following.
\begin{cor}\label{cor:the_asymptotic_beh_Gamma_fcn}
It holds $\lim_{r\to \infty} e^r r^{1-s} \Gamma(s,r)=1$.    
\end{cor}
In Section \ref{section:existence_kernels_satisfying_LSS_inequality}, we will need another expansion formulated below.

\begin{lem}\label{lem:expansion_into_matern_kernels}
Let $k \geq s-1$ and $r \geq 1$. Then, there exists constants $C_1$, ..., $C_{k+1}\in\R$ and a function $\varepsilon_k:[1,\infty)\ra \R$, both depending on $k$ and $s$, such that
$$
e^r r^{-s} \Gamma(s,r) = {C_1}\,{(1+r)^{-1}} + {C_2}\,{(1+r)^{-2}} + \ldots + {C_k}\,{(1+r)^{-k}} + \varepsilon_k(r),
$$
where $|\varepsilon_k(r)| \leq C_{k+1}\, r^{-(k+1)}$.
\end{lem}
\begin{proof}
Using Lemma \ref{lem:expansion_series_Gamma} with $n = k-1$ to be chosen later 
\begin{multline*}
e^r r^{-s} \Gamma(s,r) =r^{-1} + {(s-1)}\,{r^{-2}} + {(s-1)(s-2)}\,{r^{-3}} \, + \\ +
\ldots + {(s-1) \ldots (s-n)}\,{r^{-n-1}} + \epsilon_n(r)\, e^r r^{-s},
\end{multline*}
where $\epsilon_n(r)\, e^r r^{-s} \leq C\,r^{-n-2} = C\,r^{-(k+1)}$ for some constant $C>0$. Now, for each exponent $m$, we write
$$
r^{-m} = (r+1)^{-m} + \Big(r^m- (r+1)^{m}\Big)\,r^{-m}\, (r+1)^{-m} =  (r+1)^{-m} + (r+1)^{-m} \sum_{l=1}^m C^{m}_l r^{-l},
$$
for some constants $C_1^m,\ldots,C_m^m\in\R$. This formula allows to transform coefficient $r^{-m}$ into two terms: one of which is of the correct form and the another is one with order strictly larger than $m$ (with respect to $r$). Hence, after finitely many iterations we arrive at the claim.
\end{proof}

\begin{lem}\label{lem:Gamma_asymptotics_derivative}
We have 
$$
\partial_r (e^r \, r^{-s}\, \Gamma(s,r))= e^r \, r^{-s}\, \Gamma(s,r) + (-s)\, e^r \, r^{-s-1}\, \Gamma(s,r) - r^{-1}.
$$
Moreover, $\lim_{r \to \infty}  r^{\frac{1}{2}}\, |\partial_r (e^r \, r^{-s}\, \Gamma(s,r))|=0$.
\end{lem}
\begin{proof}
The formula follows by a direct computation. To prove the decay, we will only need to study the first term $ e^r \, r^{-s+\frac{1}{2}}\, \Gamma(s,r)$ as the second has even better decay and the third is clear. By Corollary \ref{cor:the_asymptotic_beh_Gamma_fcn}
$$
\lim_{r\to \infty} e^r \, r^{-s+\frac{1}{2}}\, \Gamma(s,r) = \lim_{r\to \infty} r^{-1/2}\, e^r \, r^{-s+1}\, \Gamma(s,r) = 0.
$$ 
\end{proof}

\section{Properties of the $H^{-s}(\Rd)$ space}
\label{appendix:prop_H-1}
We collect here several facts related to the space $H^{-s}(\Rd)$ defined in \eqref{eq:definition_H-1}.

\begin{lem}\label{lem:Plancherel_f_H-1}
    Let $k \in L^1(\Rd) + L^2(\Rd)$ be a function such that $|\hat{k}(\xi)| \leq \frac{C}{(1+|\xi|^2)^s}$ for $s\geq 1$. 
    \begin{enumerate}[label=(P\arabic*)]
        \item \label{lem:Plancherel_f_H-1:wd} The map $f \mapsto k \ast f$, defined on $\mathcal{S}(\Rd)$, extends to a unique bounded linear map from $H^{-s}(\Rd)$ to $H^s(\Rd)$ with the estimate $\|k\ast f\|_{H^s(\Rd)} \leq C\, \|f\|_{H^{-s}(\Rd)}$.
        \item \label{lem:Plancherel_f_H-1:plancherel} We have the following variant of the Plancherel theorem: for all $f, g \in H^{-s}(\Rd)$
        $$
        \langle g, k\ast f \rangle_{H^{-s}, H^s} = \langle \hat{g},\,\hat{k}\,  \hat{f}\rangle,
        $$
        where $\inner{\cdot}{\cdot}$ is the complex $L^2(\Rd)$ inner product.
    \end{enumerate}
\end{lem}
\begin{proof}
\underline{\it Property \ref{lem:Plancherel_f_H-1:wd}.}
Let $f \in \mathcal{S}(\Rd)$. Then, 
\begin{align*}
    \|k\ast f\|_{H^s(\Rd)}^2 &= \int_{\Rd} (1+|\xi|^2)^s\, |\hat{k}(\xi)|^2 \, |\hat{f}(\xi)|^2 \diff \xi\\ 
    &\leq C^2\, \int_{\Rd} \frac{(1+|\xi|^2)^s}{(1+|\xi|^2)^{2s}}\cdot |\hat{f}|^2 \diff \xi = C^2\, \|f\|_{H^{-s}(\Rd)}^2
\end{align*}
for all $f \in \mathcal{S}(\Rd)$. By completeness of $H^{-s}(\Rd)$ and a density argument, there exists a unique extension as desired. 

\underline{\it Property \ref{lem:Plancherel_f_H-1:plancherel}.} Let $\{f_n\}_{n\geq 1}, \{g_n\}_{n\geq 1} \subset \mathcal{S}(\Rd)$ be two sequences such that $g_n \to g$, $f_n \to f$ in $H^{-s}(\Rd)$. Then, by the Plancherel theorem
\begin{multline*}
\langle g_n, k\ast f_n \rangle_{H^{-s}, H^s} = \int_{\Rd} g_n\cdot \, k \ast f_n \diff x = \langle\widehat{g}_n,\,\widehat{f}_n \, \hat{k}\rangle= \\ =
\inner{(1+|\xi|^2)^{-\frac{s}{2}} \, \widehat{g}_n} { (1+|\xi|^2)^{-\frac{s}{2}} \, \widehat{f}_n\, (1+|\xi|^2)^s\, \hat{k}}.
\end{multline*}
By the definition of the $H^{-s}$ norm
$$
(1+|\xi|^2)^{-\frac{s}{2}} \, \widehat{g}_n(\xi)  \to (1+|\xi|^2)^{-\frac{s}{2}} \, \widehat{g}(\xi) \mbox{ strongly in } L^2(\Rd)
$$
and similarly for the sequence $\{(1+|\xi|^2)^{-\frac{s}{2}} \, \widehat{f}_n(\xi)\}_{n\geq 1}$. As $(1+|\xi|^2)^s\, \hat{k}(\xi)$ is a bounded function, we obtain
$$
\lim_{n\to \infty} \langle g_n, k\ast f_n \rangle_{H^{-s}, H^s} = \inner{\hat{g}}{\hat{k}\, \hat{f}}.
$$
On the other hand, by the first part $k \ast f_n \to k \ast f$ strongly in $H^s(\Rd)$ so that 
$$
\lim_{n\to \infty} \langle g_n, k\ast f_n \rangle_{H^{-s}, H^s} =  \langle g, k\ast f \rangle_{H^{-s}, H^s}
$$
and the proof is concluded. 
\end{proof}

\begin{lem}\label{lem:representation_derivative_f_x_fourier} 
Let $f \in L^2(\Rd)$ and let $g(x)= A(x-\mu)\, f(x) \in H^{-1}(\Rd)$ with $\mu \in \Rd$ and $A \in \R^{d\times d}$ an invertible matrix. Then, $\nabla \hat{f}$ exists in the Sobolev sense and $\nabla \hat{f} \in L^2_{\text{loc}}(\Rd)$ with $\int_{\Rd} \frac{|\nabla \hat{f}(\xi)|^2}{1+|\xi|^2} \diff \xi < \infty$.     
\end{lem}
\begin{proof}
Let us reduce to the case $\mu = 0$ and $A = \id$. First, since $A \mu f(x) \in L^2(\Rd)$ and $g\in H^{-1}(\Rd)$, we get $Ax f(x) \in H^{-1}(\Rd)$. Second, the composition $A^{-1}$ with $Ax f(x)$ stays in $H^{-1}(\Rd)$ so we deduce $x f(x) \in H^{-1}(\Rd)$. Now, let $h_j = f x_j$. We know from \eqref{eq:definition_H-1} that $\hat{h}_j$ exists as a function and not only as a distribution. We compute for $\varphi \in C_c^{\infty}(\Rd)$
\begin{equation}\label{eq:computation_id_fourier_of_g}
\int_{\Rd} \widehat{h}_j(\xi)\, \varphi(\xi) \diff \xi = \int_{\Rd} h_j(\xi)\, \check{\varphi}(\xi) \diff \xi = \int_{\Rd} \xi_j\, f(\xi)\, \check{\varphi}(\xi) \diff \xi.
\end{equation}
Furthermore,
$$
\xi_j \, \check{\varphi}(\xi) = \frac{1}{2\pi i} \int_{\Rd} \varphi(x) \, \partial_{x_j} e^{2\pi i \xi \cdot x} \diff x = -\frac{1}{2\pi i} \widecheck{\partial_j \varphi}(\xi).
$$
Therefore, continuing the computation in \eqref{eq:computation_id_fourier_of_g},
$$
-2\pi i \int_{\Rd} \widehat{h}_j(\xi)\, \varphi(\xi) \diff \xi= \int_{\Rd} f(\xi)\,\widecheck{\partial_j \varphi}(\xi) \diff \xi = \int_{\Rd} \hat{f}(\xi)\, \partial_j \varphi(\xi) \diff \xi.
$$
Therefore, $2\pi i \, \hat{h}_j$ is a Sobolev derivative $\partial_j \widehat{f}$. The desired integrability of $\partial_j \widehat{f}$ follows from the integrability of $\hat{h}_j$.
\end{proof}

An example of function $k$ satisfying Lemma \ref{lem:Plancherel_f_H-1} is a fundamental solution to the PDE \eqref{eq:screened_poisson_equation} whose regularity properties are listed below. 
\begin{lem}
\label{lem:fundamental_solution_to_screened_poisson_equation}
    There exists function $\spk \geq 0$ solving (in the sense of distributions)
    \begin{align} \label{eq:screened_poisson_equation}
        -\Delta \spk+\spk=\delta_0.
    \end{align}
    Moreover, $\spk$ enjoys the following regularity:
    \begin{itemize}
        \item if $d=1$ then $\spk, \nabla \spk \in L^p(\Rd)$ for all $p \in [1,\infty]$,
        \item if $d=2$ then $\spk \in L^p(\Rd)$ for all $p \in [1,\infty)$ and $\nabla \spk \in L^q(\Rd)$ for all $q \in  [1, 2)$,
        \item if $d > 2$ then $\spk \in L^p(\Rd)$ for $p \in [1, \frac{d}{d-2})$ and $\nabla \spk \in L^q(\Rd)$ for $q \in [1, \frac{d}{d-1})$. 
    \end{itemize}. 
\end{lem}
\begin{proof}
This result is well-known in the literature (see, e.g. \cite[Section 2(a)]{Perthame-incompressible-visco}) but for convenience we provide a sketch using \cite{MR143935}. The results could be also formulated in a sharper form by using the setting of Marcinkiewicz spaces as in \cite[Appendix A]{MR390473}. \\

We write $\spk_d$ for the solution to \eqref{eq:screened_poisson_equation} to stress the dependence on the dimension $d$ as it will be useful in the sequel. It is known from \cite[Ch. 2, eq. (4.1)]{MR143935} that $\spk_d$ is an analytic function of $|x|$ except at $x = 0$. From \cite[Ch. 2, eq. (4.2)]{MR143935} there exists a constant $C_0>0$ depending on $d$ such that
\begin{equation}\label{eq:asymptotics_around_0_spkd}
\spk_d(x) \sim \begin{cases} 
C_0 &\mbox{ if } d= 1,\\
-C_0\, \log(|x|)  &\mbox{ if } d= 2,\\
C_0\, |x|^{2-d} &\mbox{ if } d> 2,
\end{cases} 
\mbox{ as } x \to 0
\end{equation}
where the notation $f(x) \sim g(x)$ as $x \to 0$ means $\lim_{x\to0} \frac{f(x)}{g(x)} = 1$. Similarly, by \cite[Ch. 2, eq. (4.5)]{MR143935}, there exists another constant $C_{\infty}>0$ depending on $d$ such that
\begin{align}
    \label{eq:behaviour_spk_infinity}
    \spk_d(x) \sim C_{\infty}\, |x|^{\frac{1-d}{2}}\, e^{-|x|} \mbox{ as } |x| \to \infty.
\end{align}
Hence, to check the integrability of $\spk_d(x)$ it is sufficient to control the $L^p$ norm of $\spk$ on $B_1$. The result is clear for $d=1, 2$ and for $d >2$ we note that
$$
\int_{B_1} |x|^{(2-d)\,p}\diff x <\infty \iff  \int_0^1 r^{(2-d) p} \, r^{d-1} \diff r < \infty \iff  p < \frac{d}{d-2}.
$$
Concerning the estimates on $\nabla \spk_d$, we observe from \cite[Ch. 2, eq. (4.4)]{MR143935} that there exists a constant $C>0$ depending on $d$ such that
$$
|\nabla \spk_d(x)| = C\,|x| \, \spk_{d+2}(x) \sim C\, C_0(d+2) \, |x|^{1-d} \mbox{ as } x\to 0,
$$
where in the second step we used \eqref{eq:asymptotics_around_0_spkd}. 
By \eqref{eq:behaviour_spk_infinity} it is again sufficient to control $L^p$ norm of $\nabla \spk$ on $B_1$. Hence, the assertion is again clear for $d=1$ while for $d\geq 2$ we compute
$$
\int_{B_1} |x|^{(1-d)\,q} \diff x <\infty \iff \int_0^1 r^{(1-d)\,q} \, r^{d-1} \diff r \iff q < \frac{d}{d-1}.
$$
\end{proof}

\begin{lem}\label{lem:lem_compactness_measures_H-m} 
    Let $m> \frac{d}{2}$. Then, $\mathcal{M}(B_R)$ with the total variation norm $\|\cdot\|_{TV}$ is compactly embedded in $H^{-m}(B_R)$. In particular, $L^1(B_R)$ is compactly embedded in $H^{-m}(B_R)$.
\end{lem}
\begin{proof}
    The proof is a simple adaptation of \cite[Theorem 6, Chapter 1]{MR1034481}. We start with a general observation. Let $\mathcal{B}:= \{ f \in H^m_0(B_R): \|f\|_{H^m_0(B_R)} \leq 1\}$ and note that $\mathcal{B}$ is compactly embedded in $C^0(B_R)$ with the supremum norm. Hence, for each $\varepsilon>0$, there exists $N\in\N$ and functions $f_1, ..., f_N \in C^0(B_R)$ such that for each $f \in \mathcal{B}$ there exists $i\in \{1,\ldots,N\}$ such that $\|f-f_i\|_{\infty}\leq \varepsilon$.

    \smallskip

    Let $\varepsilon>0$ and let $\{\mu_n\}_{n\geq 1}$ be a bounded sequence of measures. We have to prove that $\{\mu_n\}_{n\geq 1}$ has a subsequence converging strongly in $H^{-m}(B_R)$. Clearly, it has a subsequence (not relabelled) converging weakly* to $\mu$ in $\mathcal{M}(B_R)$. For this subsequence and $f \in \mathcal{B}$ we have
    \begin{align*}
    \left|\int_{B_R} f(x) \diff (\mu_n - \mu)(x) \right| &\leq  \left|\int_{B_R} f_i(x) \diff (\mu_n - \mu)(x) \right| + \left|\int_{B_R} (f-f_i(x)) \diff (\mu_n - \mu)(x) \right| \\
    &\leq \sup_{j=1,..., N} \left|\int_{B_R} f_i(x) \diff (\mu_n - \mu)(x) \right| + 2\,C\,\varepsilon,
    \end{align*}
    where $C = \sup_{n\geq 1} \|\mu_n\|_{TV}$ so that taking supremum over $f \in \mathcal{B}$ and taking $\limsup_{n \to \infty}$ we get
    $$
    \limsup_{n \to \infty} \| \mu_n - \mu\|_{H^{-m}(B_R)} \leq 2\,C\,\eps. 
    $$
    As $\varepsilon$ is arbitrary, the proof is concluded. For the case of $L^1(B_R)$, this follows easily because $L^1(B_R)$ embedds continuously into $(\mathcal{M}(B_R), \|\cdot \|_{TV})$.
\end{proof}

\section{Proof of existence of solutions to the approximate problem} \label{appendix:existence_approximations_former_appendix}
This Appendix is devoted to proving the following: 

\begin{lem}   \label{lemma:existence_of_classical_solutions_to_step_2}
    Under the assumptions of Theorem \ref{thm:existence_of_distributional_solutions_to_mfsvgd}, there exists a unique classical solution to the PDE \eqref{eq:regularised_weak_pde}--\eqref{ass:mfsvgd:ini} on $(0,\infty)\times B_R$. In particular, for all $T>0$, there exists $\alpha\in (0,1)$ such that $\rho\in C^{1+\frac{\alpha}{2}, 2+\alpha}_{t,x}([0,T]\times B_R)$.
\end{lem}
Let us explain the strategy. We assume that the parameters $\eta\in (0,1)$ and $R>0$ are fixed and that any PDE in this section will satisfy the boundary conditions \eqref{eq:Dirichlet_approximation_sys} and \eqref{ass:mfsvgd:ini}. We start with the PDE on $(0,T)\times B_R$
\begin{equation}\label{eq:regularised_classical_pde}
		\prt_t \rho=\divv\Big(\rho\; \vp_\tau \ast\omega_\eta\ast\Big[e^{V_\eta-\frac{\pot}{2}}\; k\ast \Big(\nabla \left( \beta \ast \omega_{\eta}\, e^{V_\eta} \right) e^{-\frac{\pot}{2}}\Big)\Big) \Big] \Big) +\frac{1}{R}\,\divv\left(\nabla \rho+\rho\nabla V_\eta\right).
\end{equation}
Here, $\beta\in L^1((0,T)\times\Rd)$ and $\{\varphi_{\tau}\}_{\tau>0}$ is the usual temporal mollification kernel $\vp_\tau(t):=\frac{1}{\tau}\tilde \vp\left(\frac{t}{\tau}\right)$ where $\tilde \vp$ is a~function such that 
$$
\tilde\vp>0 \text{ on }(-1,0),
\qquad 
\tilde\vp=0\text{ on }\R\setminus(-1,0),
\qquad 
\int_{\R}\tilde\vp\diff t=1,
\qquad 
\tilde\vp\in C_c^\infty(\R).
$$
To make the convolutions in time well-defined, we always extend $\beta(t,x)=\rho(t,x) = \rho_0^{R,\eta}(x)$ for $t < 0$. The existence and uniqueness of solutions $\rho\in C^\infty([0, T] \times \overline{B_R})$ to \eqref{eq:regularised_classical_pde} follows by a standard parabolic PDE theory (see \cite[Chapter~7]{MR2597943}). Moreover, $\rho \geq 0$ by the weak maximum principle \cite[Chapter 7, Problem 7]{MR2597943}, the boundary condition \eqref{eq:Dirichlet_approximation_sys} and the fact that the following function is bounded
    \begin{align*}
        \divv\left(\vp_\tau\ast\omega_\eta\ast\Big[e^{V_\eta-\frac{\pot}{2}}\; k\ast \Big(\nabla(\beta \ast \omega_\eta)+(\beta\ast\omega_\eta)\nabla V_\eta\Big)e^{V_\eta-\frac{\pot}{2}}\Big]\right)
        +\frac{1}{R}\Delta V_\eta.
    \end{align*}

\smallskip 

Exploiting the Leray-Schauder fixed-point theorem, we prove existence of classical solutions to
\begin{equation}
\label{eq:regularised_classical_pde_eta=rho}
		\prt_t \rho=\divv\Big(\rho\; \vp_\tau \ast\omega_\eta\ast\Big[e^{V_\eta-\frac{\pot}{2}}\; k\ast \Big(\nabla \left( \rho \ast \omega_{\eta}\, e^{V_\eta} \right) e^{-\frac{\pot}{2}}\Big)\Big) \Big] \Big) +\frac{1}{R}\,\divv\left(\nabla \rho+\rho\nabla V_\eta\right).
\end{equation}
Next, we obtain a solution to \eqref{eq:regularised_weak_pde} by passing to the limit $\tau \to 0$ in~\eqref{eq:regularised_classical_pde_eta=rho}. This solution is then found to be classical through the usual bootstrapping arguments.

\smallskip

To simplify notation, we define the maps $E_{\tau,\eta}:L^\infty(0,T;L^1(\Rd))\ra C^\infty([0,T]\times \overline{B_R};\,\Rd)$ and $F_{\eta}:L^\infty(0,T; L^1(\Rd))\ra L^1(0,T; C^\infty(\overline{B_R};\,\Rd))$ given by
\begin{align}
    F_\eta(\beta)
    &:= \omega_\eta\ast\Big[e^{V_\eta-\frac{\pot}{2}}\; k\ast \Big(\nabla(\beta \ast \omega_\eta)+(\beta\ast\omega_\eta)\nabla V_\eta\Big)e^{V_\eta-\frac{\pot}{2}}\Big], \label{eq:formula_for_Feta}\\
    E_{\tau,\eta}(\beta)
    &:=\vp_\tau \ast F_\eta(\beta), \notag
\end{align}
The vector fields $E_{\tau,\eta}$ and $F_\eta$ depend continuously on its arguments by the following lemma. We quickly remark that the constants in the rest of this appendix will depend on $R$ and $\eta$, but we will not track them explicitly.

\begin{lem}
    \label{lem:moll_operator_are_cts}
    The operators $E_{\tau,\eta}$ and $F_\eta$ are well-defined. Moreover, there exists a constant $C>0$ independent of $\tau$ such that for all $\beta \in L^1((0,T)\times\Rd)$
    \begin{align*}
        \norm{E_{\tau,\eta}(\beta)}_{L_{t,x}^\infty}
        +\norm{F_\eta(\beta)}_{L_{t,x}^\infty}
        \leq C\norm{\beta}_{L^{\infty}_t L^1_{x}},\qquad
        \norm{E_{\tau,\eta}(\beta)}_{L_{t,x}^\infty}
        \leq \frac{C}{\tau}\norm{\beta}_{L_{t,x}^1}.
    \end{align*}
\end{lem}
\begin{proof}
    Follows immediately by repeatedly applying the Young convolution inequality.
\end{proof}

To apply the fixed-point argument, we define the map $S:L^1((0,T)\times \Rd )\ra L^1((0,T)\times \Rd )$, such that $\rho=S(\beta)$ is the solution of \eqref{eq:regularised_classical_pde}. First, let us state some preliminary estimates.

\begin{lem}[{Uniform estimates}]
    \label{lemma:H1 estimate of MF SVGD step 1}
    For any value $m>\frac{d}{2}$, there exists a constant $C>0$ independent of $\tau$ such that for any $\beta\in L^1((0,T)\times \Rd)$, the function $\rho=S(\beta)$ satisfies
    \begin{gather}
        \label{lemma:H1 estimate of MF SVGD step 1:eq1}
        \int_0^T \int_{B_R}\frac{|\nabla \rho+\rho\nabla V_\eta|^2}{\rho}\diff x\diff t
        \leq C\left(R\cdot\KL(\rho_0^{R,\eta}||\rho_\infty^{\eta})+T\,R^2\,\norm{\beta}_{\ast}^2\right),\\
        \label{lemma:H1 estimate of MF SVGD step 1:eq3}
        \norm{\sqrt{\rho}}^2_{L_t^2 H_x^1}
        \leq C\left(R\cdot\KL(\rho_0^{R,\eta}||\rho_\infty^{\eta})+T \norm{V}_{L^\infty(B_R)}^2+T\,R^2\,\norm{\beta}_{\ast}\right),\\
        \label{lemma:H1 estimate of MF SVGD step 1:eq2}
        \norm{\prt_t \rho}_{L_t^2 H_x^{-m}}
        \leq C\Big(\KL(\rho_0^{R,\eta}||\rho_\infty^{\eta})+R^{d-1}+T \norm{\beta}_{\ast} \Big).
    \end{gather}
    where we define $\norm{\beta}_{\ast}:= \min\Big(\norm{\beta}_{L_t^\infty L_x^1},  \frac{1}{\tau}\norm{\beta}_{L_{t,x}^1} \Big)$.
\end{lem}
\begin{proof}
    First we prove the estimate \eqref{lemma:H1 estimate of MF SVGD step 1:eq1}. We multiply the PDE in \eqref{eq:regularised_classical_pde} with $\ln\big(\rho e^{V_\eta}\big)$, integrate in space on $B_R$ and apply integration by parts with Lemma \ref{lem:chng_var} (passing $\kappa\ra 0$) to derive
    \begin{align*}
        \prt_t \int_{B_R} \rho \ln(\rho e^{V_{\eta}})\diff x
        +\int_{B_R} E_{\tau,\eta}(\beta)\cdot (\nabla \rho+\rho \nabla V_\eta)\diff x
        +\frac{1}{R}\int_{B_R}\frac{|\nabla \rho+\rho\nabla V_\eta|^2}{\rho}\diff x=0.
    \end{align*}
    Using Cauchy-Schwarz and Lemma \ref{lem:moll_operator_are_cts} we can bound the middle term by
    \begin{align*}
        \left\vert\int_{B_R} E_{\tau,\eta}(\beta)\cdot (\nabla \rho+\rho \nabla V_\eta)\diff x\right\vert &\leq \frac{1}{2R} \int_{B_R} \frac{|\nabla \rho+\rho \nabla V_\eta|^2}{\rho}\diff x + \frac{R}{2} \, \int_{B_R} \rho\, |E_{\tau,\eta}(\beta)|^2\diff x\\
        &\leq \frac{1}{2R} \, \int_{B_R} \frac{|\nabla \rho+\rho \nabla V_\eta|^2}{\rho}\diff x +\frac{R\,C^2}{2} \norm{\beta}_{\ast}^2,
    \end{align*}
    where $C>0$ is the constant from Lemma \ref{lem:moll_operator_are_cts}. After integrating in time, we get the bound \eqref{lemma:H1 estimate of MF SVGD step 1:eq1}. The estimate \eqref{lemma:H1 estimate of MF SVGD step 1:eq3} is an immediate consequence of \eqref{lemma:H1 estimate of MF SVGD step 1:eq1}. Finally, we prove \eqref{lemma:H1 estimate of MF SVGD step 1:eq2}. Let $g\in L^2(0,T; W^{1,\infty}(B_R))$, then using Lemma \ref{lem:moll_operator_are_cts}, Cauchy-Schwarz and the Young inequality we derive 
	\begin{align*}
		\bigg\vert\int_0^T \int_{B_R} \prt_t \rho&\, g\diff x\diff t\bigg\vert\\
            &\leq \frac{1}{2R}\int_0^T \int_{B_R}  \frac{|\nabla \rho+\rho\nabla V_\eta|}{\sqrt{\rho}}\cdot \sqrt{\rho}|\nabla g|\diff x\diff t
		+\frac{1}{2}\int_0^T \int_{B_R} \rho |E_{\tau,\eta}(\beta)|\cdot|\nabla g|\diff x\diff t\\
            &\leq \frac{\sqrt{|B_R|}}{2}\left[
            \frac{1}{R} \left(\int_0^T \int_{B_R} \frac{|\nabla \rho+\rho\nabla V_\eta|^2}{\rho}\diff x\diff t\right)^{\frac{1}{2}}+C\sqrt{T}\norm{\beta}_{\ast}
            \right]\norm{g}_{L_t^2 W_x^{1,\infty}},
	\end{align*}
    where $C>0$ is the constant from Lemma \ref{lem:moll_operator_are_cts}. We use \eqref{lemma:H1 estimate of MF SVGD step 1:eq1} to get a uniform bound. The Sobolev embedding theorem gives us a continuous embedding $H^m(B_R)\hookrightarrow W^{1,\infty}(B_R)$. Hence, there is a constant $C$, such that $\norm{g}_{L_t^2 W_x^{1,\infty}} \leq C\, \norm{g}_{L_t^2 H_x^m}$. Taking the supremum over $g$ such that $\norm{g}_{L_t^2 H_x^m}\leq 1$ will yield the desired estimate \eqref{lemma:H1 estimate of MF SVGD step 1:eq2}.
\end{proof}

The next step is to use the Leray-Schauder fixed-point theorem to construct a solution to \eqref{eq:regularised_classical_pde_eta=rho}.

\begin{lem}[existence to \eqref{eq:regularised_classical_pde_eta=rho}]
    \label{lemma:existence_of_classical_solution_to_reg_pde}
    Suppose that $\rho_0$, $\rho_\infty$ and $V$ satisfy assumptions of Theorem~\ref{thm:existence_of_distributional_solutions_to_mfsvgd}. Then, there exists a classical solution $\rho\in C^\infty([0,T]\times \overline{B_R})$ to the PDE \eqref{eq:regularised_classical_pde_eta=rho} with conditions \eqref{eq:Dirichlet_approximation_sys}, \eqref{ass:mfsvgd:ini}.
\end{lem}

\begin{proof}
    The main idea of the proof is to apply the Leray-Schauder fixed-point theorem. Thus we divide the proof into four steps.

    \underline{\it $S$ is a compact operator.} Let $\{\beta_n\}_{n\geq 1}$ be a bounded sequence in $L^1((0,T)\times B_R)$. We define $\rho_n=S(\beta_n)$ and observe from Lemma \ref{lemma:H1 estimate of MF SVGD step 1} that $\norm{\sqrt{\rho_n}}_{L_t^2 H_x^1}$ is uniformly bounded. In particular, by writing $\nabla \rho_n = 2\,\nabla{\sqrt{\rho_n}} \, \sqrt{\rho_n}$, we see that $\norm{\nabla \rho_n}_{L_t^2 L_x^1}$ is uniformly bounded. Together with the fact $\norm{\prt_t {\rho_n}}_{L_t^2 H_x^{-m}}$ is uniformly bounded for $m>\frac{d}{2}$ (Lemma \ref{lemma:H1 estimate of MF SVGD step 1}), the Aubin-Lions lemma and the Banach-Alaoglu theorem, up to passing to a subsequence, we have strong convergence $\rho_n\ra \rho$ in $L^1((0,T)\times B_R)$, strong convergence $\sqrt{\rho_n}\ra \sqrt{\rho}$ and weak convergence $\nabla \sqrt{\rho_n}\weak \nabla \sqrt{\rho}$ in $L^2((0,T)\times B_R)$ for some $\rho$. In particular, $S$~is a compact operator.

    \smallskip

    \underline{\it $S$ is continuous.} Take a sequence $\{\beta_n\}_{n\geq 1}$ such that $\beta_n\ra \beta$ in $L^1((0,T)\times B_R)$ and define $\rho_n=S(\beta_n)$. The pair $(\beta_n,\rho_n)$ satisfy the following equation for any test function $\psi\in C_c^\infty((0,T)\times B_R)$
    \begin{multline*}
        \int_0^T \int_{B_R}\rho_n \prt_t \psi\diff x\diff t
        =\int_0^T \int_{B_R}\rho_n\; E_{\tau,\eta}(\beta_n)\cdot \nabla\psi+\frac{1}{R}\,(\nabla \rho_n+\rho_n\nabla V_\eta)\cdot \nabla\psi\diff x\diff t\\
        =\int_0^T \int_{B_R}\rho_n\; E_{\tau,\eta}(\beta_n)\cdot \nabla\psi+\frac{2}{R}\nabla\sqrt{\rho_n}\cdot \sqrt{\rho_n}\nabla\psi +\frac{1}{R} \rho_n\nabla V_\eta\cdot\nabla\psi\diff x\diff t.
    \end{multline*}
    From Lemma \ref{lem:moll_operator_are_cts}, we can choose a subsequence with $E_{\tau,\eta}(\beta_n)\weaks E_{\tau,\eta}(\beta)$ weakly* in $L^\infty((0,T)\times B_R)$. Furthermore, using the same convergent sequences as in Step 1, we can pass to the limit $n\ra \infty$ to deduce that the following PDE holds in the sense of distributions
    \begin{align*}
        \prt_t \rho=\divv(\rho\; E_{\tau,\eta}(\beta))+\frac{1}{R}\,\divv\left(\nabla \rho+\rho\nabla V_\eta\right)
         \text{ on }(0,T)\times B_R.
    \end{align*}
    By uniqueness of the solution to the above PDE, we know that $S(\beta)=\rho$. To conclude, we observe that the argument above shows that every subsequence of $\{S(\beta_n)\}_{n\in \N}$ has a further subsequence that converges strongly in $L^1((0,T)\times B_R)$ to $S(\beta)$ so we have $S(\beta_n)\ra S(\beta)$ in $L^1((0,T)\times B_R)$. This proves the continuity of $S$.

    \smallskip

    \underline{\it Eigenvectors are bounded.} We will prove that the set 
    $$
    \{\beta\in L^1((0,T) \times B_R)\,\vert\, \exists \lambda\in [0,1]\,:\, \beta=\lambda S(\beta)\}
    $$ is bounded in $L^1((0,T) \times B_R)$. Suppose there exists $\beta\in L^1((0,T) \times B_R)$ and $\lambda \in [0,1]$ such that for $\rho=S(\beta)$ we have $\beta=\lambda \rho$. Then by conservation of mass
    \begin{align*}
        \int_0^T \int_{B_R} |\beta|\diff x\diff t
        =\lambda \int_0^T \int_{B_R} |\rho|\diff x\diff t = \lambda T \int_{B_R} \rho_0^{R,\eta} \diff x 
        \leq T.
    \end{align*}

    \smallskip

    \underline{\it Leray-Schauder fixed-point theorem.} 
    Using the previous three steps, we see that all the necessary conditions for the Leray-Schauder fixed-point theorem \cite[Theorem 11.3]{MR1814364} are satisfied. Thus, there exists at least one fixed point $\rho\in L^1((0,T)\times B_R)$. Clearly, $\rho$ is also a solution to \eqref{eq:regularised_classical_pde} with $\beta=\eta$ so it is a classical solution.
\end{proof}

Having proven all the necessary preliminary results, we are now in a position to pass to the limit $\tau\ra 0$ and prove the existence of a classical global solution to the PDE \eqref{eq:regularised_weak_pde}--\eqref{ass:mfsvgd:ini}.

\begin{proof}[Proof of Lemma \ref{lemma:existence_of_classical_solutions_to_step_2}]
    Let $\rho^\tau$ be the solution constructed in Lemma \ref{lemma:existence_of_classical_solution_to_reg_pde} with parameter $\tau~>~0$. Using Lemma \ref{lemma:H1 estimate of MF SVGD step 1} with $\beta=\rho^\tau$ (note that we can use better estimates since we know that $\beta \in L^{\infty}(0,T; L^1(B_R))$) we obtain the uniform bounds
    $$
        \sup_{\tau \in (0,1)} \norm{\sqrt{\rho^\tau}}_{L_t^2 H_x^1} \,+\, \sup_{\tau \in (0,1)}\norm{\prt_t \rho^\tau}_{L_t^2 H_x^{-m}}
        <\infty.
    $$
    Hence, arguing as in the first step of Lemma~\ref{lemma:existence_of_classical_solution_to_reg_pde}, we show that $\rho$ is a distributional solution to the PDE \eqref{eq:regularised_weak_pde}.

    \smallskip
    
    Now we prove that $\rho$ is a classical solution. Using \eqref{eq:estimates_Schauder_1_alpha_in_space} in Lemma~\ref{lem:boostrap_aggr_diff_app}, there exists $\alpha \in (0,1)$ such that the sequence $\{\rho^{\tau}\}_{\tau\in (0,1)}$ is bounded in $C^{\frac{\alpha}{2}, 1+\alpha}_{t,x}$. This implies that the sequences $\{\nabla E_{\tau,\eta}(\rho^{\tau})\}_{\tau\in (0,1)}$, $\{\Delta E_{\tau,\eta}(\rho^{\tau})\}_{\tau\in (0,1)}$ are bounded in $C^{\frac{\alpha}{2},\alpha}_{t,x}$, so \eqref{eq:estimates_Schauder_2_alpha_in_space} implies that the sequence $\{\rho^{\tau}\}_{\tau\in (0,1)}$ is bounded in $C^{1+\frac{\alpha}{2}, 2+\alpha}_{t,x}$. Hence, $\rho \in C^{1+\frac{\alpha}{2}, 2+\alpha}_{t,x}$ and the proof is concluded.

    \smallskip

    We now prove that the solution can be extended from $[0,T]\times B_R$ to $[0,\infty)\times B_R$. Indeed, it can be constructed for all $T>0$ and the only issue is to guarantee that solutions constructed for $T_1 < T_2$ coincide on $[0,T_1]$. Suppose there are two solutions $\rho_1$, $\rho_2$ to \eqref{eq:regularised_weak_pde}--\eqref{ass:mfsvgd:ini}. Then, the difference $\rho:= \rho_1-\rho_2$ satisfies 
    $$
    \prt_t \rho=\divv(\rho\;F_\eta(\rho_1)) + \divv(\rho_2\;(F_\eta(\rho_1)-F_\eta(\rho_2))) +\frac{1}{R}\,\divv\left(\nabla \rho+\rho\nabla V_\eta\right)
    $$
    with initial condition $\rho(0,x)=0$. Let $\mbox{sgn}_+(\rho) = \mathds{1}_{\rho \geq 0}$ and let $f_n:\R \to \R$ be a sequence of smooth functions such that $f_n(x) = 0$ for $x \leq 0$, $f_n'(x) \geq 0$ for $x \geq 0$, $f_n \to \mbox{sgn}_+$. We multiply by $f_n(\rho)$, integrate over $B_R$ and pass to the limit $n \to \infty$. Since $\int_{B_R} \Delta \rho \, f_n(\rho) \diff x = -\int_{B_R} |\nabla \rho|^2\, f_n'(\rho) \diff x \leq 0$, we can ignore the term with Laplacian and deduce
    $$
    \prt_t \int_{B_R} |\rho|_+ \diff x \leq \int_{B_R} \left(\divv\left(\rho\;F_\eta(\rho_1) +  \frac{1}{R}\,\rho\,\nabla V_\eta\right) + \divv(\rho_2\;(F_\eta(\rho_1)-F_\eta(\rho_2))) \right) \mbox{sgn}_+(\rho) \diff x.
    $$
    Furthermore, note that
    $$
    \divv\left(\rho\;F_\eta(\rho_1) +  \frac{1}{R}\,\rho\,\nabla V_\eta\right)\, \mbox{sgn}_+(\rho) =  \divv\left(|\rho|_+\;F_\eta(\rho_1) +  \frac{1}{R}\,|\rho|_+\,\nabla V_\eta\right)
    $$
    so that the estimate simplifies to 
    $$
    \prt_t \int_{B_R} |\rho|_+ \diff x \leq \int_{B_R} \divv(\rho_2\;(F_\eta(\rho_1)-F_\eta(\rho_2))) \,  \mbox{sgn}_+(\rho) \diff x.
    $$
    By the regularity, $\rho_2$ and $\nabla \rho_2$ are bounded. Moreover, directly from the formula \eqref{eq:formula_for_Feta}, there exists a constant $C$ such that
    $$
    \|F_\eta(\rho_1)-F_\eta(\rho_2)\|_{L^1(B_R)}, \|\nabla F_\eta(\rho_1)-\nabla F_\eta(\rho_2)\|_{L^1(B_R)} \leq C\, \|\rho_1 - \rho_2\|_{L^1(B_R)} = C\, \|\rho\|_{L^1(B_R)}.
    $$
    Hence, we get $\prt_t \int_{B_R} |\rho|_+ \diff x \leq  C\, \|\rho\|_{L^1(B_R)}$. Swapping $\rho_1$ and $\rho_2$, we obtain $\prt_t \|\rho\|_{L^1(B_R)} \leq  C\, \|\rho\|_{L^1(B_R)}$ so by the Gronwall inequality $\rho = 0$ and the proof is concluded.
\end{proof}

In the estimates above we used the following result.

\begin{lem}[bootstrapping]\label{lem:boostrap_aggr_diff_app}
Let $\rho$ be a classical nonnegative solution to
\begin{equation}\label{eq:general-aggr-diff-for-bootstrap-only-space}
\partial_t \rho - \Delta \rho - \DIV(\rho \, \mathcal{Q}) = 0 \mbox{ on } (0,T)\times B_R
\end{equation}
with initial condition $\rho(0,x) = \rho_0(x) \in C^{\infty}_c(B_R)$ and boundary condition $\rho(t,x) = 0$ for $x \in \partial B_R$. Suppose that $\mathcal{Q} \in L^{\infty}(0,T; W^{1,\infty}(B_R; \Rd))$. Then, there exists $\alpha \in (0,1)$ and constant $C$ depending only on $\rho_0$ and $\|\mathcal{Q}\|_{L^{\infty}_t W^{1,\infty}_x}$ such that
\begin{equation}\label{eq:estimates_Schauder_1_alpha_in_space}
\|\rho\|_{C^{\frac{\alpha}{2}, 1+\alpha}_{t,x}} \leq C.
\end{equation}
Furthermore, if $\mathcal{Q}, \divv\mathcal{Q} \in C^{\frac{\alpha}{2},\alpha}_{t,x}$, we have
\begin{equation}\label{eq:estimates_Schauder_2_alpha_in_space}
\|\rho\|_{C^{1+\frac{\alpha}{2}, 2+\alpha}_{t,x}} \leq C,
\end{equation}
where $C$ depends additionally on the norm of $\mathcal{Q}$, $\divv \mathcal{Q}$ in $C^{\frac{\alpha}{2},\alpha}_{t,x}$.
\end{lem}
\begin{proof} We split the reasoning into a few steps. The constant $C$ may change from line to line as long as it satisfies the condition of the lemma. We remark that the reasoning is in the spirit of \cite[Theorem 2.3]{carrillo2024well} but here $\mathcal{Q}$ depends on time so the argument needs small adjustments. 

\underline{\textit{Step 1. An $L^{\infty}_{t,x}$ estimate.}} The following inequality holds
\begin{align*}    \norm{\rho}_{L_{t,x}^\infty}\leq \norm{\rho_0}_{L_x^\infty}e^{T \norm{\divv \mathcal{Q}}_{L_{t,x}^\infty}}.
\end{align*}
Consider function $u(t,x):=\rho_t(x)e^{-\lambda t}$ and note that it solves
\begin{align*}
    \partial_t u - \Delta u - \nabla u\cdot \mathcal{Q}+(\lambda-\divv \mathcal{Q}) u = 0 \mbox{ on } (0,T)\times B_R
\end{align*}
with initial condition $u(0,x) = \rho_0(x) \in C^{\infty}_c(B_R)$ and boundary condition $u(t,x) = 0$ for $x \in \partial B_R$. Let $\lambda:=\norm{\divv \mathcal{Q}}_{L_{t,x}^\infty}$ so that by nonnegativity of $u$
$$
\partial_t u - \Delta u - \nabla u\cdot \mathcal{Q} \geq 0 \mbox{ on } (0,T)\times B_R.
$$
The weak maximum principle \cite[Theorem 8, Section 7.1.4]{MR2597943} asserts that $u$ is controlled by the values at the boundary ($u = 0$) or at $t = 0$ which implies the inequality for $\rho$.

\underline{\textit{Step 2. Gradient estimates in $L^{\infty}_t L^p_x$ for $p<\infty$.}} We compute
$$
\partial_t\, \frac{1}{p} \int_{B_R} |\nabla \rho|^p \diff x = \int_{B_R} |\nabla \rho|^{p-2} \, \nabla \rho \cdot \nabla \prt_t\rho_t \diff x = - \int_{B_R} \DIV(|\nabla \rho|^{p-2} \, \nabla \rho) \, \prt_t\rho_t \diff x,
$$
where the integration by parts is justified by the fact that for smooth solutions $\partial_t \rho(t,x) = 0$ for $x \in \partial B_R$. A direct computation yields $\DIV(|\nabla \rho|^{p-2} \, \nabla \rho) = (p-1) |\nabla \rho|^{p-2} \,\Delta \rho$ so that using PDE \eqref{eq:general-aggr-diff-for-bootstrap-only-space} we get 
\begin{align*}
\partial_t\, &\frac{1}{p\,(p-1)} \int_{B_R} |\nabla \rho|^p \diff x + \int_{B_R} |\nabla \rho|^{p-2}\, |\Delta \rho|^2 \diff x \leq \\ & \leq \int_{B_R}|\nabla \rho|^{p-2} |\Delta \rho| \left( |\nabla \rho|\, |\mathcal{Q}| + |\rho| \,|\divv \mathcal{Q}|\right) \diff x \leq C\, \int_{B_R}  |\nabla \rho|^{p-2} |\Delta \rho| ( |\nabla \rho|\,  +1) \diff x.
\end{align*}
The error on the RHS can be controlled with the Cauchy-Schwartz inequality with a small parameter $\delta>0$ as follows:
\begin{align*}
\int_{B_R}  |\nabla \rho|^{p-2} |\Delta \rho| \left( |\nabla \rho|\,  +1\right) \diff x &\leq \frac{\delta}{2} \int_{B_R} |\nabla \rho|^{p-2} |\Delta \rho|^2 \diff x + \frac{1}{2\delta} \int_{B_R} |\nabla \rho|^{p-2} \, ( |\nabla \rho|\,  +1)^2 \diff x  \\
& \leq \frac{\delta}{2} \int_{B_R} |\nabla \rho|^{p-2} |\Delta \rho|^2 \diff x + \frac{2^p}{\delta} \int_{B_R} |\nabla \rho|^{p} \diff x + \frac{2^p\,|B_R|}{\delta},
\end{align*}
where in the last step we estimated $|\nabla \rho|\leq \max(|\nabla \rho|,1)$ and we used inequality $|x+y|^{p} \leq 2^{p-1}(|x|^p + |y|^p)$. Choosing $\delta = 2^p\, p(p-1)$, we conclude the proof by the Gronwall inequality.\\

\underline{\textit{Step 3. Estimates on $\partial_t \rho$, $\partial_{x_i}\partial_{x_j} \rho$ in $L^p_{t,x}$ for any $p<\infty$ and $\rho$ in $C^{\frac{\alpha}{2}, 1+\alpha}_{t,x}$ \eqref{eq:estimates_Schauder_1_alpha_in_space}.}} We have
\begin{equation}\label{eq:general_aggr_diff_for_bootrstrapping_for_max_regularity}
\partial_t \rho - \Delta \rho = \nabla \rho \, \mathcal{Q} + \rho \, \divv \mathcal{Q}.
\end{equation}
Since the RHS of \eqref{eq:general_aggr_diff_for_bootrstrapping_for_max_regularity} is bounded in $L^{\infty}_t L^p_x$, in particular in $L^{p}_{t,x}$. By the maximal regularity theory \cite[Theorem 3.1]{HieberPruss1997} we obtain $\partial_t \rho$, $\partial_{x_i}\partial_{x_j} \rho$ in $L^p_{t,x}$ for any $p<\infty$. Then, the Sobolev embedding theorem \cite[Theorem 1.4.1]{MR2309679} for $p$ sufficiently large implies the desired Holder regularity. \\

\underline{\textit{Step 4. Estimates on $\rho$ in $C^{1+\frac{\alpha}{2}, 2+\alpha}_{t,x}$ \eqref{eq:estimates_Schauder_2_alpha_in_space}.}} Under the assumption that $\mathcal{Q}, \divv \mathcal{Q} \in C^{\frac{\alpha}{2},\alpha}_{t,x}$, we deduce that the RHS of \eqref{eq:general_aggr_diff_for_bootrstrapping_for_max_regularity} is in $C^{\frac{\alpha}{2},\alpha}_{t,x}$. Hence, Schauder's theory for parabolic equations \cite[Theorem 5.1.8]{MR1329547} gives the claim. 

\end{proof}

\section{Other auxiliary results}
\label{appendix:auxiliary}
\begin{lem}
\label{lem:chi_squared_bounded_from_below_by_kullback_leibler}
	Let $\mu,\nu\in\Pro$, such that $\mu$ is absolutely continuous wrt $\nu$. Let $t\in\R$ and let $h=\frac{\diff\mu}{\diff\nu} \geq 0$. Then 
	\begin{align*}
		\int_{\Rd} |h(x)-t|^2\diff\nu(x)
		\geq \int_{\Rd} h(x)\ln(h(x)) \diff\nu(x).
	\end{align*}
\end{lem}
\begin{proof}
	Using the concavity of the logarithm, we see that $\ln(x)\leq x-1$ for all $x>0$. Then
	\begin{multline*}
		\int_{\Rd} h(x)\ln(h(x)) \diff\nu(x)
		\leq \int_{\Rd} h(x)(h(x)-1) \diff\nu(x)
		=\int_{\Rd} h(x)^2\diff\nu(x)-1\\
		\leq \int_{\Rd} h(x)^2\diff\nu(x)-1+(t-1)^2
		=\int_{\Rd} |h(x)-t|^2\diff\nu(x),
	\end{multline*}
	where we have used the fact that $\int_{\Rd} h(x)\diff\nu(x)=1$, $\int_{\Rd} 1\diff\nu(x)
		=1$.  
\end{proof}

\begin{lem}\label{lem:control_neg_log}
 Let $V \geq 0$ be such that $e^{-V} \in L^1(\Rd)$. Then, for every measurable function $\rho: \R^d \to \R^+$ we have 
\begin{equation}\label{eq:pointwise_lower_bound}
\rho\, \log \rho + \rho\,V  \geq - \frac{2}{e} \, e^{-V/2}. 
\end{equation}
Moreover,
\begin{equation}\label{eq:lower_bound_negative_log}
\frac{1}{2}\int_{\R^d} \rho\, |\log \rho| \diff x \leq \frac{1}{2}\int_{\R^d} \rho\, \log \rho \diff x +   \int_{\R^d} \rho\, V(x) \diff x + \frac{2}{e} \int_{\R^d} e^{-V(x)/2} \diff x.
\end{equation}
\end{lem}
\begin{proof}
We closely follow the proof from \cite[Proposition 4, Step 3]{doumic_multispecies_2024}. To bound $\rho \, \log \rho$ from below, we split the set $\{\rho \leq 1\}$ for $A:=\{\rho<e^{-V}\}$ and $B:=\{e^{-V} \leq \rho < 1\}$ and we want to bound $\rho$ on each of them. On $A$, we use that $\sqrt{\rho} \, \log \rho \geq -\frac{2}{e}$ for $\rho \in [0,1]$ so that
\begin{equation}\label{eq:estim_on_set_A}
\rho \, \log \rho =\sqrt{\rho}\, \sqrt{\rho}\, \log \rho \geq -\frac{2}{e} \, e^{- V/2} \mbox{ for } x\in A.
\end{equation}
On $B$ we estimate $\log \rho \geq - V$ so that
\begin{equation}\label{eq:estim_on_set_B}
\rho \, \log \rho \geq - \rho\, V  \mbox{ for } x\in B.
\end{equation}
Hence, we deduce \eqref{eq:pointwise_lower_bound} from \eqref{eq:estim_on_set_A} and \eqref{eq:estim_on_set_B}.\\

To see the other assertions, we write
$$
\int_{\R^d} \rho\, |\log \rho| \diff x = \int_{\rho \geq 1} \rho\, \log \rho \diff x -  \int_{\rho < 1} \rho\, \log \rho \diff x = 
\int_{\R^d} \rho\, \log \rho \diff x - 2 \int_{\rho < 1} \rho\, \log \rho \diff x.
$$
Now, using \eqref{eq:pointwise_lower_bound} proves \eqref{lem:control_neg_log}, with which we conclude the proof.
\end{proof}

\begin{lem}\label{lem:poincare-wirtinger-optimal}
For all $f \in H^1(B_R)$ such that $\int_{B_R} f \diff x = 0$ we have
$$
\int_{B_R} |f|^2 \diff x \leq C\, \int_{B_R} |\nabla f|^2 \diff x,
$$
where $C = \left(\frac{R}{p_{d/2,1}}\right)^2$ and the numerical coefficient $p_{m,1}$ is the first positive zero of the first derivative of the Bessel function $J_m$.
\end{lem}
\begin{proof}
We refer to \cite{MR3343056} for the characterization of the optimal constant of the Poincaré-Wirtinger inequality as the minimal eigenvalue of the Laplacian with Neumann boundary conditions and \cite{MR0079286} for its explicit value for the ball. 
\end{proof}

\begin{lem}[Kullback-Leibler of Gaussians]
\label{lem:kl_gaussians}
    Let $\rho_0=\mathcal{N}(\mu_0,\Sigma_0)$ and $\rho_\infty=\mathcal{N}(\mu_\infty,\Sigma_\infty)$ be two Gaussian distributions. Then the Kullback-Leibler divergence has the explicit value
    \begin{align*}
        \KL(\rho_0||\rho_\infty)
        =\frac{1}{2}\left(\trace(\Sigma_\infty^{-1}\Sigma_0)+(\mu_0-\mu_\infty)\Sigma_\infty^{-1}(\mu_0-\mu_\infty)-d-\ln\left(\frac{\det(\Sigma_0)}{\det(\Sigma_\infty)}\right)\right).
    \end{align*}
\end{lem}
\begin{proof}
    First we compute
    \begin{align*}
        \KL(\rho_0||\rho_\infty)
        =-\frac{\ln(\det(\Sigma_0))}{\ln(\det(\Sigma_\infty))}
        &-\frac{1}{2}\int_{\Rd} (x-\mu_0)\cdot\Sigma_0^{-1}(x-\mu_0)\diff\rho_0(x)\\
        &\qquad +\frac{1}{2}\int_{\Rd} (x-\mu_\infty)\cdot\Sigma_\infty^{-1}(x-\mu_\infty)\diff\rho_0(x).
    \end{align*}
    We use the well-know fact that $\trace(AB)=\trace(BA)$ for any two matrices $A\in \R^{n\times m}$ and $B\in \R^{m\times n}$ to compute the third term
    \begin{align*}
        &\int_{\Rd} (x-\mu_\infty)\cdot\Sigma_\infty^{-1}(x-\mu_\infty)\diff\rho_0(x)\\
        &\qquad=\int_{\Rd} \trace((x-\mu_0+\mu_0-\mu_\infty)\otimes (x-\mu_0+\mu_0-\mu_\infty)\cdot\Sigma_\infty^{-1})\diff\rho_0(x)\\
        &\qquad=\trace(\Sigma_0\Sigma_\infty^{-1})+(\mu_0-\mu_\infty)\cdot\Sigma_\infty^{-1}(\mu_0-\mu_\infty).
    \end{align*}
    Exactly analogously, we compute the second term as
    \begin{align*}
        \int_{\Rd} (x-\mu_0)\cdot\Sigma_0^{-1}(x-\mu_0)\diff\rho_0(x)
        =\trace(\Sigma_0\Sigma_0^{-1})
        =d.
    \end{align*}
    Combining everything yields the desired result.
\end{proof}

\begin{lem}\label{lem:chng_var}
    Let $R>0$, $\kappa>0$ and $\rho\in C^2(\overline{B_R})$ be a function with $\rho=0$ on $\prt B_R$. Then
    \begin{align*}
        \int_{B_R} \Delta\rho\cdot \ln(\rho+\kappa) \diff x
        =-\int_{B_R}\frac{|\nabla\rho|^2}{\rho+\kappa}\diff x.
    \end{align*}
\end{lem}
\begin{proof}
    Let $\vp_n$ be a smooth nondecreasing function such that $\vp_n(x) = 1$ for $|x|< R - \frac{2}{n}$, $\vp_n(x)  = 0$ for $|x| > R -\frac{1}{n}$ and $|\nabla \vp_n| \leq n$. In particular,
    $$
        \norm{\nabla \vp_n}_{L^1}
        =n\, |B_1| \, \left(R^d - \left(R-\frac{2}{n}\right)^d\right) \leq n\,|B_1| \, d\,  R^{d-1} \frac{2}{n} = 2\, d\, |B_1| \, R^{d-1}.
    $$
    Applying integration by parts we get
    \begin{align*}
        \int_{B_R} \Delta\rho\cdot \ln\left(\rho+\kappa\right)\, \vp_n\diff x
        =-\int_{B_R} \frac{|\nabla \rho|^2}{\rho+\kappa} \vp_n\diff x-2\int_{ B_R}\nabla\sqrt{\rho}\cdot \sqrt{\rho}\ln\left(\rho+\kappa\right)\nabla \vp_n\diff x.
    \end{align*}
    From \cite[Theorem 1]{lions1995regularite} we know that $\norm{\nabla\sqrt{\rho}}_{L^\infty}<\infty$. Moreover, for all $\varepsilon>0$ there exists $n\geq 1$ such that we have $|\sqrt{\rho}\ln\left(\rho+\kappa\right)|\leq \varepsilon$ on $B_{R-2/n}$. Thus, we use the Hölder inequality to estimate
    \begin{align*}
        \left\vert \int_{B_R}\nabla\sqrt{\rho}\cdot \sqrt{\rho}\ln\left(\rho+\kappa\right)\nabla \varphi_n\diff x\right\vert
        \leq \varepsilon \norm{\nabla\sqrt{\rho}}_{L^\infty}\,  \norm{\nabla \vp_n}_{L^1} \leq C\varepsilon \norm{\nabla\sqrt{\rho}}_{L^\infty} 
    \end{align*}
    for some constant $C>0$ independent of $n$ and $\varepsilon$. As $\varepsilon$ can be arbitrarily small, we have proven the change of variables.
\end{proof}

\subsection*{Acknowledgements}
JAC and JS were supported by the Advanced Grant Nonlocal-CPD (Nonlocal PDEs for Complex Particle Dynamics: Phase Transitions, Patterns and Synchronization) of the European Research Council Executive Agency (ERC) under the European Union’s Horizon 2020 research and innovation programme (grant agreement No. 883363). JAC was also partially supported by the EPSRC grant numbers EP/T022132/1 and EP/V051121/1. JW was supported by the Engineering and Physical Sciences Research Council (grant number EP/W524311/1)

\bibliographystyle{abbrv}
\bibliography{fastlimit}
\end{document}